%% file: main.tex
\documentclass[11pt,twoside,a4paper]{article}

\input{commands}

\def\unif{\mathsf{unif}}
\def\kl{\mathrm{KL}}
\def\iid{\mathrm{i.i.d.}}

\def\supp{\mathrm{supp}}
\def\LSI{\mathrm{LSI}}
\def\diag{\mathrm{diag}}

\DeclareMathOperator\Div{div}
\DeclareMathOperator{\Cov}{Cov}
\DeclareMathOperator{\trace}{tr}
\usepackage[table]{xcolor}
\usepackage{pifont}
\mathtoolsset{showonlyrefs}

\usepackage{tikz} 
\usetikzlibrary{arrows.meta, positioning, calc}

\usepackage{amsmath, amssymb}
\usepackage{algpseudocode}
\usepackage{pgfplots} 
\pgfplotsset{compat=1.18}
\usepackage{mathtools}
\setheader{Sampling via Stochastic Interpolants}{Duan, Jiao, Steidl, Wald, Yang, and Zhang}

\title{Sampling via Stochastic Interpolants by Langevin-based Velocity  and Initialization Estimation in Flow ODEs}

\author[1]{Chenguang Duan}
\author[2]{Yuling Jiao}
\author[3]{Gabriele Steidl}
\author[4]{Christian Wald}
\author[5]{Jerry Zhijian Yang}
\author[5]{Ruizhe Zhang}

\affil[1]{Institut f{\"u}r Geometrie und Praktische Mathematik, RWTH Aachen University,  \protect\\ Im S{\"u}sterfeld 2, 52072 Aachen, Germany. \protect\\ \texttt{duan@igpm.rwth-aachen.de}}

\affil[2]{School of Artificial Intelligence, Wuhan University, Wuhan, Hubei 430072, China. \protect\\ \texttt{yulingjiaomath@whu.edu.cn}}

\affil[3]{Institut f{\"u}r Mathematik, Technische Universit{\"a}t Berlin, \protect\\ Stra{\ss}e des 17. Juni 136, 10623 Berlin, Germany. \protect\\ \texttt{steidl@math.tu-berlin.de}}

\affil[4]{INSA Lyon, Centrale Lyon, Universit{\'e} Lyon 1, Universit{\'e} Jean Monnet, CNRS, ICJ UMR5208, 69621 Villeurbanne, France. \protect\\ \texttt{christian.wald@insa-lyon.fr}}

\affil[5]{School of Mathematics and Statistics, Wuhan University, Wuhan, Hubei 430072, China. \protect\\ \texttt{\{zjyang.math,ruizhezhang\}@whu.edu.cn}}
\date{\today}

\begin{document}

\thispagestyle{plain}
\maketitle

\begin{abstract}
We propose a novel method for sampling from unnormalized Boltzmann densities based on a probability flow ordinary differential equation (ODE) derived from linear stochastic interpolants. The key innovation of our approach is the use of a sequence of Langevin samplers to enable efficient simulation of the flow. Specifically, these Langevin samplers are employed (i) to generate samples from the interpolant distribution at intermediate times and (ii) to construct, starting from these intermediate times, a robust estimator of the velocity field governing the probability flow ODE. Theoretically, we provide convergence guarantees for both Langevin components, and establish a non-asymptotic convergence rate for the probability flow ODE. Extensive numerical experiments demonstrate the efficiency of the proposed method on challenging multimodal distributions across a range of dimensions, as well as its effectiveness in Bayesian inference tasks.
\end{abstract}

\input{introduction}

\input{method}

\input{convergence}

\input{precondition}

\input{experiments}

\input{conclusion}

\section*{Acknowledgement}

GS acknowledges funding by Deutsche Forschungsgemeinschaft (DFG, German Research Foundation), Project STE 571/23-1 and ChW  by the DFG within the SFB “Tomography Across the Scales” (STE 571/19-1, project number: 495365311). CD was supported in part by Deutsche Forschungsgemeinschaft (DFG, German Research Foundation) -- project number  442047500/SFB 1481 \emph{Sparsity and Singular Structures}. YJ was supported by the National Key Research and Development Program of China (2024YFA1014202) and the National Natural Science Foundation of China (No. 12371441, No.  12526216). JZY was supported by the National Natural Science Foundation of China (No. 12125103, No. U24A2002) and the Hubei Natural Science Foundation (No. 2025AFA002).

\bibliographystyle{abbrv}
\bibliography{reference}

\appendix
\input{appendix}

\end{document}

%% file: commands.tex
\usepackage[utf8]{inputenc}     
\usepackage[T1]{fontenc}        
\usepackage[english]{babel}     
\usepackage{microtype}          

\usepackage{lmodern}            

\usepackage[leqno]{mathtools}   
\usepackage{amssymb}            
\usepackage{amsthm}             
\usepackage{nicefrac}           

\usepackage{bm}                 
\usepackage{dsfont}             
\usepackage{bbm}                
\usepackage{mathrsfs}           
\usepackage{euscript}           
\usepackage[bb=px]{mathalfa}    

\usepackage[dvipsnames]{xcolor} 
\usepackage{geometry}
\usepackage{graphicx}
\usepackage{booktabs}           
\usepackage{enumitem}
\setlist[enumerate]{itemsep=0pt}
\usepackage{fancyhdr}
\usepackage{titlesec}
\usepackage{authblk}            
\usepackage{float}

\usepackage[numbers,sort]{natbib}

\usepackage[ruled,linesnumbered]{algorithm2e}

\usepackage{tikz}
\usetikzlibrary{positioning, fit, tikzmark, calc}
\usetikzlibrary{decorations.pathreplacing, calligraphy}

\colorlet{inlinkcolor}{purple}
\colorlet{exlinkcolor}{purple}
\colorlet{citecolor}{teal!50!blue}

\usepackage[colorlinks=true,linkcolor=inlinkcolor,citecolor=citecolor,urlcolor=exlinkcolor,linktoc=page,breaklinks=true,plainpages=false,pagebackref=true,backref=page]{hyperref}
\renewcommand*{\backrefalt}[4]{%
    \ifcase #1 \footnotesize{(Not cited.)}%
    \or        \footnotesize{(Cited on page~#2.)}%
    \else      \footnotesize{(Cited on pages~#2.)}%
    \fi}

\newtheorem{theorem}{Theorem}[section]
\newtheorem{proposition}[theorem]{Proposition}
\newtheorem{lemma}[theorem]{Lemma}

\newtheorem{innercustomproposition}{Proposition}
\newenvironment{customproposition}[1]
{\renewcommand\theinnercustomproposition{#1}\innercustomproposition}
{\endinnercustomproposition}
\newtheorem{innercustomlemma}{Lemma}
\newenvironment{customlemma}[1]
{\renewcommand\theinnercustomlemma{#1}\innercustomlemma}
{\endinnercustomlemma}
\newtheorem{innercustomtheorem}{Theorem}
\newenvironment{customtheorem}[1]
{\renewcommand\theinnercustomtheorem{#1}\innercustomtheorem}
{\endinnercustomtheorem}

\theoremstyle{definition}

\newtheorem{assumption}{Assumption}
\theoremstyle{remark}
\newtheorem{remark}[theorem]{Remark}
\newtheorem{example}[theorem]{Example}

\numberwithin{equation}{section} 

\newcommand{\setheader}[2]{%
    \pagestyle{fancy}
    \fancyhf{}
    \fancyhead[LO]{#1}
    \fancyhead[RO]{}
    \fancyhead[LE]{}
    \fancyhead[RE]{#2}
    \renewcommand{\headrulewidth}{0.4pt}
    \fancyfoot[C]{\thepage}
    \renewcommand{\headrule}{\color{gray!75}\hrule width\textwidth height 1.0pt}
}

\makeatletter
\renewcommand{\maketitle}{
    \begin{center}
        {\Large\bfseries\@title\par}
        \vskip 1.5em
        {\@author\par}
        \vskip 1em
        {\@date\par}
    \end{center}
    \vskip 1.5em
}
\makeatother

\usepackage{changepage}
\makeatletter
\renewenvironment{abstract}{
    \small
    \begin{adjustwidth}{0.3in}{0.3in} 
    \noindent{\bfseries Abstract.}
    \ignorespaces
}{
    \end{adjustwidth}
    \par\vspace{0.5em}
}
\makeatother

\newcommand{\keywords}[1]{
    \begin{adjustwidth}{0.3in}{0.3in} 
    \noindent{\small\bfseries Keywords:} {\small #1}
    \end{adjustwidth}
    \par\noindent\textcolor{gray}{\rule{\linewidth}{0.5pt}}
}

\titleformat{\section}
    {\normalfont\large\bfseries\centering}
    {\thesection}{1em}{}
\titleformat{\subsection}[runin]
    {\normalfont\normalsize\bfseries}
    {\thesubsection}{1em}{}[.]
\titleformat{\subsubsection}[runin]
    {\normalfont\normalsize\itshape}
    {\thesubsubsection}{1em}{}[.]
\titleformat{\paragraph}[runin]
    {\normalfont\normalsize\bfseries}
    {\theparagraph}{1em}{}[.]

\titlespacing{\section}
    {0pt}{2.0ex plus 0.5ex minus 0.2ex}{1.0ex plus 0.1ex}
\titlespacing{\subsection}
    {0pt}{1.0ex plus 0.1ex minus 0.1ex}{1.0ex plus 0.1ex}
\titlespacing{\subsubsection}
    {0pt}{1.0ex plus 0.1ex minus 0.1ex}{1.0ex plus 0.1ex}
\titlespacing{\paragraph}
    {0pt}{0.5ex plus 0.1ex minus 0.1ex}{0.5ex plus 0.1ex}

\def\d{\,\mathrm{d}}
\def\dt{\,\mathrm{d}t}
\def\ds{\,\mathrm{d}s}

\def\what{\widehat}

\usepackage{bm,bbm}

\def\calN{{\mathcal{N}}}
\def\calO{{\mathcal{O}}}

\def\bbE{{\mathbb{E}}}

\def\bbP{{\mathbb{P}}}

\def\bbR{{\mathbb{R}}}

\usepackage{mathrsfs}

%% file: introduction.tex

\section{Introduction}

\par The fundamental goal of sampling involves generating particles from an unnormalized Boltzmann density, serving as a cornerstone in statistical physics~\cite{Landau2014Guide,Frank2019Boltzmann,midgley2023flow, he2025feat}, machine learning~\cite{Song2019Generative,Cobb2021Scaling,Izmailov2021what}, and Bayesian inference~\cite{Durmus2019High,Overstall2020Bayesian,Heng2021Gibbs,Stanton2022Accelerating,chung2023diffusion,ding2025nonlinear,martin2025pnpflow,chang2025provable,Durmus2025Generative,cardoso2024monte}.

\par However, sampling becomes difficult when the target distribution is multi-modal. In such scenarios, classical Markov Chain Monte Carlo (MCMC) methods~\cite{GRS1995}, such as Langevin Monte Carlo, Hamiltonian Monte Carlo (HMC)~\cite{HG2014}, often fail to explore the global structure of the probability space. They tend to become trapped in local modes, particularly when these modes are separated by high-energy barriers or extensive low-density regions where transitions occur rarely.

\par A common strategy for sampling from complex multi-modal target distributions is to construct a curve in a measure space that gradually pushes an easy-to-sample initialization distribution to the target distribution. This approach traces its roots to sequential Monte Carlo (SMC) methods~\cite{neal2001annealed,DDJ2006}, which typically define a sequence of intermediate distributions via a linear interpolation between the target and initialization energy functions. Mathematically, this interpolant scheme corresponds to the normalized product of the target and initialization densities, which is also utilized 
by~\cite{wu2020Stochastic,HHS2022,Heng2021Gibbs,  wu2025annealing,guo2025provable,ding2025nonlinear,Arbel2021Annealed}. However, multiplying the target by a Gaussian merely rescales the mode heights without necessarily merging them. Consequently, deep low-density regions separating modes are preserved, potentially causing samplers to become trapped~\cite{Chen2024Diffusive,he2025training}. Furthermore, because the modes effectively remain distinct until the very end of the annealing process, the necessary mass transport between distant modes is forced to occur at a late stage. This phenomenon is commonly referred to as the ``teleportation issue''~\cite{mate2023learning,chemseddine2025neural}.

\par Inspired by diffusion models~\cite{Sohl2015Deep,Ho2020Denoising,Song2019Generative,song2021scorebased} and flow-based models~\cite{lipman2023flow,albergo2023building, albergo2025stochastic}, constructing interpolants via Gaussian convolution has garnered significant attention. By convolving the target density with a Gaussian kernel, the energy landscape is effectively smoothed. This process merges isolated modes and populates low-density regions, yielding intermediate distributions that are significantly easier to sample from~\cite{Mobahi2015Link,Song2019Generative,mate2023learning,Chen2024Diffusive,he2025training}.

\par  Within the framework of flow-based or diffusion-based sampling methods, the central technical challenge lies in accurately estimating the velocity field that drives the probability flow ODE, or equivalently, the drift term of the associated SDE. In the sampling setting only an unnormalized target density is accessible, making the velocity field or score of the intermediate distributions hard to estimate. Leveraging Tweedie's formula~\cite{Tweedie1956}, see also~\cite{Efron2011Tweedie,Laumont2022an}, the velocity field or score function can be expressed as an conditional expectation, enabling on-the-fly Monte Carlo estimation without pre-training a neural network. Existing estimators include importance sampling~\cite{huang2024reverse,Huang2025Schrodinger}, 
see~\cite[Section 4.2]{vacher2025sampling} for a review, and rejection sampling~\cite{he2024zeroth}. However, these approaches typically suffer in high-dimensional settings due to the curse of dimensionality. 
The recent work \cite{young2026diffusion}
tackles the score estimation problem by developing an efficient sequential Monte Carlo sampler that evolves auxiliary variables from conditional distributions along the path, providing principled score and density estimates for time-varying distributions inclusive theoretical guarantees. An alternative line of work uses Langevin-based methods to estimate the score function~\cite{huang2024reverse,Huang2024faster,Grenioux2024Stochastic}, thereby improving scalability in high-dimensional settings.  Our approach builds on those line of work, particularly by using Langevin techniques to construct estimators of the velocity field. However, rather than relying on diffusion-based methods, which estimate the drift of an SDE, we focus on the probability-flow ODE and provide a detailed error analysis.

\par In this work, we propose a novel framework for sampling from unnormalized Boltzmann densities, driven by linear stochastic interpolants~\cite{albergo2023building, albergo2025stochastic}; see also~\cite{WS2025} for an overview. Linear stochastic interpolants induce a probability flow ODE that transports an easy-to-sample initialization distribution to the complex target distribution via Gaussian convolution. Our approach decomposes the intractable problem of sampling from a multi-modal target into a sequence of tractable subproblems, each solvable via Langevin Monte Carlo. Specifically, we employ Langevin Monte Carlo to (i) generate samples from the distribution of an intermediate time of the interpolant, where the distribution is still easy to sample from via Langevin Monte Carlo. Starting from this intermediate time we (ii) use the ODE associated to the interpolant and estimate the vector field using Langevin Monte Carlo as well. For the whole procedure see Fig. \ref{fig:strategy}. In order to improve both (i) and (ii), we apply preconditioning to these Langevin diffusions. This adaptively scales step sizes based on the local geometry of the distribution, thereby enhancing convergence.

\subsection{Contributions} 

\par Our main contributions are summarized as follows:  
\begin{enumerate}[label=(\roman*)]
\item We propose a novel framework for sampling from unnormalized Boltzmann densities based on linear stochastic interpolants. Unlike previous works, such as \cite{huang2024reverse,Huang2024faster,Grenioux2024Stochastic}, our approach is based on stochastic interpolants rather than the diffusion framework. We derive a probability-flow ODE that transports an initial distribution to the target distribution. Both the generation of the initial particles and the estimation of the velocity field are reduced to substantially simpler sampling problems, which can be solved using Langevin Monte Carlo.
\item We provide a rigorous convergence analysis for the Langevin-based velocity estimation and ODE initialization. Furthermore, we establish non-asymptotic convergence rates for the probability flow ODE, detailing how the sampling error is governed by discretization, velocity estimation, and initialization errors.
\item We introduce an RMSprop-based preconditioning strategy for Langevin Monte Carlo, which enables adaptive step sizes. This preconditioning improves the ability of the sampler to escape saddle points in complex energy landscapes, thereby enhancing exploration and facilitating transitions across energy barriers. Numerical results demonstrate clear advantages over vanilla Langevin dynamics.
\item We perform extensive numerical experiments to validate the efficiency of the proposed method on multi-modal target distributions in various dimensions. We further apply our method to Bayesian inference, demonstrating its capability for uncertainty quantification on toy problems. In addition, ablation studies investigate the impact of the initialization time, the necessity of Langevin-based initialization, and the performance gains achieved through preconditioning.
\end{enumerate}

\subsection{Organization}

\par Section~\ref{section:prelim} provides a brief introduction to Langevin diffusion and establishes the basic notation used throughout the work. In Section~\ref{section:method}, we present the proposed framework for sampling via stochastic interpolants, followed by the Langevin-based velocity field estimation in Section~\ref{section:method:velocity}. Section~\ref{section:computation:flow} details the computational aspects of the probability flow ODE, including initialization procedures and simulation techniques. In Section~\ref{section:convergence}, we provide rigorous error analysis for Langevin-based velocity estimation, ODE initialization, and the sampling error of the probability flow ODE. To enhance sampling efficiency, we introduce a preconditioning strategy for Langevin samplers in Section~\ref{section:preconditioning}. In Section~\ref{section:experiments}, we demonstrate the efficacy of our approach through extensive numerical experiments. Finally, we conclude the paper in Section~\ref{section:conclusion}. Detailed proofs and additional experimental settings are provided in the appendices.

\section{Preliminaries} \label{section:prelim}
Let $(\Omega,\Sigma, \mathbb P)$ be a probability space.
In the following, we denote random variables $X: \Omega \to \bbR^d$ by capital letters and write $X \sim \mu_X$ if $\mu_X = X_\sharp \mathbb P = \mathbb P \circ X^{-1}$. If $\mu$ is absolutely continuous with density $p$, we write just $p_X$.

Our approach is heavily based on the Langevin diffusion with a stationary density $p \in C^1(\mathbb R^d)$ given by
\begin{equation}\label{eq:langevin_basic}
\d U_{s} = \nabla \log p (U_s)\d s + \sqrt{2}\d B_{s}, \quad s\geq 0
\end{equation}
and starting in the random variable $U_0 \in \mathbb R^d$. Here $B_s$ denotes a $d$-dimensional Brownian motion.
Under some assumptions on the stationary density $p$, we have that the marginals $U_s$ converge for $s \to \infty$
in distribution to a random variable with law $p$. 
To this end, recall that a density $p$
satisfies a \emph{log-Sobolev inequality} with constant $C_{\LSI}(p)$, 
if
\begin{equation}\label{eq:lsi}
\mathrm{Ent}_{p}(f^{2})\leq 2 C_{\LSI}(p) \bbE_{p} |\nabla f|^{2}, \quad \text{for all } f\in C_b^{\infty}(\bbR^{d}),
\end{equation}
where $\mathrm{Ent}_p$ represents the entropy defined as
\begin{equation*}
\mathrm{Ent}_{p}(f^2) \coloneq \int f^2 \log \frac{f^2}{\bbE_p[f^2]} \, p \d x.
\end{equation*}
A density $p \in C^2(\mathbb{R}^{d})$ is called 
\emph{$\beta$-strongly log-concave} with some $\beta >0$, if the Hessian of its logarithm fulfills 
\begin{equation*}
\beta\,  I_d \preceq  -\nabla^2 \log p.
\end{equation*}
By the result of Bakry-{\'E}mery~\cite{Bakry1985Diffusions} a $\beta$-strongly log-concave density satisfies a log-Sobolev inequality with a constant $C_{\LSI}(p) = \beta^{-1}$.

The \emph{Kullback-Leibler} (KL) \emph{divergence} of two densities $p$ and $q$ is defined by $\kl (p,q) \coloneqq \int \log(\frac{p}{q}) \, p \d x$, if $q(x) = 0$ implies $p(x) = 0$ a.e., and otherwise by $\kl (p,q) \coloneqq + \infty$.

Then the following theorem was proved in~\cite{Vempala2019Rapid} and~\cite{Bakry1985Diffusions}, see also the recent paper~\cite{CELSZ2025}.

\begin{theorem} \label{thm:conv}
Let $U_s$ with law $p_{U_s}$ be generated by~\eqref{eq:langevin_basic} with an arbitrary initial distribution $p_{U_0}$. Then the following holds true.
\begin{itemize}
\item[(i)]
If $p$ satisfies a log-Sobolev inequality with constant $C_{\LSI}(p)$, then
\begin{equation*}
\kl(p_{U_s},p) \leq e^{-\frac{2s}{C_{\LSI}(p)}} \kl(p_{U_0},p),
\end{equation*}
\item[(ii)] If  $p$ is $\beta$-strongly log-concave, then 
\begin{equation*}
\kl(p_{U_s},p) \leq e^{-2s \beta} \kl(p_{U_0},p).
\end{equation*}
\end{itemize}
\end{theorem}

In the following, we denote by $\gamma_{\sigma^2}$  the density of a Gaussian distribution $\mathcal N(0, \sigma^2  I_d)$, and set $\gamma \coloneqq \gamma_1$ for the density of the standard Gaussian distribution. Further, $\Xi$ always denotes a standard Gaussian distributed random variable, so that $\sigma \Xi$, $\sigma \not = 0$, has density $\gamma_{\sigma^2}$.
Recall that for two independent random variables $X$ and $Z$ with densities
$p_X$ and $p_Z$, resp., the random variable $a X + b Z$, $a,b \in \mathbb R\setminus \{0\}$ has the density
\begin{equation} \label{eq:conv}
p_{a X + b Z} = p_{a X} * p_{b Z} =\int p_X(a^{-1} x) p_Z\big(b^{-1} (\cdot -x) \big)
\d x.
\end{equation}
In particular, the convolution of two Gaussians $\mathcal{N}(m, \sigma^2)$ and $\mathcal{N}(\tilde m, \tilde{\sigma}^2)$ results in the Gaussian 
$\mathcal{N}(m + \tilde m, \sigma^2 + \tilde{\sigma}^2)$.

%% file: method.tex

\section{Sampling via Stochastic Interpolants}\label{section:method}

To sample from a potentially complex target distribution $p_{X_1}$, we consider the linear interpolant~\cite{albergo2023building, albergo2025stochastic} for $(X_{0},X_{1}) \sim \gamma \otimes p_{X_1}$ defined by
\begin{equation}\label{eq:stochastic:interporlant}
X_t=tX_{1}+(1-t)X_{0}, \quad t\in(0,1).
\end{equation}
This evolution is governed by a  \textbf{continuity equation}
\begin{equation}\label{eq:transport}
\partial_{t}p_{X_t}(x_{t})+\nabla\cdot(u(t,x_{t})p_{X_t}(x_{t}))=0, \quad t\in(0,1), ~ x_{t}\in\mathbb{R}^{d}, 
\end{equation}
where the velocity field $u:(0,1)\times\bbR^{d}\to\bbR^{d}$ is defined as
\begin{align}
u(t,x_{t})&\coloneq\bbE[X_{1}-X_{0}|X_{t}=x_{t}]=-\frac{1}{1-t} x_{t}+\frac{1}{1-t}\bbE[X_{1}|X_{t}=x_{t}] \nonumber \\
&=-\frac{1}{1-t} x_{t}+\frac{1}{1-t} D(t,x_{t}), \label{eq:velocity}
\end{align}
see~\cite{AGS2008,liu2023flow}. Here $D:(0,1)\times\bbR^{d}\to\bbR^{d}$ is known as the denoiser in the context of generative models~\cite{Kingma2021Variational,Karras2022Elucidating,song2023Consistency,albergo2025stochastic}. If $X_1\in L^2(\mathbb R^d,\mathbb P)$, then $\|u\|_{L^2(p_{X_t})}\in L^1([0,1])$ is always true for a curve of type~\eqref{eq:stochastic:interporlant}, see~\cite[Theorem 4.6 and Corollary 6.4]{WS2025}. Note that Assumption \ref{assumption:Gaussian:convolution} below guarantees that $X_1\in L^2(\mathbb R^d, \mathbb P)$. Conditioned on $X_{1}=x_{1}$, the density of the interpolant $X_{t}$ is given as 
\begin{equation}\label{eq:transition}
p_{X_t|X_1 = x_1} = \gamma_{(1-t)^{2}}(\cdot - tx_{1}),
\end{equation}
and by the law of total probability we have
\begin{equation}\label{eq:marginal}
p_{X_t}
\coloneq\int p_{X_t|X_1= x_1}\, p_{X_1}( x_{1}) \d x_{1}=
\int\gamma_{(1-t)^{2}}(\cdot - tx_{1}) p_{X_1} (x_{1}) \d x_{1}.
\end{equation}
Moreover, we have the important relation
between the score of $p_{X_t}$ and the velocity field
\begin{equation}\label{eq:score:velocity}
 \nabla\log p_{X_t}(x_{t})= \frac{t}{1-t} u(t,x_{t})-\frac{1}{1-t}x_{t}, \quad t \in [0,1), ~ x_{t}\in\mathbb{R}^{d},
\end{equation}
see, e.g.,~\cite[Appendix A1]{zhang2024flow},~\cite[Remark 6.5]{WS2025} and~\cite[Proposition 3.1]{ding2024characteristic}. 

Let $T_0\in(0,1)$ be a fixed starting time point. The characteristic curves of the continuity equation~\eqref{eq:transport} define a flow map $\psi:(T_0,1)\times\bbR^{d}\rightarrow\bbR^{d}$, specified by the following \textbf{probability flow ODE}:
\begin{equation}\label{eq:PFODE}
\begin{aligned}
\frac{\d}{\dt}\psi(t,x_{T_{0}})&=u(t,\psi(t,x_{T_{0}})), \quad t\in (T_0,1), ~ x_{T_{0}}\in\mathbb{R}^{d}, \\
\psi(T_0,x_{T_{0}})&=x_{T_{0}}. 
\end{aligned}
\end{equation}
If the following condition is fulfilled
\begin{align}\label{eq:cond}
\int_0^1\sup_B\{\|u(t,\cdot)\|_{2}\}+\mathrm{Lip}(u(t,\cdot),B)\d t <\infty \text{ for all compact }B\subset \bbR^d,
\end{align}
then the ODE has a unique solution, see~\cite[Proposition 8.1.8]{AGS2008}. Further, it
holds that if $X_{T_0}\sim p_{X_{T_0}}$, then $\psi(t,X_{T_0})\sim p_{X_t}$, which implies
\begin{equation*}
p_{X_t} = \psi(t,\cdot)_{\sharp} p_{X_{T_0}}
\coloneqq  p_{X_{T_0}} \circ \psi^{-1}(t,\cdot), \quad t\in (T_0,1).
\end{equation*}

\par Our sampling algorithm via the probability flow ODE is based on the observation that intuitively, the Gaussian-convolved distribution $p_{X_{T_0}}$ in~\eqref{eq:marginal} is already much easier to sample than the original distribution $p_{X_1}$, when $T_0$ is small. 
Let us illustrate this by the following example. See~\cite{Biroli2024Dynamical,Ambrogioni2025Statistical} for more discussions on mode speciation.

\begin{example}[Gaussian mixture] \label{example:gmm}
Consider a mixture of two 1-dimensional Gaussians
\begin{equation*}
p_{X_1}(x_1) = \frac{1}{2}\gamma(x_1 - m) + \frac{1}{2}\gamma(x_1 + m)
\end{equation*}
with a constant $m > \sqrt{2}$. Then, the marginal density of the interpolant is a Gaussian mixture
\begin{align} \label{eq:gmm_interpolant}
p_{X_t}(x_{t}) 
&= \frac{1}{2}\int \gamma_{(1-t)^2}(x_{t} - t x_1) \gamma(x_1 - m) \d x_1 + \frac{1}{2}\int \gamma_{(1-t)^2}(x_{t} - t x_1) \gamma(x_1 + m) \d x_1 \nonumber \\
&= \frac{1}{2}\gamma_{\sigma_t^2}(x_{t} - mt) + \frac{1}{2}\gamma_{\sigma_t^2}(x_{t} + mt), \qquad \sigma_t^2 \coloneqq t^2 + (1-t)^2.
\end{align}
Differentiating with respect to $x_{t}$ yields
\begin{equation}
\frac{\d}{\d x_{t}}p_{X_t}(x_{t}) 
= -\frac{1}{(2\pi\sigma_t^2)^{1/2}}\Big\{\frac{x_{t}-mt}{2\sigma_t^2}\exp\Big(-\frac{(x_{t}-mt)^2}{2\sigma_t^2}\Big)+\frac{x_{t}+mt}{2\sigma_t^2}\exp\Big(-\frac{(x_{t}+mt)^2}{2\sigma_t^2}\Big)\Big\},
\end{equation}
so that critical points have to fulfill 
\begin{equation*}
(x_{t}-mt)\exp(-\frac{(x_{t}-mt)^2}{2\sigma_{t}^2})=-(x_{t}+mt)\exp(-\frac{(x_{t}+mt)^2}{2\sigma_{t}^2}),
\end{equation*}
which simplifies to the transcendental equation $x_{t} = mt \tanh(\frac{mtx_{t}}{\sigma_t^2})$. While $x_{t}=0$ is always a solution, non-zero solutions indicating bimodality emerge when the curvature at the origin becomes positive, i.e., when the squared mean exceeds the variance $\sigma_t^2$. Thus, we obtain
\begin{align*}
    (m t)^2 &> \sigma_t^ 2\quad \Longleftrightarrow \quad   (m^2 - 2)t^2 + 2t - 1 > 0.
\end{align*}
Solving for root in $(0,1)$, we find the distribution transitions from unimodal to bimodal at
\begin{equation}
    t^* = \frac{\sqrt{m^2 - 1}-1}{m^2 - 2}.
\end{equation}
Crucially, the interpolant density $p_{X_t}$ remains unimodal for all $t \in (0, t^*]$ which is a key property, see Figure~\ref{fig:mixture_transition}.
\hfill $\Box$
\end{example}

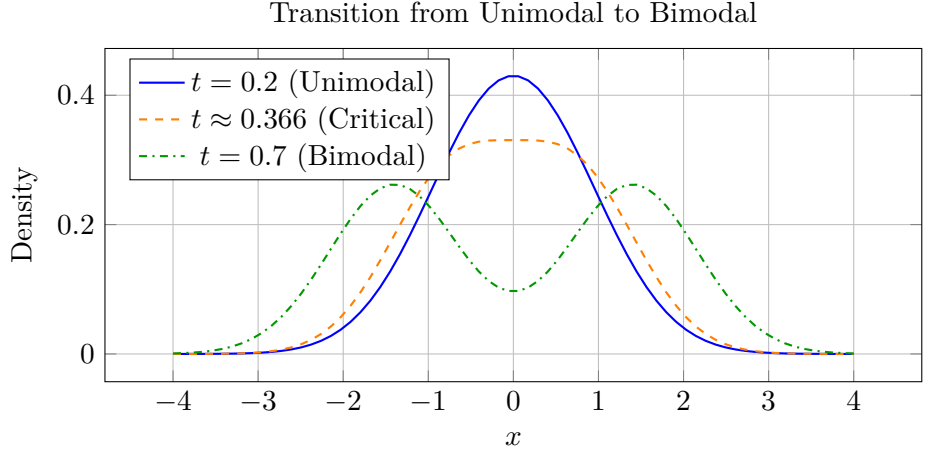
\begin{figure}[htbp]
\centering
\begin{tikzpicture}
    \begin{axis}[
        width=0.8\textwidth,
        height=6cm,
        xlabel={$x$},
        ylabel={Density},
        domain=-4:4,
        samples=70,
        legend pos=north west,
        grid=major,
        title={Transition from Unimodal to Bimodal}
    ]
    \addplot [blue, thick] {
        0.5 * (1/sqrt(2*3.14159*0.68)) * exp(-(x-0.4)^2/(2*0.68)) +
        0.5 * (1/sqrt(2*3.14159*0.68)) * exp(-(x+0.4)^2/(2*0.68))
    };
    \addlegendentry{$t=0.2$ (Unimodal)}

    \addplot [orange, thick, dashed] {
        0.5 * (1/sqrt(2*3.14159*0.536)) * exp(-(x-0.732)^2/(2*0.536)) +
        0.5 * (1/sqrt(2*3.14159*0.536)) * exp(-(x+0.732)^2/(2*0.536))
    };
    \addlegendentry{$t \approx 0.366$ (Critical)}

    \addplot [green!60!black, thick, dashdotted] {
        0.5 * (1/sqrt(2*3.14159*0.58)) * exp(-(x-1.4)^2/(2*0.58)) +
        0.5 * (1/sqrt(2*3.14159*0.58)) * exp(-(x+1.4)^2/(2*0.58))
    };
    \addlegendentry{$t=0.7$ (Bimodal)}
    \end{axis}
\end{tikzpicture}
\caption{Density $p_{X_t}$ in~\eqref{eq:gmm_interpolant} with $m=2$ at three time points, illustrating the pitchfork bifurcation.}
\label{fig:mixture_transition}
\end{figure}

Based on the above observation, depending  on  some critical hyperparameter $T_0\in(0,1)$, we propose the following \textbf{sampling strategy}.
\\[1ex]
1.
\textbf{Velocity field estimator:} 
For any $t \in [T_0,1)$ and any $x_t \in \mathbb R^d$, we propose the construction of estimators of the velocity field $u(t,x_t)$ based on~\eqref{eq:velocity}, or~\eqref{eq:stable:velocity:representation} introduced later, by approximating the conditional expectation using samples from $p_{X_{1}|X_{t}=x_{t}}$. These samples are provided by the Euler-Maruyama scheme of the Langevin diffusion of $Z_s \coloneq Z_s^{t,x_t}$ determined by
\begin{equation*}
\d Z_{s} = \nabla\log p_{X_{1}|X_{t}=x_{t}}(Z_{s})\ds + \sqrt{2}\d B_{s}, \quad s\geq 0.
\end{equation*}    
2. \textbf{Computation of flow ODE:}
Having such velocity estimators at hand, it is used twofold, namely to sample from the initial distribution $p_{X_{T_0}}$ of our probability flow ODE~\eqref{eq:PFODE} and to simulate the flow itself:
\begin{itemize}
\item[(i)] \textbf{Flow initialization:} We apply the Langevin diffusion
\begin{equation*}
\d U_{s}=\nabla\log p_{X_{T_{0}}}(U_{s})\d r+\sqrt{2}\d B_{s}, \quad s\geq 0.
\end{equation*}
starting in a standard normal distribution $U_0\sim\calN(0, I_{d})$. Using our velocity estimators, the score  $\nabla\log p_{X_{T_{0}}}$ can be simply approximated by~\eqref{eq:score:velocity}.
\item[(ii)] \textbf{Sampling via flow ODE:}
Having an estimator for the velocity field $u$, we can successive run the time-discretization of the probability flows ODE~\eqref{eq:PFODE} starting in samples from $p_{X_{T_0}}$.
\end{itemize}

\begin{figure}
\includegraphics[clip, trim={0cm 0.0cm 0cm 0.0cm},width=\textwidth]{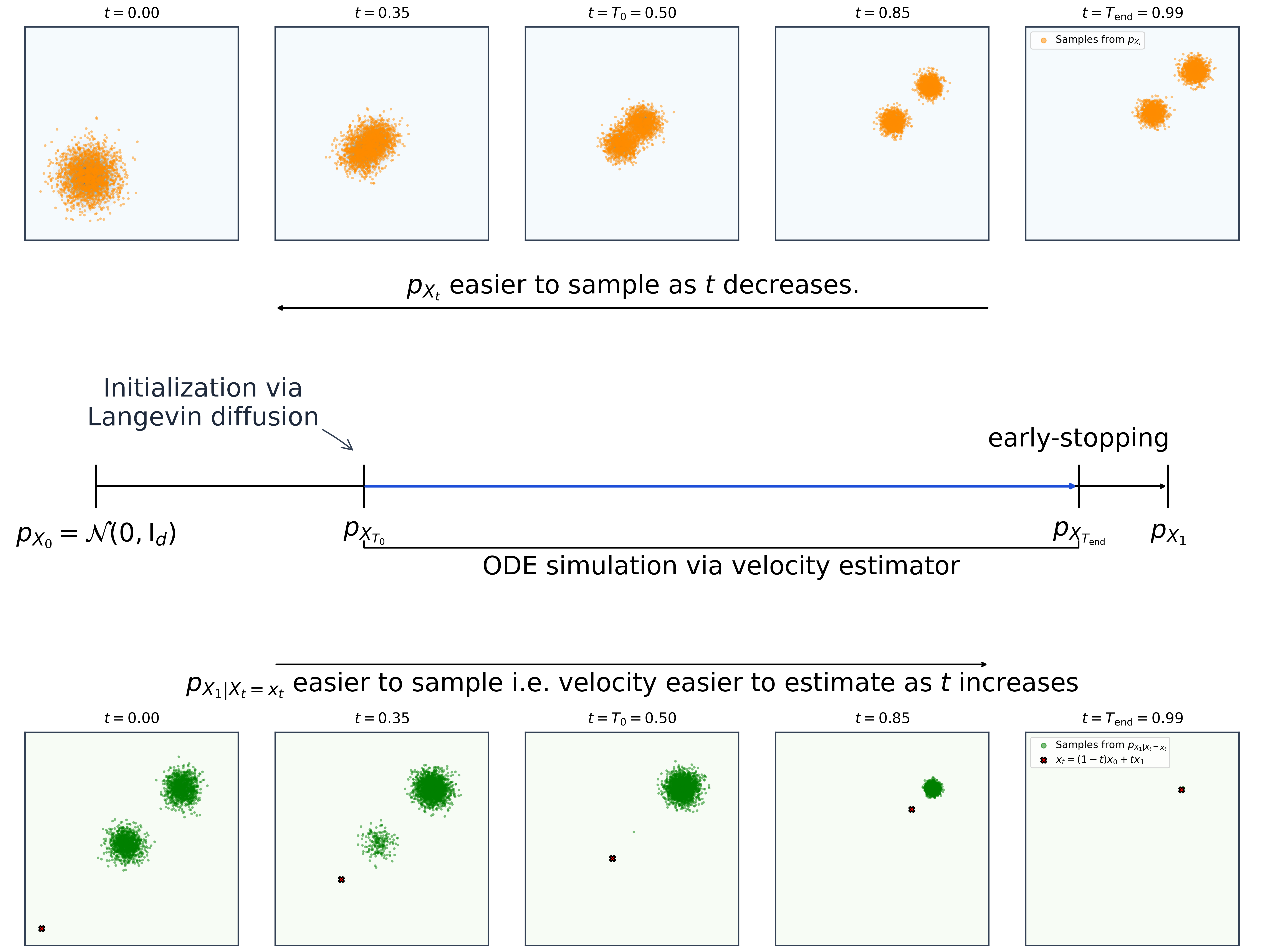}
\caption{Illustration of the algorithm, which consists of two parts. For small $t$ it is easy to sample from $p_{X_t}$ via Langevin diffusion. For the first part $\mathrm{(i)}$ \textbf{Flow initialization}, we initialize the samples at time $T_0$ via this technique. For the second part $\mathrm{(ii)}$ \textbf{Sampling via flow ODE}, note that for big $t$, it is easy to sample from $p_{X_1|X_t=x_t}$ and thus the velocity field is easy to approximate. That is why for $t\in(T_0,T_{\mathrm{end}}]$ we estimate the velocity field and simulate the ODE.}\label{fig:strategy}
\end{figure}

In the following sections, we describe our strategy in detail including convergence results.

We adopt a standard Gaussian convolution assumption, see, e.g.,~\cite{beyler2025convergence,ding2024characteristic, Grenioux2024Stochastic,saremi2024chain}.

\begin{assumption}\label{assumption:Gaussian:convolution}
There exists a random variable 
$Y \in \mathbb R^d$ with law  $p_Y$, where
$$\supp(p_Y)\subseteq B(0,R) \coloneq\{y\in\bbR^{d}:\| y\|_{2}\leq R\}, \quad R>0,
$$
so that 
$X_1 = Y + \sigma \Xi$, $\sigma \in (0,1)$, i.e., the target density fulfills
\begin{equation}\label{eq:ass1}
p_{X_1} =\int\gamma_{\sigma^{2}}(\cdot-y) p_Y(y)\d y.
\end{equation}
\end{assumption}

The following result from~\cite[Corollary 1]{Chen2011Dimension} states that $p_{X_1}$ satisfies a log-Sobolev inequality.

\begin{lemma}\label{lem_chen}
Let $p$ be a density with $\supp ( p ) \subseteq B(0,R)$. Then, for any $\sigma^2 \in (0,1)$, the density $p * \gamma_{\sigma^2}$  fulfills a log-Sobolev inequality with constant
\begin{equation*}
C_{\LSI} (p * \gamma_{\sigma^2}) \le 6\,(4 R^2 + \sigma^2 )\exp\Big(\frac{4R^2}{\sigma^2}\Big)
= 6\sigma^2 \Big(1+ \frac{4R^2}{\sigma^2} \Big)\exp\Big(\frac{4R^2}{\sigma^2}\Big)
\le 6 \exp\Big(\frac{8R^2}{\sigma^2}\Big).
\end{equation*}
\end{lemma}

\begin{remark}[Dimension-independent log-Sobolev constant]
The bound in Lemma~\ref{lem_chen} is independent of the dimension $d$. This independence is only apparent, however, since the dimension is implicitly encoded in the $\ell^{2}$-radius $R$. Indeed, if the support condition in Assumption~\ref{assumption:Gaussian:convolution} were instead imposed with respect to the $\ell^{\infty}$-norm, say $\supp(p_Y)\subseteq\{y\in\mathbb{R}^{d}:\|y\|_{\infty}\leq R_{\infty}\}$, then $\|y\|_{2}\leq\sqrt{d}\,\|y\|_{\infty}$ yields the $\ell^{2}$-radius $R=\sqrt{d}\,R_{\infty}$, and the resulting bound $6\exp\big(8dR_{\infty}^{2}\sigma^{-2}\big)$ grows exponentially in $d$.
\end{remark}

\section{Velocity Field Estimators}\label{section:method:velocity}

\par According to~\eqref{eq:velocity}, estimating the velocity field $u(t,\cdot)$ of the flow ODE reduces to estimating the conditional expectation of $X_{1}$ given $X_{t} = x_t$. 
By  Bayes' rule, we know that
\begin{equation}\label{eq:denoising:density}
p_{X_{1}|X_{t}=x_{t}} = \frac{1}{p_{X_t}} \gamma_{(1-t)^{2}}(x_{t}-t\cdot )p_{X_{1}} =\frac{1}{p_{X_t}}
\exp\Big\{-\frac{\|x_{t}-t \cdot \|_{2}^{2}}{2(1-t)^{2}}+\log p_{X_{1}}\Big\},
\end{equation}
which is knows as the denoising density.

\subsection{Langevin diffusion}
For a fixed $x_{t}\in\bbR^{d}$, the Langevin diffusion admitting $p_{X_{1}|X_{t}=x_{t}}$ as its stationary density is given by
\begin{equation}\label{eq:langevin:score:estimation}
\d Z_{s}^{t,x_t} = \nabla\log p_{X_{1}|X_{t}=x_{t}}(Z_{s}^{t,x_t})\ds + \sqrt{2}\d B_{s}, \quad s\geq 0. 
\end{equation}
For simplicity of notation, we write $Z_s$ instead of $Z_{s}^{t,x_t}$ if the conditioning on $X_t = x_t$ is clear from the context.
By~\eqref{eq:denoising:density}, the score function reads as
\begin{equation}\label{eq:denoising_score_function}
\nabla\log p_{X_{1}|X_{t}=x_{t}}(x_{1}) = \frac{t(x_{t}-tx_{1})}{(1-t)^{2}} + \nabla\log p_{X_{1}}(x_{1}).
\end{equation}
Langevin diffusions involving this type of score have previously been investigated in the context of proximal sampling, parallel tempering, and annealed Langevin diffusion, e.g., in~\cite{Lee2021Structured,Dong2022Spectral,guo2025provable}. 

\par By the following proposition, the Langevin diffusion~\eqref{eq:langevin:score:estimation} converges to the denoising density~\eqref{eq:denoising:density} under mild conditions. For a similar result with different interpolation, see also \cite{Grenioux2024Stochastic}. The proof is provided in Appendix~\ref{section:proof:method_1}.

\begin{proposition}\label{lemma:denoising:langevin:convergence}
Let Assumption~\ref{assumption:Gaussian:convolution} be fulfilled, and let $X_t$ be given by the linear interpolant~\eqref{eq:stochastic:interporlant}. 
For $t \in (0,1)$ and $\sigma^2 >0$, set
\begin{equation}\label{beta_t}
\beta_t \coloneqq \frac{t^2}{(1-t)^{2}}+\frac{1}{\sigma^{2}} - \frac{R^2}{\sigma^4}.
\end{equation}
Then, for every $t\in(0,1)$, it holds
\begin{equation*} 
\beta_t \,  I_{d}
\preceq 
-\nabla^2 \log p_{X_1|X_t = x_t}
\preceq
\Big(\beta_t + \frac{R^2}{\sigma^4} \Big)\,  I_{d},
\end{equation*}
and $p_{X_1|X_t = x_t}$ is $\beta_t$-strongly log-concave  for $t \in (T^*,1)$, where
\begin{equation}\label{eq:T}
T^* =
\left\{
\begin{array}{ll}
0& \text{if } R^{2}\leq\sigma^{2}
\\[1ex]
\frac{1}{2} & \text{if } R^{2}=\sigma^{2}+\sigma^{4},
\\[1ex]
\frac{R^{2}-\sigma^{2}-\sigma^{2}\sqrt{R^{2}-\sigma^{2}}}{R^{2}-\sigma^{2}-\sigma^{4}}&\text{ otherwise}.
\end{array}
\right.
\end{equation}
We have $T^*< \frac12$ if $\sigma^{2}<R^{2}<\sigma^{2}+\sigma^{4}$, and $T^*> \frac12$ if $R^{2}>\sigma^{2}+\sigma^{4}$.

Further, for $t \in (T^*,1)$, it holds
\begin{equation}\label{est}
\kl \left(p_{Z_s}, p_{X_1|X_t = x_t} \right)
\leq
e^{-2\beta_t s} \, 
\kl \left(p_{Z_0},p_{X_1|X_t = x_t} \right).
\end{equation}
where $Z_{s}$,  $s\geq 0$ is defined by the Langevin diffusion~\eqref{eq:langevin:score:estimation}. 
\end{proposition}

The proposition indicates that the denoising density $p_{X_1|X_t = x_t}$ becomes $\beta_t$-strongly log-concave with a larger $\beta_{t}$ as $t$ increases. Specifically, in the over-smoothed regime, i.e., $R^{2}\leq\sigma^{2}$, the denoising density is log-concave for all $t\in(0,1)$. In contrast, in the under-smoothed regime, $R^{2}>\sigma^{2}$, to ensure the log-concavity of the denoising density for any $t\in(T^*,1)$, we cannot allow the initial time $T_0$  to approach zero. 

\subsection{Euler-Maruyama scheme} 
Although the convergence of the Langevin diffusion ~\eqref{eq:langevin:score:estimation} is theoretically guaranteed, analytical simulation is intractable, necessitating Euler-Maruyama time-discretization. Let $\eta>0$ denote the step size, and consider the time partition $0<\eta<\cdots<K\eta$. Solving the SDE~\eqref{eq:langevin:score:estimation} with the Euler-Maruyama scheme results in
\begin{equation}\label{eq:langevin:monte:carlo}
\bar{Z}_{(k+1)\eta} = \bar{Z}_{k\eta}+\eta\nabla\log p_{X_{1}|X_{t}=x_{t}}(\bar{Z}_{k\eta})+\sqrt{2\eta}\, \xi_{k}, \quad \xi_{k}\sim\calN(0,I_{d})
\end{equation} 
with the score from~\eqref{eq:denoising_score_function} and $\bar{Z}_{0}\sim p_{Z_{0}}$. This scheme is also known as Langevin Monte Carlo (LMC) or unadjusted Langevin algorithm (ULA).  

\par By~\eqref{est}, the convergence of the Langevin diffusion~\eqref{eq:langevin:score:estimation} depends on the quality of the initialization distribution $p_{Z_{0}}$. To select an initial distribution that approximates the target density $p_{X_1|X_t=x_t}$ as closely as possible, we employ importance sampling~\cite{neal2001annealed}. Specifically, we utilize the Gaussian $\calN(t^{-1}x_{t},(1-t)^{2}t^{-2} I_{d})$ as a proposal distribution, and treat the target density $p_{X_{1}}$ as the weight function:
\begin{equation}\label{eq:is}
p_{X_{1}|X_{t}=x_{t}}\propto\underbrace{p_{X_{1}}(x_{1})}_{\text{weight}}\underbrace{\gamma_{(1-t)^{2}t^{-2}}(x_{1}-t^{-1}x_{t})}_{\text{proposal}}.
\end{equation}
We first take i.i.d. samples $Z^{1},\ldots,Z^{n}$ from the Gaussian proposal distribution. Substituting the proposal distribution with its empirical approximation and applying~\eqref{eq:is}, we obtain an empirical approximation of the denoising distribution
\begin{equation}\label{eq:is:Z0}
\what{p}_{X_{1}|X_{t}=x_{t}}\coloneq\sum_{i=1}^{n}\frac{p_{X_{1}}(Z^{i})}{\sum_{k=1}^{n}p_{X_{1}}(Z^{k})}\delta_{Z^{i}}.
\end{equation}
This empirical measure $\what{p}_{X_{1}|X_{t}=x_{t}}$ converges to the denoising distribution $p_{X_{1}|X_{t}=x_{t}}$ as the number of samples $n$ approaches infinity~\cite{neal2001annealed}. By resampling from this empirical measure, we generate particles that approximately follow the denoising distribution, providing a suitable initialization for Langevin sampling, i.e., $p_{Z_{0}} \coloneq \what{p}_{X_{1}|X_{t}=x_{t}}$. Some other papers which employ similar importance sampling strategies are \cite{DPVH2025,durmus2026sampling,huang2024reverse,VCK2024}.

\subsection{Langevin-based velocity approximation}\label{subsect:vel_approx}
For a fixed $(t,x_{t})\in(0,1)\times\mathbb{R}^{d}$, using~\eqref{eq:velocity} and i.i.d. random variables 
$\bar{Z}_{K\eta}^{1},\ldots,\bar{Z}_{K\eta}^{n}$ generated by~\eqref{eq:langevin:monte:carlo}, we propose the following \textbf{estimator of the velocity field}:
\begin{equation}\label{eq:velocity:estimator}
\what{U}(t,x_{t})\coloneqq  -\frac{1}{1-t}x_{t} + \frac{1}{1-t}\what{D}(t,x_{t}),
\end{equation}
where the estimator of the denoiser in~\eqref{eq:velocity} is defined as 
\begin{equation}\label{eq:denoiser:estimator}
\what{D}(t,x_{t}) \coloneq \frac{1}{n}\sum_{i=1}^n\bar{Z}_{K\eta}^i.
\end{equation}
The Langevin-based velocity estimator~\eqref{eq:velocity:estimator} is also utilized in~\cite{durmus2026sampling} for sampling from distributions on Riemannian manifolds.

\par The velocity estimator defined in~\eqref{eq:velocity:estimator} diverges if $t$ approaches $1$, resulting in numerical instabilities. To address this issue, we introduce the following rescaled representation of the velocity field, similarly as in~\cite[Equation 9]{akhound2024iterated},~\cite[Corollary 2.4]{bortoli2024target}, and~\cite[Equation (10)]{Grenioux2024Stochastic}. 

\begin{proposition}[Rescaled representation of velocity field]\label{proposition:rescaling:velocity}
The velocity field $u:\bbR\times\bbR^{d}\to\bbR^{d}$ in~\eqref{eq:velocity} satisfies for any $t\in(0,1)$ and any $x_{t}\in\bbR^{d}$ the relation
\begin{align}
u(t,x_{t}) &= \frac{1}{t}x_{t} + \frac{1-t}{t^2}\bbE[\nabla\log p_{X_1} (X_1)|X_t= x_t] \nonumber \\
&= \frac{1}{t}x_{t} + \frac{1}{t^2}F(t,x_t). \label{eq:stable:velocity:representation}
\end{align}
\end{proposition}

\par The proof of Proposition~\ref{proposition:rescaling:velocity} is provided in Appendix~\ref{section:proof:method_1}. For a fixed $(t,x_{t})\in(0,1)\times\mathbb{R}^{d}$, using the Monte Carlo method and $\bar{Z}_{K\eta}^{1},\ldots,\bar{Z}_{K\eta}^{n}$ generated by~\eqref{eq:langevin:monte:carlo}, we obtain the \textbf{stable estimator of the velocity field}:
\begin{equation}\label{eq:stable:velocity:estimator}
\widehat{U}_{\mathrm{stab}}(t,x_{t}) \coloneq \frac{1}{t}x_{t} + \frac{1}{t^2}\what{F}(t,x_{t}),
\end{equation}
where
\begin{equation}\label{eq:stable:F:estimator}
\what{F}(t,x_{t}) \coloneq \frac{1-t}{n}\sum_{i=1}^{n}\nabla\log p_{X_1}(\bar{Z}_{K\eta}^{i}).
\end{equation} 
In contrast to the estimator in~\eqref{eq:velocity:estimator}, the stable velocity estimator $\widehat{U}_{\mathrm{stab}}$ eliminates the numerical instability as $t$ approaches $1$. Note that the stabilized estimator is superior for times near to one, while for time near to zero the original estimator should be preferred, see~\cite{bortoli2024target}.

\par The complete algorithm for the velocity estimation is summarized in Algorithm~\ref{alg:velocity:estimation}.

\begin{algorithm}[htbp]\label{alg:velocity:estimation}
\caption{Langevin-based velocity field estimation}
\KwIn{Time and location of interest $(t,x_{t})$, target score $\nabla\log p_{X_{1}}$,  step size $\eta$, and number of steps $K$.}
\KwOut{Velocity estimator $\what{U}(t,x_{t})$ or $\widehat{U}_{\mathrm{stab}}(t,x_{t})$.}
Initialize particles $\bar{Z}_{0}^{1},\ldots,\bar{Z}_{0}^{n}$ via importance sampling~\eqref{eq:is:Z0}. \\    
\For {$k= 0,\ldots,K-1$}{
\texttt{\# Calculate the denoising score} \\
$S_{k+1}^{i}\leftarrow-\frac{t^{2}}{(1-t)^{2}}\bar{Z}_{k\eta}^{i}+\frac{t}{(1-t)^{2}}x_{t}+\nabla\log p_{X_{1}}(\bar{Z}_{k\eta}^{i})$ for $1\leq i\leq n$. \\
\texttt{\# Langevin update} \\
$\bar{Z}_{(k+1)\eta}^{i}\sim\calN(\bar{Z}_{k\eta}^{i}+\eta S_{k+1}^{i},2\eta \,  I_{d})$ for $1\leq i\leq n$. \\
} 
\texttt{\# Monte Carlo velocity estimator} \\
\eIf{\normalfont using stable estimation}{
$\widehat{U}_{\mathrm{stab}}(t,x_{t}) \leftarrow \frac{1}{t}x_{t} + \frac{1-t}{t^2}\frac{1}{n}\sum_{i=1}^{n}\nabla\log p_{X_1}(\bar{Z}_{K\eta}^{i})$.}{
$\what{U}(t,x_{t})=-\frac{1}{1-t}x_{t}+ \frac{1}{1-t}\frac{1}{n}\sum_{i=1}^{n}\bar{Z}_{K\eta}^{i}$.
}
\Return{\normalfont $\what{U}(t,x_{t})$ or $\widehat{U}_{\mathrm{stab}}(t,x_{t})$}
\end{algorithm}

\section{Computation of Flow ODE}\label{section:computation:flow}

Based on the velocity field estimators constructed in the previous section,  we can handle the probability flow ODE.
We start with the initialization, i.e. sampling from $p_{X_{T_0}}$, where $T_0 > T^{*}$ and consider the probability flow ODE numerically afterwards.

\subsection{Flow initialization}\label{section:method:initialization}
Langevin diffusion admitting  $p_{X_{T_0}}$ as its stationary distribution reads as
\begin{equation}\label{eq:langevin:diffusion:warmstart}
\d U_{s}=\nabla\log p_{X_{T_{0}}}(U_{s})\d s + \sqrt{2}\d B_{s}, \quad s\geq 0.
\end{equation}
In practice, we initialize $U_{0}$ 
by the standard Gaussian distribution $\calN(0, I_{d})$.
For $T_0 > T^*$, we apply a score estimator based on the velocity estimator~\eqref{eq:velocity:estimator} and~\eqref{eq:score:velocity}, i.e.,
\begin{equation}\label{eq:score:estimation}
\what{S}(T_{0},x_{T_{0}})=\frac{T_{0}}{1-T_{0}}\what{U}(T_{0},x_{T_{0}})-\frac{1}{1-T_{0}}x_{T_{0}}.
\end{equation}
Alternatively, one may employ the stable velocity estimator~\eqref{eq:stable:velocity:estimator} to obtain a score estimator through~\eqref{eq:score:velocity}. However, this additional stabilization is generally unnecessary in our setting, since $T_{0}$ is not close to one. For simplicity, we therefore restrict our attention to the score estimator based on~\eqref{eq:velocity:estimator}. 

\par Let $\tau>0$ be the step size, and let $L$ be the number of steps. Then, applying the Euler-Maruyama discretization with step size $\tau>0$ and $L$ steps yields 
\begin{equation}\label{eq:langevin:monte:carlo:warmstart}
\what{U}_{(\ell+1)\tau}=\what{U}_{\ell\tau}+\tau\what{S}(T_{0},\what{U}_{\ell\tau})+\sqrt{2\tau}\zeta_{\ell}, \quad \zeta_{\ell}\sim \calN(0,I_{d}), \quad 0\leq\ell\leq L-1.
\end{equation} 
The resulting particle $\what{U}_{L\tau}$ approximates  $p_{X_{T_{0}}}$, thereby serving as the initialization for the probability flow ODE~\eqref{eq:PFODE}. 
The initialization algorithm is summarized in Algorithm~\ref{alg:initialization}.

\begin{algorithm}[htbp]\label{alg:initialization}
\caption{Initialization of probability flow ODE}
\KwIn{The initialization time $T_{0}$, the step size $\tau$, and the number of steps $L$.}
\KwOut{Particle $\what{U}_{L\tau}$ approximately following $p_{X_{T_{0}}}$.}
Initialize particle: $\what{U}_{0}\sim\calN(0, I_{d})$. \\    
\For {$\ell\in\{0,\ldots,L-1\}$}{
\texttt{\# Score estimation} \\
Calculate the velocity estimator $\what{U}(T_{0},\what{U}_{\ell\tau})$ using Algorithm~\ref{alg:velocity:estimation}. \\
Calculate the score estimator $\what{S}(T_{0},\what{U}_{\ell\tau})$ using~\eqref{eq:score:estimation}. \\
\texttt{\# Langevin update} \\
$\what{U}_{(\ell+1)\tau}\sim\calN(\what{U}_{\ell\tau}+\tau\what{S}(T_{0},\what{U}_{\ell\tau}),2\tau I_{d})$. \\
} 
\Return{$\what{U}_{L\tau}$}
\end{algorithm}

By following proposition, $p_{U_{s}}$ converges exponentially to the stationary distribution $p_{X_{T_{0}}}$.

\begin{proposition}\label{lemma:initial:langevin:convergence}
Let Assumption~\ref{assumption:Gaussian:convolution}  be fulfilled. Let $p_{X_{T_{0}}}$ be the distribution defined by~\eqref{eq:marginal} and let $U_{s}$ be given by in~\eqref{eq:langevin:diffusion:warmstart}.
Then, for every $s>0$, it holds 
\begin{equation*}
\kl\left(p_{U_{s}},p_{X_{T_{0}}}\right)
\leq
\exp\Big\{
-\frac{s}{3} 
\exp \Big\{-\frac{8 T_0^2 R^2}{ T_0^2 \sigma^2+ (1-T_0)^2}\Big\} \Big\} \kl\left(p_{U_0},p_{X_{T_0}}\right).
\end{equation*}
\end{proposition}

\begin{proof}
By assumptions we have
\begin{equation*}
X_{T_0}=T_0 X_{1}+(1-T_0)X_{0}
\stackrel{\d}{=}T_0(Y+\sigma \Xi)+(1-T_0)X_{0}
\stackrel{\d}{=}
T_0 Y + \left(T_0^2 \sigma^2+ (1-T_0)^2\right)^\frac12 \Xi.
\end{equation*}
This is a convolution of a density with support in the ball of radius $T_0 R$ and a Gaussian of variance $T_0^2 \sigma^2+ (1-T_0)^2 \in (0,1)$.
Thus, by Lemma~\ref{lem_chen}, we conclude
\begin{align}\label{eq:Clsi}
C_{\LSI}(p_{X_{T_0}}) 
&
\le 6 \, \exp\Big\{\frac{8T_0^{2}R^{2}}{T_0^2 \sigma^2+ (1-T_0)^2}\Big\}
\end{align}
and the assertion follows by Theorem~\ref{thm:conv} (ii).
\end{proof}

The proposition suggests that the Langevin diffusion~\eqref{eq:langevin:diffusion:warmstart} converges more rapidly when $T_0^2/(T_0^2 \sigma^2+ (1-T_0)^{2})$ is small, i.e.,
$T_{0}$ is small. 
Combining Propositions~\ref{lemma:denoising:langevin:convergence} and~\ref{lemma:initial:langevin:convergence} reveals a trade-off in the selection of $T_{0}$: while a larger initial time $T_{0}$ accelerates the convergence of~\eqref{eq:langevin:score:estimation}, it hinders the convergence of~\eqref{eq:langevin:diffusion:warmstart}.
A $T_0$-tradeoff was also discussed in another setting in~\cite{Grenioux2024Stochastic}.

The following remark underlines the intuition from Example~\ref{example:gmm}
that $p_{T_0}$ is significantly better tractable by Langevin diffusion than the target distribution $p_{X_{1}}$.

\begin{remark}[Regularizing effect of Gaussian convolution]
By Assumption~\ref{assumption:Gaussian:convolution} and Lemma~\ref{lem_chen}
the target distribution $p_{X_{1}}$  satisfies a log-Sobolev inequality with constant $C_{\rm{LSI}}(p_{X_{1}})\le 6 \, {\rm e}^{\frac{8 R^2}{\sigma^2}}$, $\sigma \in (0,1)$.
By Theorem~\ref{thm:conv} (i), this provides an exponential convergence of the Langevin diffusion
\begin{equation*}
\d U_s = \nabla \log p_{X_1} ( U_s) \d s + \sqrt{2} \d B_s, \quad s \ge 0
\end{equation*}
towards the stationary distribution $p_{X_{1}}$. However, for a fixed initialization time $T_{0}\in(0,1)$, the log-Sobolev inequality constant $C_{\rm{LSI}}(p_{X_{T_{0}}})$ in~\eqref{eq:Clsi} is significantly smaller than $C_{\rm{LSI}}(p_{X_{1}})$, since  
\begin{equation*}
\frac{1}{\sigma^2}  > \frac{T_0^2}{T_0^2 \sigma^2+ (1-T_0)^2}
\end{equation*}
and the right-hand side becomes smaller the smaller $T_0 \in (T^*,1)$ is. In particular, in the regime, where $\sigma$ is small,
the Langevin diffusion towards $p_{X_{T_0}}$ converges much faster making sampling more efficient. For reverse diffusion interpolation a similar expression for the log-Sobolev constant under a Lipschitz assumption on the score function, which, in the present work, follows from Assumption 1, was established in~\cite{huang2024reverse}.
\end{remark}

\subsection{Sampling via flow ODE}\label{section:method:sampling}
 Once we can sample from our initial distribution $p_{X_{T_{0}}}$ and have  an estimator for the velocity field, we proceed to construct a practical sampling scheme via time discretization of the probability flow ODE~\eqref{eq:PFODE}. The use of a uniform time discretization is primarily motivated by simplicity. Alternative discretization strategies, such as log-SNR-based time discretizations could also be applied~\cite{Grenioux2024Stochastic}. Notably, both the vanilla velocity estimator~\eqref{eq:velocity:estimator} and the stable velocity estimator~\eqref{eq:stable:velocity:estimator} exhibit a semi-linear structure. Consequently, we can integrate the linear part exactly and approximate only the nonlinear one using an exponential integrator~\cite{Hochbruck2010Exponential}, which offers certain advantages over standard explicit Euler scheme.

\par We apply the exponential integrator to the probability flow ODE driven by the stable velocity estimator~\eqref{eq:stable:velocity:estimator}. For a total number $M$ of discretization steps, we choose the step size $h\coloneqq (T_{{\rm end}}-T_{0})/M$, and set $t_{m}\coloneqq T_{0}+mh$, $0\le m\le M$. In each sub-interval, we approximate the nonlinear term of the velocity estimator using its value at the left endpoint
\begin{equation*}
\frac{\d}{\dt}\what{\psi}_{\rm stab}(t,x_{T_{0}}) 
= \frac{1}{t}\what{\psi}_{\rm stab}(t,x_{T_{0}}) + \frac{1}{t^{2}}\what{F}(t_{m},\what{\psi}_{\rm stab}(t_{m},x_{T_{0}})), \quad t\in(t_{m},t_{m+1}).
\end{equation*}
Hence, we obtain 
\begin{equation*}
\frac{1}{t}\frac{\d}{\dt}\what{\psi}_{\rm stab}(t,x_{T_{0}}) 
= \frac{1}{t^{2}}\what{\psi}_{\rm stab}(t,x_{T_{0}}) + \frac{1}{t^{3}}\what{F}(t_{m},\what{\psi}_{\rm stab}(t_{m},x_{T_{0}})), \quad t\in(t_{m},t_{m+1})
\end{equation*}
and using the integrating factor $1/t$ further 
\begin{equation*}
\frac{\d}{\dt}\Big(\frac{1}{t}\what{\psi}_{\rm stab}(t,x_{T_{0}})\Big)=-\frac{1}{t^{2}}\what{\psi}_{\rm stab}(t,x_{T_{0}})+\frac{1}{t}\frac{\d}{\dt}\what{\psi}_{\rm stab}(t,x_{T_{0}})=\frac{1}{t^{3}}\what{F}(t_{m},\what{\psi}_{\rm stab}(t_{m},x_{T_{0}}))
\end{equation*}
integrating this ODE from $t_{m}$ to $t_{m+1}$ yields
\begin{equation}\label{eq:PFODE:velocity:ei:stable}
\begin{aligned}
\what{\psi}_{\rm stab}(t_{m+1},x_{T_{0}}) 
&= \frac{t_{m+1}}{t_{m}}\what{\psi}_{\rm stab}(t_{m},x_{T_{0}}) + \frac{t_{m}+t_{m+1}}{2t_{m}^{2}t_{m+1}} h\what{F}(t_{m}, \what{\psi}_{\rm stab}(t_{m},x_{T_{0}})), \\
\what{\psi}_{\rm stab}(T_{0},x_{T_{0}}) &= x_{T_{0}}.
\end{aligned}
\end{equation}
The exponential integrator in~\eqref{eq:PFODE:velocity:ei:stable} approximates only the nonlinear component of the velocity field while integrating the linear component analytically, thereby eliminating the approximation error associated with the linear part found in the Euler method.

\par For the  velocity estimator~\eqref{eq:velocity:estimator},  the exponential integrator yields the same  discrete-time flow as the Euler forward scheme
\begin{equation}\label{eq:PFODE:velocity:ei}
\begin{aligned}
\what{\psi}(t_{m+1},x_{T_{0}}) &= \frac{1-t_{m+1}}{1-t_m}\what{\psi}(t_{m},x_{T_{0}}) + \frac{1}{1-t_m} h\what{D}(t_m, \what{\psi}(t_{m},x_{T_{0}})), \\
\what{\psi}(T_{0},x_{T_{0}}) &= x_{T_{0}},
\end{aligned}
\end{equation}
see Remark~\ref{remark:euler:ei} at the end of Appendix \ref{section:proof:method}. 

\par Under mild assumptions, the pushforward density 
$(\what{\psi}(T_{\rm end},\cdot))_{\sharp} p_{T_{0}}$ or $(\what{\psi}_{\rm stab}(T_{\rm end},\cdot))_{\sharp} p_{T_{0}}$ approximates the target distribution $p_{X_{T_{\rm end}}}$, provided the velocity estimation errors is sufficiently small.
Incorporating the initialization mechanism from Section~\ref{section:method:initialization}, the complete sampling pipeline is given by
\begin{equation}\label{eq:PFODE:map}
\what{X}_{{T_{\rm end}}}\coloneq\what{\psi}(T_{\rm end},\what{X}_{T_{0}}), \quad \text{or} \quad \what{X}_{{T_{\rm end}}}\coloneq\what{\psi}_{\rm stab}(T_{\rm end},\what{X}_{T_{0}}), 
\end{equation}
where the samples $\what{X}_{T_{0}}\coloneq\what{U}_{L\tau}$ serve as the starting point for the ODE integration, and $\what{U}_{L\tau}$ is defined in~\eqref{eq:langevin:monte:carlo:warmstart}. The ODE integration stops at $T_{\rm end}<1$. Since $X_{T_{\rm end}}\stackrel{\d}{=}T_{\rm end}X_{1}+(1-T_{\rm end})X_{0}$ and $\bbE[X_{0}]=0$, the scaled particle 
\begin{equation*}
T_{\rm end}^{-1}X_{T_{\rm end}} \stackrel{\d}{=} X_{1}+\frac{1-T_{\rm end}}{T_{\rm end}}X_{0}
\end{equation*}
serves as an unbiased estimator of $X_{1}$, i.e., $\mathbb{E}[T_{\rm end}^{-1} X_{T_{\rm end}}]=\mathbb{E}[X_{1}]$. We apply this scaling map to the particle $\what{X}_{T_{\rm end}}$ in~\eqref{eq:PFODE:map}.

\par The full sampling procedure of~\eqref{eq:PFODE:velocity:ei} is summarized in Algorithm~\ref{alg:pfode:sampling}. The computing procedure of~\eqref{eq:PFODE:velocity:ei:stable} can be obtained similarly.

\begin{algorithm}[htbp]\label{alg:pfode:sampling}
\caption{Sampling via probability flow ODE}
\KwIn{The initial time $T_{0}$, the terminal time $T_{\rm end}$, and the number of steps $M$.}
\KwOut{Particle $T_{\rm end}^{-1}\what{\psi}(T_{\rm end},\what{X}_{T_{0}})$ approximately following $p_{X_{1}}$.}
Calculate the step size: $h=(T_{\rm end}-T_{0})/M$. \\
Generate time points: $t_{m}=T_{0}+mh$ for $0\leq m\leq M$. \\
Initialize particle $\what{\psi}(t_{0},\what{X}_{T_{0}})=\what{X}_{T_{0}}\leftarrow\what{U}_{L\tau}$ using Algorithm~\ref{alg:initialization}. \\
\For {$m\in\{0,\ldots,M-1\}$}{
\texttt{\# Denoising Estimation} \\
Obtain the denoiser $\what{D}_m \leftarrow \what{D}(t_{m},\what{\psi}(t_{m},\what{X}_{T_{0}}))$ using Algorithm \ref{alg:velocity:estimation}. \\
\texttt{\# Exponential integrator update} \\
$\alpha_m \leftarrow (1-t_{m+1})/(1-t_m)$. \\
$\what{\psi}(t_{m+1},\what{X}_{T_{0}})\leftarrow \alpha_m \what{\psi}(t_{m},\what{X}_{T_{0}}) + (1-\alpha_m)\what{D}_m$. \\
}
\Return{$T_{\rm end}^{-1}\what{\psi}(T_{\rm end},\what{X}_{T_{0}})$}
\end{algorithm}

\subsection{Well-posedness of the flow ODE}
Finally, we verify that the condition~\eqref{eq:cond} is fulfilled. Clearly, it suffices to prove that $\nabla u(t,x)$ is bounded for all 
$t\in (0,1)$ and all $x \in \mathbb R^d$.
Indeed, following similar lines as in the paper~\cite{ding2024characteristic} of some of the  authors, we have the following results which proofs can be found for convenience in Appendix~\ref{section:proof:method}.

\begin{lemma}[Linear growth of velocity field]\label{lem1}
Let Assumption~\ref{assumption:Gaussian:convolution} be fulfilled, and let the velocity field $u:(0,1)\times\mathbb{R}^{d}\to\mathbb{R}^{d}$ be defined by~\eqref{eq:velocity}. Then, for every $t\in(0,1)$, it holds 
\begin{equation*}
\|u(t,x_{t})\|_{2}\leq B(1+\|x_{t}\|_{2}), \quad x_{t}\in\bbR^{d},
\end{equation*}
where $B$ is a constant only depending on $R$ and $\sigma$.
\end{lemma}

\begin{theorem}[Lipschitz continuity of velocity field]\label{thm}
Let Assumption~\ref{assumption:Gaussian:convolution} be fulfilled, and let the velocity field $u:(0,1)\times\mathbb{R}^{d}\to\mathbb{R}^{d}$ be defined by~\eqref{eq:velocity}. Then, for every $t\in(0,1)$, it holds 
\begin{equation*}
\|\nabla u(t,x_t)\|_{\mathrm{op}}\leq G,
\end{equation*}
where $G$ is a constant only depending on $R$ and $\sigma$.
\end{theorem}

%% file: convergence.tex

\section{Non-Asymptotic Convergence Rates}\label{section:convergence}

\par This section presents an end-to-end error analysis for sampling via stochastic interpolants. In particular, we establish the convergence rates of the velocity estimator~\eqref{eq:velocity:estimator}, the Langevin-based initialization scheme~\eqref{eq:langevin:monte:carlo:warmstart}, and the probability flow ODE discretization~\eqref{eq:PFODE:velocity:ei:stable} and~\eqref{eq:PFODE:velocity:ei}.

\par Before proceeding, we recall the definition of the Wasserstein-2 distance; see~\cite[Chapter~6]{Villani2009Optimal}. The \emph{Wasserstein-2 distance} between two  measures $\mu, \nu \in \mathcal P_2(\mathbb R^d)$, where $\mathcal P_2(\mathbb R^d)$ 
denotes the space of probability measures with finite second moments,  is defined as
\begin{equation*}
W_{2}(\mu,\nu) \coloneq \Big(\inf_{\pi\in\Pi(\mu,\nu)}\int\|x-y\|_{2}^{2} \d \pi(x,y)
\Big)^{1/2} = \inf_{X \sim \mu, Y \sim \nu} \bbE \big[\|X-Y\|_{2}^{2}\big]^{1/2},
\end{equation*}
where $\Pi(\mu,\nu)$ denotes the set of all probability measures on $\bbR^d \times \bbR^d$
having marginal $\mu$ and $\nu$.
We are interested in bounding the Wasserstein-2 distance between the target density \(p_{X_{1}}\) and the distribution of particles generated by the approximate flow $\what{\psi}_{\mathrm{stab}}(T_{\rm end},\cdot)$ in~\eqref{eq:PFODE:velocity:ei:stable}, and $\what{\psi}(T_{\rm end},\cdot)$ in~\eqref{eq:PFODE:velocity:ei}.

\paragraph{Notations} 
Throughout this section, we say $f(\delta) \lesssim g(\delta)$, if there exist an absolute constant $C>0$, such that $f(\delta)\leq Cg(\delta)$. We also use the Big-Theta notation, meaning that $g(\delta) = \Theta(f(\delta))$, if both $f(\delta) \lesssim g(\delta)$ and $g(\delta)\lesssim f(\delta)$ hold.

\subsection{Error bounds for velocity estimations}

\par To begin with, the following lemma provides non-asymptotic convergence rates for velocity estimation in~\eqref{eq:velocity:estimator}.

\begin{lemma}[Error bounds of velocity estimation]\label{lemma:velocity:estimation}
Let Assumption~\ref{assumption:Gaussian:convolution} be fulfilled. For $t\in(T^{*},1)$ with $T^{*}$ in~\eqref{eq:T}, let $\what{U}(t,\cdot)$ be the velocity estimator defined in~\eqref{eq:velocity:estimator}, and let $\widehat{D}(t,\cdot)$ be the denoiser estimator defined in~\eqref{eq:denoiser:estimator}. Then, for any $\delta\in(0,1)$ and every $x_{t}\in\bbR^{d}$, we have
\begin{equation}\label{eq:lemma:velocity:estimation:0:0}
\bbE\Big[\|D(t,x_{t})-\what{D}(t,x_{t})\|_{2}^{2}\Big]
\lesssim\delta^{2},
\end{equation}
and, as a consequence, 
\begin{equation}\label{eq:lemma:velocity:estimation:0}
\bbE\Big[\|u(t,x_{t})-\what{U}(t,x_{t})\|_{2}^{2}\Big]
\lesssim
\frac{\delta^{2}}{(1-t)^{2}},
\end{equation}
provided that, with $\beta_t$ in~\eqref{beta_t}, the following relations are fulfilled
\begin{align*}
&n=\Theta\Big(\frac{(1-t)^{2}((1-t)^{2}R^2 +d\sigma^{4}t^{2}+d\sigma^{2}(1-t)^{2})}{(\sigma^{2}t^{2}+(1-t)^{2})^{2}\delta^{2}}\Big), \quad \eta=\Theta\Big(\frac{\sigma^{8}\beta_{t}^{2}\delta^{2}}{d(\sigma^{4}\beta_{t}+R^{2})^{2}}\Big), \\
&K=\Theta\Big(\frac{d}{\beta_{t}\delta^{2}}\Big(1+\frac{R^{2}}{\sigma^{4}\beta_{t}}\Big)^{2}\log\Big(\frac{W_{2}^{2}(p_{Z_{0}},p_{X_{1}|X_{t}=x_{t}})}{\delta^{2}}\vee 1\Big)\Big).
\end{align*}
\end{lemma}

\par The proof  is given in Appendix~\ref{section:proof:convergence:velocity}. 
Lemma~\ref{lemma:velocity:estimation} provides theoretical guidance for selecting the number of particles $n$ in the Monte Carlo approximation, the step size $\eta$ of the Langevin algorithm, and the number of Langevin iterations $K$. In particular, larger values of $t$ correspond to larger $\beta_{t}$, which in turn require fewer Langevin iterations $K$ and allow larger step size $\eta$. Moreover, larger $t$ also necessitates a fewer number $n$ of particles  for the Monte Carlo approximation. Consequently, the difficulty of denoiser estimation~\eqref{eq:denoiser:estimator} decreases as $t$ approaches 1. However, the denoiser estimation error in~\eqref{eq:lemma:velocity:estimation:0:0} is amplified by a factor of $(1-t)^{-2}$ when bounding the velocity field estimation error in~\eqref{eq:lemma:velocity:estimation:0}, which diverges as $t\to 1$.

\begin{lemma}[Error bounds of stable velocity estimation]\label{lemma:velocity:estimation:stable}
Let Assumption~\ref{assumption:Gaussian:convolution} be fulfilled. For $t\in(T^{*},1)$ with $T^{*}$ in~\eqref{eq:T}, let $\what{U}_{\mathrm{stab}}(t,\cdot)$ be the velocity estimator defined in~\eqref{eq:stable:velocity:estimator}, and let $\widehat{F}(t,\cdot)$ defined as~\eqref{eq:stable:F:estimator}. Then, for any $\delta\in(0,1)$ and every $x_{t}\in\bbR^{d}$, we have 
\begin{equation}
\bbE\Big[\|F(t,x_{t})-\what{F}(t,x_{t})\|_{2}^{2}\Big]
\lesssim (1-t)^{2}\delta^{2},
\end{equation}
and, as a consequence, 
\begin{equation}\label{eq:stable:velocity:error}
\bbE\Big[\|u(t,x_{t})-\what{U}_{\mathrm{stab}}(t,x_{t})\|_{2}^{2}\Big]
\lesssim
\frac{(1-t)^{2}}{t^{4}}\delta^{2},
\end{equation}
provided that, with $\beta_t$ in~\eqref{beta_t}, the following relations are fulfilled 
\begin{align*}
&n=\Theta\Big(\frac{d(1-t)^{2}(\sigma^{2}t^{2}+(1-t)^{2})+\sigma^{2}t^{4}R^{2}}{\sigma^{2}(\sigma^{2}t^{2}+(1-t)^{2})^{2}\delta^{2}}\Big), \quad \eta=\Theta\Big(\frac{\sigma^{8}\beta_{t}^{2}\delta^{2}}{d(\sigma^{4}\beta_{t}+R^{2})^{2}}\Big(\frac{\sigma^{4}}{\sigma^{2}+R^{2}}\Big)^2\Big), \\
&K=\Theta\Big(\frac{d}{\beta_{t}\delta^{2}}\Big(\frac{\sigma^{2}+R^{2}}{\sigma^{4}}\Big)^{2}\Big(1+\frac{R^{2}}{\sigma^{4}\beta_{t}}\Big)^{2}\log\Big(\Big(\frac{\sigma^{2}+R^{2}}{\sigma^{4}}\Big)^{2}\frac{W_{2}^{2}(p_{Z_{0}},p_{X_{1}|X_{t}=x_{t}})}{\delta^{2}}\vee 1\Big)\Big).
\end{align*}
\end{lemma}

\par The proof of Lemma~\ref{lemma:velocity:estimation:stable} is deferred to Appendix~\ref{section:proof:convergence:velocity}. In contrast to the velocity estimator in~\eqref{eq:velocity:estimator}, which represents the conditional expectation of $X_{1}$ given $X_{t}=x_{t}$, the stable velocity estimator in~\eqref{eq:stable:velocity:estimator} represents the conditional expectation of $\nabla\log p_{X_{1}}(X_{1})$ given $X_{t}=x_{t}$. Consequently, the convergence rate of the stable velocity estimator~\eqref{eq:stable:velocity:estimator} in Lemma~\ref{lemma:velocity:estimation:stable} depends on both the Lipschitz constant of the target score $\nabla\log p_{X_{1}}$ and the conditional covariance of $\nabla\log p_{X_{1}}(X_{1})$, whereas the convergence rate of the velocity estimator~\eqref{eq:velocity:estimator} provided in Lemma~\ref{lemma:velocity:estimation} does not. More importantly, the error of the stable estimator in~\eqref{eq:stable:velocity:error} does not diverge as $t\to 1$. Although this error diverges as $t\to 0$, this behavior is inconsequential for our purposes, as we are exclusively interested in the regime $t\in(T^{*},1)$, where $T^{*}$ is strictly away from zero.

\begin{remark}[Condition number of Langevin Monte Carlo]\label{remark:condition:number}
As demonstrated in Lemmas~\ref{lemma:velocity:estimation} and~\ref{lemma:velocity:estimation:stable}, the required number of iterations $K$ of Langevin Monte Carlo is governed by the condition number of the denoising distribution:
\begin{equation*}
\mathrm{cond}(p_{X_{1}|X_{t}=x_{t}}) \coloneq \frac{\lambda_{\max}(-\nabla^{2}\log p_{X_{1}|X_{t}=x_{t}})}{\lambda_{\min}(-\nabla^{2}\log p_{X_{1}|X_{t}=x_{t}})}\leq 1+\frac{R^{2}}{\sigma^{4}\beta_{t}^{2}},
\end{equation*}
where $\beta_{t}$ is defined in Proposition~\ref{lemma:denoising:langevin:convergence}. This observation has been shown by, e.g.,~\cite{Dalalyan2016Theoretical,Durmus2017Nonasymptotic,Vempala2019Rapid,CELSZ2025,mou2022Improved}. Although the convergence of the Langevin Monte Carlo and the resulting velocity estimators~\eqref{eq:velocity:estimator} and~\eqref{eq:stable:velocity:estimator} is theoretically guaranteed, a large condition number necessitates an excessively high number of iterations, resulting in slow convergence. To mitigate this ill-conditioning, we employ a preconditioned Langevin sampler to our framework in Section~\ref{section:preconditioning}, which demonstrates clear advantages in our numerical experiments as shown in Section~\ref{section:experiments}. A formal theoretical analysis of this preconditioning scheme is left for future work.
\end{remark}

\subsection{Error bounds for flow initialization}

Next, we give error bounds for flow initialization in Section~\ref{section:method:initialization}.

\begin{lemma}[Error bounds of flow initialization]\label{lemma:error:initialization}
Let Assumption~\ref{assumption:Gaussian:convolution} be fulfilled. Let $\what{X}_{T_{0}}\coloneq\what{U}_{L\tau}$ be the terminal state of the Langevin Monte Carlo in~\eqref{eq:langevin:monte:carlo:warmstart}. Then, for any $\kappa \in(0,1)$, we have 
\begin{equation}\label{eq:lemma:error:initialization:0}
\bbE\Big[W_{2}^{2}\Big(p_{\what{X}_{T_{0}}},p_{X_{T_{0}}}\Big)\Big] \lesssim \kappa^{2},    
\end{equation}
provided that the error of the velocity estimation $\delta^{2}$ in~\eqref{eq:lemma:velocity:estimation:0}, the step size $\tau$, and the number of steps $L$ fulfill
\begin{align*}
&\delta^{2}=\Theta\Big(\exp\Big\{-\frac{16T_0^{2}R^{2}}{T_0\sigma^2 + (1-T_0)^2}\Big\}\frac{(1-T_{0})^{4}}{T_{0}^{2}}\kappa^{2}\Big), \\
&\tau=\Theta\Big(\exp\Big\{-\frac{16T_0^{2}R^{2}}{T_0\sigma^2 + (1-T_0)^2}\Big\}\frac{(1-T_{0})^{2}}{d(G+1)^{2}}\kappa^{2}\Big), \\
&L=\Theta\Big(\exp\Big\{\frac{24T_0^{2}R^{2}}{T_0\sigma^2 + (1-T_0)^2}\Big\}\frac{d(G+1)^{2}}{(1-T_{0})^{2}\kappa^{2}}\Big(\frac{8T_0^{2}R^{2}}{T_0^2\sigma^2+(1-T_0)^2}+\log\frac{\kl(p_{U_{0}},p_{X_{T_{0}}})}{\kappa^{2}}\Big)\vee 0\Big).
\end{align*}
Here $G$ is a constant from Theorem~\ref{thm} only depending on $R$ and $\sigma$, the expectation in~\eqref{eq:lemma:error:initialization:0} is taken with respect to particles of Monte Carlo approximation~\eqref{eq:velocity:estimator} to the velocity field.
\end{lemma}

\par The proof of Lemma~\ref{lemma:error:initialization} is given in Appendix~\ref{section:proof:convergence:initialization}. The error bounds in Lemma~\ref{lemma:error:initialization} provide guidance for choosing the accuracy tolerance $\delta$ of the velocity estimation, the step size $\tau$, and the number of iterations $L$. Smaller values of $T_{0}$ lead to a smaller log-Sobolev inequality constant $C_{\LSI}(p_{X_{T_{0}}})$; consequently, the Langevin algorithm used for initialization requires fewer iterations $L$ and allows for a larger step size $\tau$. Furthermore, smaller values of $T_{0}$ permit a larger accuracy tolerance $\delta$ in the velocity estimation.

\subsection{Main results} \label{sec:main}
\par Based on Lemmas~\ref{lemma:velocity:estimation} and~\ref{lemma:error:initialization}, the following theorem presents a full error analysis for our sampling method. The proof is provided in Appendix~\ref{section:error:ode}.

\begin{theorem}[Flow ODEs with vanilla velocity estimation]\label{theorem:error:ode}
Let Assumption~\ref{assumption:Gaussian:convolution} be fulfilled. Let $\what{\psi}$ be the discrete-time flow map defined by the exponential integrator in~\eqref{eq:PFODE:velocity:ei}. Assume $T^{*}<T_{0}<T_{\rm end}<1$. Then we have 
\begin{align*}
&\bbE\Big[W_{2}^{2}\Big(p_{X_{1}},(T_{\rm end}^{-1}\what{\psi}(T_{\rm end},\cdot))_{\sharp}p_{\what{X}_{T_{0}}}\Big)\Big]  \\
&\lesssim \underbrace{\frac{d(1-T_{\rm end})^{2}}{T_{\rm end}^{2}}}_{\text{early-stopping error}}
+ \underbrace{\frac{\exp(2G(T_{\rm end}-T_{0}))}{T_{\rm end}^{2}}\kappa^{2}}_{\text{initialization error}} \\
&\quad + \underbrace{\frac{\exp(2G(T_{\rm end}-T_{0}))}{T_{\rm end}^{2}}\frac{H^{2}}{T_{0}^{4}}\frac{d+R^{2}+1}{(1-T_{\rm end})^{2}}h^{2}}_{\text{discretization error}} + \underbrace{\frac{\exp(2G(T_{\rm end}-T_{0}))}{T_{\rm end}^{2}}\frac{(T_{\rm end}-T_0)^2}{(1-T_0)(1-T_{\rm end})}\delta^{2}}_{\text{velocity estimation error}},
\end{align*}
where $G$ the a constant from Theorem~\ref{thm} only depending on $R$ and $\sigma$, the constant $H$ from Lemma~\ref{lemma:derivative:denoiser} only depends on $R$ and $\sigma$. The error of the velocity estimation $\delta$ is defined in Lemma~\ref{lemma:velocity:estimation}, and the error of flow initialization $\kappa$ is defined in Lemma~\ref{lemma:error:initialization}.
\end{theorem}

\par By similar arguments, we also establish the convergence rate for flow ODE with stable velocity estimation~\eqref{eq:PFODE:velocity:ei:stable}, the proof of which is provided in Appendix~\ref{section:error:ode}.

\begin{theorem}[Flow ODEs with stable velocity estimation]\label{theorem:error:ode:stable}
Let Assumption~\ref{assumption:Gaussian:convolution} be fulfilled. Let $\what{\psi}_{\rm stab}$ be the discrete-time flow map defined by the exponential integrator in~\eqref{eq:PFODE:velocity:ei:stable}. Assume $T^{*}<T_{0}<T_{\rm end}<1$. Then we have 
\begin{align*}
&\bbE\Big[W_{2}^{2}\Big(p_{X_{1}},(T_{\rm end}^{-1}\what{\psi}_{\rm stab}(T_{\rm end},\cdot))_{\sharp}p_{\what{X}_{T_{0}}}\Big)\Big]  \\
&\lesssim \underbrace{\frac{d(1-T_{\rm end})^{2}}{T_{\rm end}^{2}}}_{\text{early-stopping error}} \!
+ \underbrace{\exp\Big(2\Big(G+\frac{2}{T_{0}}\Big)(T_{\rm end}-T_{0})\Big)\frac{\kappa^{2}}{T_{\rm end}^{2}}}_{\text{initialization error}} \\
&\quad + \underbrace{\exp\Big(2\Big(G+\frac{2}{T_{0}}\Big)(T_{\rm end}-T_{0})\Big)\frac{V^{2}}{T_{0}^{6}}\frac{d+R^{2}+1}{(1-T_{\rm end})^{2}}h^{2}}_{\text{discretization error}} \\ 
&\quad + \underbrace{\exp\Big(2\Big(G+\frac{2}{T_{0}}\Big)(T_{\rm end}-T_{0})\Big)\frac{(T_{\rm end}-T_{0})^{2}}{T_{\rm end}^{2}T_{0}^{4}}\delta^{2}}_{\text{velocity estimation error}},
\end{align*}
where $G$ the a constant from Theorem~\ref{thm} only depending on $R$ and $\sigma$, the constant $V$ from Lemma~\ref{lemma:derivative:stable} only depends on $R$ and $\sigma$. The error of the velocity estimation $\delta$ is defined in Lemma~\ref{lemma:velocity:estimation:stable}, and the error of flow initialization $\kappa$ is defined in Lemma~\ref{lemma:error:initialization}.
\end{theorem}

\par Both Theorems~\ref{theorem:error:ode} and~\ref{theorem:error:ode:stable} decompose the total error of the discrete-time flow $\what{\psi}(T_{\rm end},\cdot)$ into four components: (i) the early-stopping error, (ii) the initialization error, (iii) the discretization error, and (iv) the velocity estimation error. The early-stopping error quantifies the discrepancy between the terminal density $p_{X_{T_{\rm end}}}$ and the target density $p_{X_{1}}$. The discretization error arises from the exponential integrator discretization of the probability flow ODE. While the initialization error and the velocity estimation error are analyzed in Lemmas~\ref{lemma:error:initialization} and~\ref{lemma:velocity:estimation}, respectively, Theorems~\ref{theorem:error:ode} and~\ref{theorem:error:ode:stable} further elucidate how these errors are amplified during the numerical simulation of the probability flow ODE. Notably, both the discretization error and the velocity estimation error diverge as the terminal time $T_{\rm end}$ approaches one, highlighting the necessity of early-stopping.

\par The authors of~\cite{ding2024characteristic} provide an error analysis for simulating the probability flow ODE induced by stochastic interpolants. However, their discretization error analysis relies on the boundedness and Lipschitz continuity of the estimated velocity field. In this work, the velocity is parameterized by a deep neural network, for which boundedness and Lipschitz continuity can be enforced during training; see, for example,~\cite{Drucker1991Double,Drucker1992Improving,ding2025semi,benton2024error,bansal2026convergence}. In contrast, ensuring Lipschitz continuity for a Monte Carlo-based velocity estimator is generally intractable. In contrast, Theorems~\ref{theorem:error:ode} and~\ref{theorem:error:ode:stable} derive convergence rates without requiring any regularity assumptions on the velocity estimator.

\begin{remark}\label{remark:sde:ode}
Theorems~\ref{theorem:error:ode} and~\ref{theorem:error:ode:stable} highlight the theoretical advantages of sampling via stochastic interpolants over SDE-based sampling methods that rely on Monte Carlo score estimation~\cite{huang2024reverse,he2024zeroth,Grenioux2024Stochastic}. In particular, Theorems~\ref{theorem:error:ode} and~\ref{theorem:error:ode:stable} establish that the discretization error of the Euler method applied to the probability flow ODE converges at rate of $\calO(h)$, whereas the analogue discretization of SDEs achieves only a strong convergence rate of $\calO(h^{1/2})$; see~\cite[Theorem~10.6.3]{Kloeden1992Numerical}. Consequently, to attain the same accuracy in Wasserstein-2 distance, SDE-based methods require a step size on the order of the square root of that used for ODE-based methods, resulting in a substantially larger number of discretization steps. Since evaluating either the velocity field or the score function via Monte Carlo approximation is computationally expensive, this increase in the number of time steps leads to a significant growth in overall computational cost.
\end{remark}

%% file: precondition.tex

\section{Preconditioned Langevin Algorithm}
\label{section:preconditioning}

\par Although we have established the convergence of the Langevin Monte Carlo for both velocity estimation in Lemmas~\ref{lemma:velocity:estimation} and~\ref{lemma:velocity:estimation:stable}, and for flow initialization in Lemma~\eqref{lemma:error:initialization}, the vanilla Langevin Monte Carlo suffers from slow convergence when the condition number is excessively large, as noted in Remark~\ref{remark:condition:number}. To address this, the current section introduces a preconditioning scheme into the Langevin algorithms for both tasks. We demonstrate the numerical advantages of this preconditioning in Section~\ref{section:experiments}, leaving a rigorous theoretical analysis for future work.

\par Before proceeding, we recall the Langevin diffusion with invariant distribution $\pi\in C^{2}(\bbR^{d})$:
\begin{equation}\label{eq:langevin}
\d\widetilde{\theta}_{t}=\nabla\log\pi(\widetilde{\theta}_{t})\dt+\sqrt{2}\d B_{t},\quad t\geq 0.
\end{equation}
The preconditioned Langevin diffusion is defined as 
\begin{equation}\label{eq:langevin:preconditioned}
\d\theta_{t}=\underbrace{P(\theta_{t})\nabla\log\pi(\theta_{t})\dt}_{\text{preconditioned score}}+\underbrace{\Div P(\theta_{t})\dt}_{\text{correction}}+\underbrace{\sqrt{2 P(\theta_{t})}\d B_{t}}_{\text{preconditioned noise}}, \quad t\geq 0,
\end{equation}
where $P(\theta_{t})\succeq 0$ is a probably state-dependent preconditioner, and $\Div$ denotes the row-wise divergence of a matrix-valued function, that is, $(\Div P(\theta_{t}))_{k}=\sum_{\ell=1}^{d}\partial_{\ell}P_{k\ell}(\theta_{t})$. If $P(\theta)\equiv P$ is a state-independent preconditioner, then the correction term vanishes. In particular, if $P(\theta)\equiv I_{d}$, the preconditioned Langevin diffusion degenerates to the vanilla Langevin diffusion~\eqref{eq:langevin}.

\begin{remark}[Comparison with preconditioned optimization]
In preconditioned gradient flow~\cite{Amari1998Natural,Martens2010deep,Duchi2011Adaptive,Kingma2015adam,dahmen2025expansive,cayci2025riemannian}, the preconditioner modifies only the gradient of the objective function. In contrast, the preconditioned Langevin diffusion~\eqref{eq:langevin:preconditioned} alters not only the gradient (drift) term but also the noise term, and additionally introduces a correction term. 
\end{remark}

This correction term in~\eqref{eq:langevin:preconditioned} is essential to ensure that the invariant distribution of the preconditioned Langevin diffusion coincides with that of the vanilla Langevin diffusion~\eqref{eq:langevin}, as established in the following proposition, see~\cite[Theorem 1]{ma2015complete}.

\begin{proposition}[Invariant distribution of preconditioned Langevin diffusion]\label{lemma:invariant:preconditioned:LD}
Suppose that $\pi\in C^{2}(\bbR^{d})$. Assume further $P(\theta)\succeq 0$ for each $\theta\in\bbR^{d}$. Then the invariant distribution of the preconditioned Langevin diffusion~\eqref{eq:langevin:preconditioned} is the same as that of the vanilla Langevin diffusion~\eqref{eq:langevin}.
\end{proposition}

\par In this work, we adopt the RMSprop-based preconditioning strategy proposed by~\cite{Li2016Preconditioned} for Langevin diffusion. This approach is Hessian-free and circumvents the computational burden of matrix inversion. The diagonal preconditioner $P(\theta_{t})$ is defined as
\begin{align*}
P(\theta_{t})&=\diag\Big(\frac{1}{\sqrt{v(\theta_{t})}+\epsilon}\Big), \\
\frac{\d v(\theta_{t})}{\dt}&=(1-\alpha)(\nabla\log\pi(\theta_{t})\odot\nabla\log\pi(\theta_{t})-v(\theta_{t})), 
\end{align*}
where the second equality is known as accumulator dynamics, which describes an exponential moving average of squared score in continuous time. Here $\epsilon>0$ is the precision tolerance for numerical stabiltiy, $\alpha\in(0,1)$ is the smoothing parameter for exponential moving average, and the notation $\odot$ denotes the entry-wise multiplication of two vectors.

\begin{remark}
\cite[Corollary 2]{Li2016Preconditioned} shows that, when RMSprop-based preconditioning is used, the effect of the correction term $\Div P(\theta_{t})$ in~\eqref{eq:langevin:preconditioned} becomes negligible provided that the smoothing parameter $\alpha$ is close to $1$. Consequently, this term is typically omitted in practical implementations.
\end{remark}

\par RMSprop-based preconditioning is effective for both velocity estimation and initialization for the following reasons:
\begin{enumerate}[label=(\roman*)]
\item \textbf{Velocity estimation.} In Langevin Monte Carlo used for velocity estimation, the target distribution is log-concave. An RMSprop-based preconditioner can reduce the effective condition number by approximating the diagonal of the inverse Hessian~\cite{Dauphin2015Equilibrated,Li2016Preconditioned}. Intuitively, RMSprop-based preconditioning induces adaptive step sizes that reflect the local geometry of the target distribution: it assigns smaller step sizes to steep directions (characterized by large squared gradients) and larger step sizes to flatter directions (characterized by small squared gradients). This adaptive behavior is more efficient than using a fixed step size in standard Langevin dynamics~\cite{Leroy2024Adaptive}.
\item \textbf{Initialization.} During the initialization phase, although the starting distribution $p_{X_{T_{0}}}$ is generally more tractable than the final target, it may typically exhibit multi-modality. In this context, RMSprop-based preconditioning is crucial for facilitating transitions between distinct modes. The regions separating these modes (energy barriers) are often characterized by saddle points or plateaus where the energy landscape is flat, resulting in vanishing gradients~\cite{Dauphin2014Identifying}. Under vanilla Langevin diffusion, the sampler struggles to traverse these low-gradient regions due to the coupling of the step size with the gradient magnitude. In contrast, preconditioned Langevin diffusion rescales the dynamics by the inverse of the gradient magnitude; this effectively amplifies the step size and stochastic noise in flat directions, thereby enhancing exploration and significantly improving the capability to escape saddle points~\cite{Dauphin2014Identifying,Li2016Preconditioned,zhou2024on,Zuo2025Gradient}.
\end{enumerate}

\par The preconditioned counterparts of velocity estimation (Algorithm~\ref{alg:velocity:estimation}) and initialization (Algorithm~\ref{alg:initialization}) are provided in Appendix~\ref{section:precondition:alg}.

%% file: experiments.tex

\section{Numerical Experiments}
\label{section:experiments}

\par In this work, we conduct a series of numerical experiments to demonstrate the efficiency of our proposed sampling via stochastic interpolants (SSI). In Section~\ref{section:experiments:two:dim}, we investigate two-dimensional examples, while Section~\ref{section:experiments:high:dim} extends this to high-dimensional settings. In Section~\ref{Bayesian:inferrence:gmm}, we explore applications in Bayesian inference. For all of these experiments, an initialization time $T_0$ has to be chosen. From Proposition~\ref{lemma:denoising:langevin:convergence} we can deduce suitable values of $T_0$ if the target distribution satisfies Assumption~\ref{assumption:Gaussian:convolution} and we are given $R$ and $\sigma$. In practice, however, the target distribution may not satisfy Assumption~\ref{assumption:Gaussian:convolution}, or $R$ and $\sigma$ are not both available. We therefore present ablation studies in Section~\ref{section:experiments:ablation} that quantify the influence of $T_{0}$; these also demonstrate the benefits of the preconditioning introduced in Section~\ref{section:preconditioning}. The computational cost is shown in Appendix~\ref{appendix:exp:cost}. Throughout this section, SSI~w/o~precond. means~\eqref{eq:PFODE:velocity:ei} using Langevin Monte Carlo without preconditioning, SSI means~\eqref{eq:PFODE:velocity:ei} using Langevin Monte Carlo with preconditioning, and Stab.~SSI means~\eqref{eq:PFODE:velocity:ei:stable} using Langevin Monte Carlo without preconditioning for \texttt{rings} and with preconditioning for other distributions.

\subsection{Two-dimensional distributions}
\label{section:experiments:two:dim}

\par This subsection compares SSI with different baselines in two-dimensional distributions. The proposed SSI demonstrates superior performance in capturing complex, multi-modal distributions, consistently achieving the best results in distributional metrics.

\paragraph{Baselines} 
We benchmark our method against several well-established algorithms: unadjusted Langevin algorithm (ULA), Metropolis-adjusted Langevin algorithm (MALA), preconditioned ULA (pULA) in Section~\ref{section:preconditioning}, Hamiltonian Monte Carlo (HMC)~\cite{HG2014}, parallel tempering (PT), and non-reversible parallel tempering method DEO-PT~\cite{Syed2022NonReversible}. The implementation of the parallel tempering (PT) is based on~\cite{Chen2024Diffusive}. For all baselines, we employ initialization from either a standard Gaussian distribution or the origin point. Detailed hyperparameters for these methods are summarized in Appendix~\ref{section:exp:details}.

\paragraph{Target distributions}
We evaluate our method and baselines on three challenging two-dimensional distributions: the \texttt{rings} distribution~\cite{Grenioux2024Stochastic}, a mixture of $7\times7$ Gaussian grid (\texttt{MoG7x7}), and a random mixture of 40 Gaussians (\texttt{MoG40})~\cite{midgley2023flow}. These distributions are visualized in the left-top panel of Figures~\ref{fig:example2d:rings},~\ref{fig:example2d:MoG7x7}, and~\ref{fig:example2d:mog40}, respectively.

\paragraph{Experimental results}

\par The quantitative results for the different sampling methods across various metrics are summarized in Table~\ref{tab:exp:results}. Here, the reference metrics are computed using the same number of samples from the target distribution. To provide a visual comparison, the sampling results for the \texttt{rings}, \texttt{MoG7x7}, and \texttt{MoG40} distributions are illustrated in Figures~\ref{fig:example2d:rings},~\ref{fig:example2d:MoG7x7}, and~\ref{fig:example2d:mog40}, respectively.

\begin{table}[htbp]
\centering\small
\setlength{\tabcolsep}{4pt}
\caption{Comparison of negative log likelihood (NLL), maximum mean discrepancy (MMD), and Wasserstein-2 distance ($W_{2}$) for different sampling methods.}
\label{tab:exp:results}
\begin{tabular}{lrcccccccc}
\toprule
 & Metric & ULA & MALA & pULA & HMC & PT \\
\midrule
\texttt{rings}
& NLL $\downarrow$
& $1.42\times10^{2}$ & $3.83\times10^{1}$
& $1.06\times10^{1}$ & $5.66$ & $5.53$ \\
& MMD $\downarrow$
& $2.06\times10^{-3}$ & $5.14\times10^{-2}$
& $2.56\times10^{-2}$ & $5.17\times10^{-3}$
& $5.89\times10^{-2}$ \\
& $W_{2}$ $\downarrow$
& $1.77$ & $5.44$ & $4.81$ & $1.79$ & $3.53$ \\
\midrule
\texttt{MoG7x7}
& NLL $\downarrow$
& $5.58$ & $4.86$ & $5.44$ & $8.03$ & $5.34$ \\
& MMD $\downarrow$
& $2.23$ & $2.24$ & $2.24$ & $2.23$
& $2.91\times10^{-1}$ \\
& $W_{2}$ $\downarrow$
& $2.58\times10^{1}$ & $2.61\times10^{1}$ & $9.18$
& $2.53\times10^{1}$ & $1.03\times10^{1}$ \\
\midrule
\texttt{MoG40}
& NLL $\downarrow$
& $7.10$ & $6.54$ & $7.15$ & $6.94$ & $6.95$ \\
& MMD $\downarrow$
& $7.40\times10^{-1}$ & $1.95$ & $5.40\times10^{-1}$
& $3.46\times10^{-1}$ & $2.08\times10^{-1}$ \\
& $W_{2}$ $\downarrow$
& $1.86\times10^{1}$ & $2.70\times10^{1}$
& $1.65\times10^{1}$ & $1.33\times10^{1}$
& $1.06\times10^{1}$ \\
\midrule
 & Metric & DEO-PT &\textbf{SSI w/o precond.} & \textbf{SSI}
 & \textbf{Stab. SSI} & Reference \\
\midrule
\texttt{rings}
& NLL $\downarrow$ & 5.72
& $1.78\times10^{1}$ & $1.88\times10^{1}$ & {5.83} & $5.74$ \\
& MMD $\downarrow$ & $2.60\times 10^{-3}$
& $3.97\times10^{-3}$ & $1.70\times10^{-2}$
& {$3.97\times 10^{-3}$} & $3.92\times10^{-4}$ \\
& $W_{2}$ $\downarrow$ & 1.43
& $1.38$ & $2.20$ & {2.34} & $0.58$ \\
\midrule
\texttt{MoG7x7}
& NLL $\downarrow$ & 5.38
& $7.18$ & $7.42$ & $2.12\times10^{1}$ & $5.35$ \\
& MMD $\downarrow$ & $3.97\times 10^{-3}$
& $4.08\times10^{-2}$ & $6.88\times10^{-3}$
& $1.29\times10^{-3}$ & $4.98\times10^{-4}$ \\
& $W_{2}$ $\downarrow$ & 2.34
& $4.05$ & $1.94$ & $2.52$ & $0.71$ \\
\midrule
\texttt{MoG40}
& NLL $\downarrow$ & 6.87
& $7.49$ & $7.54$ & $1.27\times10^{1}$ & $6.86$ \\
& MMD $\downarrow$ & $6.30\times 10^{-3}$
& $2.86\times10^{-2}$ & $2.26\times10^{-2}$
& $1.35\times10^{-2}$ & $3.92\times10^{-4}$ \\
& $W_{2}$ $\downarrow$ & 2.35
& $3.94$ & $4.00$ & $4.33$ & $0.79$ \\
\bottomrule
\end{tabular}
\end{table}

\par Table~\ref{tab:exp:results} quantitatively compares SSI and its variants against baseline methods across four benchmark distributions. Our methods and DEO-PT consistently outperform the standard MCMC baselines
on  \texttt{MoG7x7} and \texttt{MoG40}, particularly in terms of MMD and $W_2$
distance. On \texttt{MoG7x7}, stable SSI~\eqref{eq:PFODE:velocity:ei:stable} achieves the lowest MMD, while vanilla SSI~\eqref{eq:PFODE:velocity:ei}
attains the lowest $W_2$ distance. On \texttt{MoG40}, DEO-PT achieves the best MMD and $W_2$ distance, with the SSI variants remaining substantially more accurate than the other baselines. These results indicate that both approaches capture the multimodal structure of the targets more faithfully and are substantially less affected by the energy barriers between modes. Furthermore, in most cases, stable SSI~\eqref{eq:PFODE:velocity:ei:stable} outperforms vanilla SSI~\eqref{eq:PFODE:velocity:ei} in MMD, demonstrating the benefit of stabilization.

\begin{remark}[Advantages of preconditioning]
Comparing pULA with the vanilla ULA highlights the benefits of the preconditioning strategy employed in pULA. pULA consistently outperforms ULA in the multi-modal Mixture of Gaussians experiments: \texttt{MoG7x7} and \texttt{MoG40}. This suggests that the RMSprop preconditioner helps the sampler traverse the low-density regions between modes more effectively than ULA. The same effect appears within our own method: on \texttt{MoG7x7}, preconditioned SSI significantly outperforms its unpreconditioned counterpart in both MMD and $W_{2}$, providing more direct evidence for the value of preconditioning.
\end{remark}

\begin{remark}[Limitation of NLL]
MALA attains the lowest NLL on \texttt{MoG7x7} and \texttt{MoG40}, but this reflects a limitation of the NLL metric rather than superior sampling quality. First, the NLL only measures whether the particles concentrate in high-probability regions, and not whether they cover the target. A sampler that places all of its particles at the global maximizers of the target density therefore attains an NLL even lower than that of the ground-truth distribution, despite being a poor approximation to it; a smaller NLL is thus symptomatic of over-concentration rather than of accurate sampling. Second, for an equally weighted mixture such as \texttt{MoG7x7} and \texttt{MoG40}, a sampler restricted to a subset of the modes attains the same NLL as one covering all of them, so the metric is blind to mode collapse by construction. Both failure modes are evident in Figures~\ref{fig:example2d:rings},~\ref{fig:example2d:MoG7x7}, and~\ref{fig:example2d:mog40}, and they are captured by MMD and $W_{2}$, on which MALA performs poorly.
\end{remark}

\begin{figure}[htbp]
\centering
\includegraphics[width=1.0\linewidth]{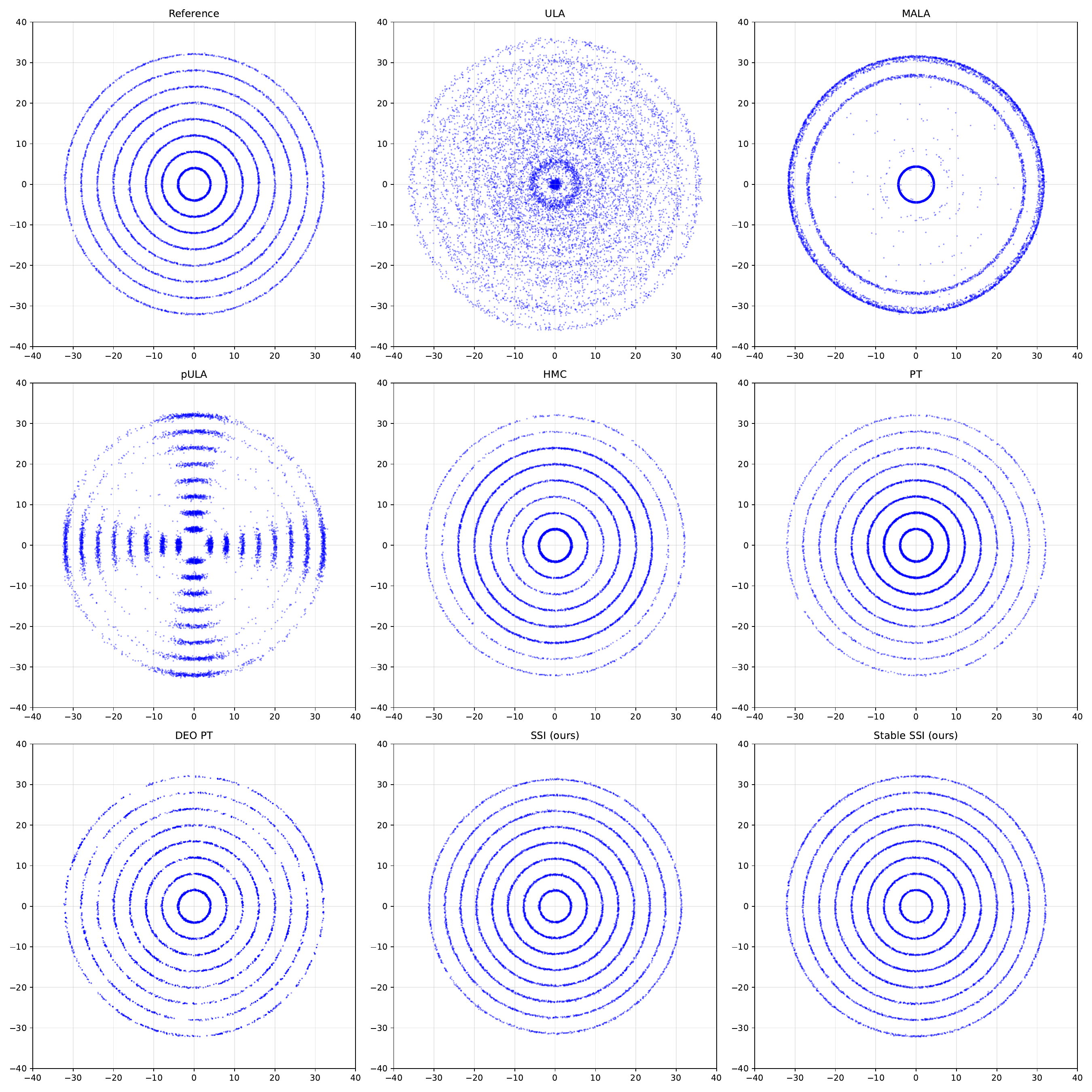}
\caption{Comparison of SSI against baseline methods for the \texttt{rings} distribution in Section~\ref{section:experiments:two:dim}. The blue points represent $10^{4}$ particles generated by each method.}
\label{fig:example2d:rings}
\end{figure}

\par The results for the \texttt{rings} distribution are presented in Figure~\ref{fig:example2d:rings}. To enable ULA to capture all rings of the target distribution, we employed a large step size to enhance its exploratory capability. However, this large step size prevents ULA from capturing the local detailed features of the target distribution. In contrast, due to its rejection procedure and adaptive step size, MALA successfully captures the local detailed structure; nevertheless, it is only able to identify three of the eight rings. Similarly, pULA exhibits an adaptive step size that allows it to resolve local geometry. As discussed in Section~\ref{section:preconditioning}, while preconditioning generally improves exploration, the anisotropy of the RMSprop preconditioner causes pULA to capture structure primarily along the axes. HMC and PT not only capture all eight rings but also adapt to the local geometry. However, by the construction of the \texttt{rings} distribution, each ring carries equal probability mass. Consequently, particles should appear sparser on the outer rings and denser on the inner rings. The particles sampled by HMC or PT do not adhere to this distribution property. In contrast, our proposed SSI not only accurately captures all rings but also correctly recovers the relative weights of the different rings. Empirical estimations of weights of rings generated by different sampling methods are shown in Table~\ref{tab:exp:rings:weight} in Appendix~\ref{appendix:exp:weights}.

\begin{figure}[htbp]
\centering
\includegraphics[width=1.0\linewidth]{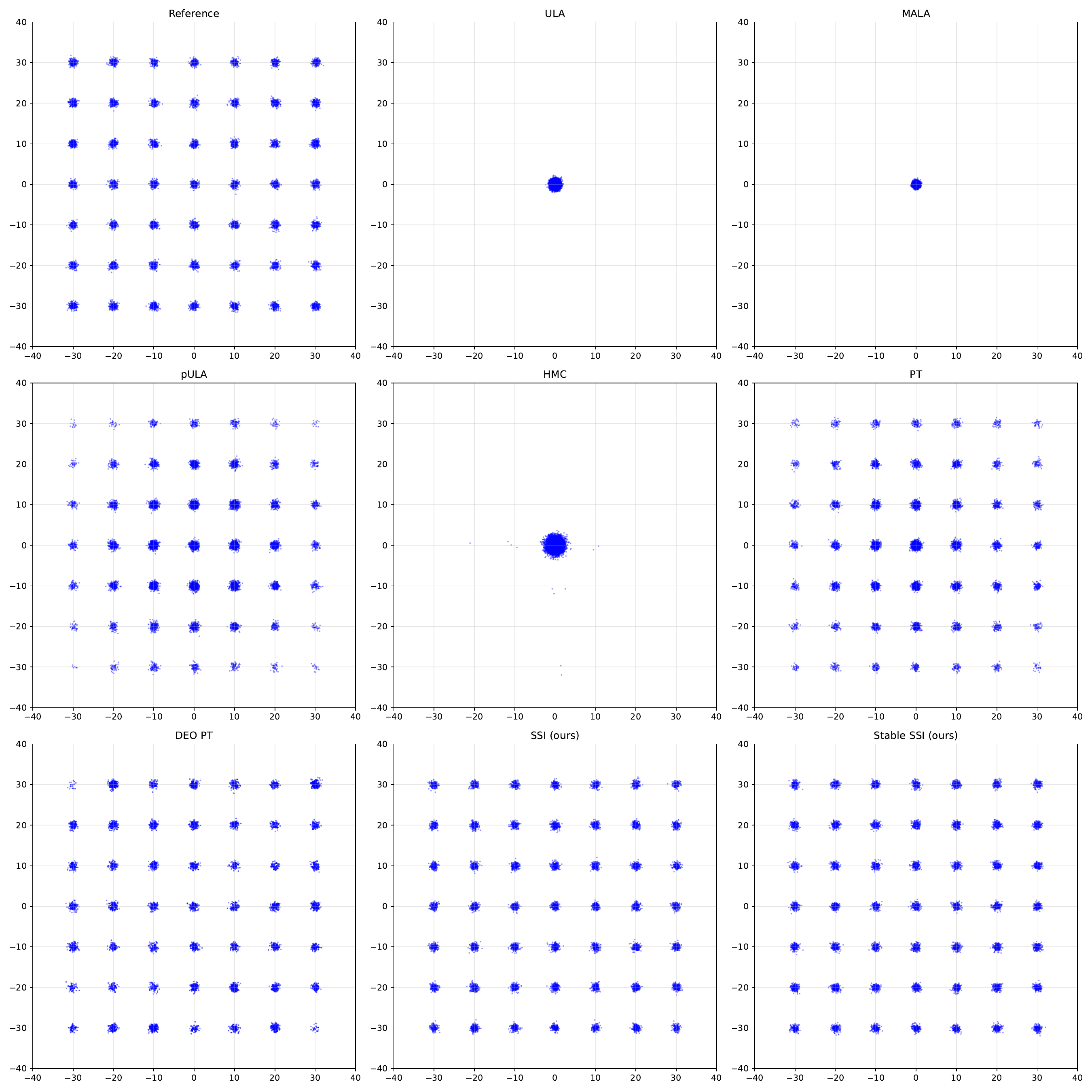}
\caption{Comparison of SSI against baseline methods for the \texttt{MoG7x7} distribution in Section~\ref{section:experiments:two:dim}. The blue points represent $10^{4}$ particles generated by each method.}
\label{fig:example2d:MoG7x7}
\end{figure}

\par The results for the \texttt{MoG7x7} distribution are displayed in Figure~\ref{fig:example2d:MoG7x7}. The particles sampled by ULA, MALA, and HMC remain concentrated around the initialization region. This limitation arises because these methods are unable to traverse the high energy barriers separating the modes. In the case of pULA, the RMSprop preconditioner induces larger step sizes and significant noise perturbations in the low-density regions between distinct modes, thereby facilitating exploration. However, both pULA and PT fail to correctly recover the relative weights of the different modes, as evidenced by the insufficient number of particles in the outer modes. In contrast, SSI not only captures all modes of the Gaussian mixture but also accurately recovers their relative weights.

\begin{figure}[htbp]
\centering
\includegraphics[width=1.0\linewidth]{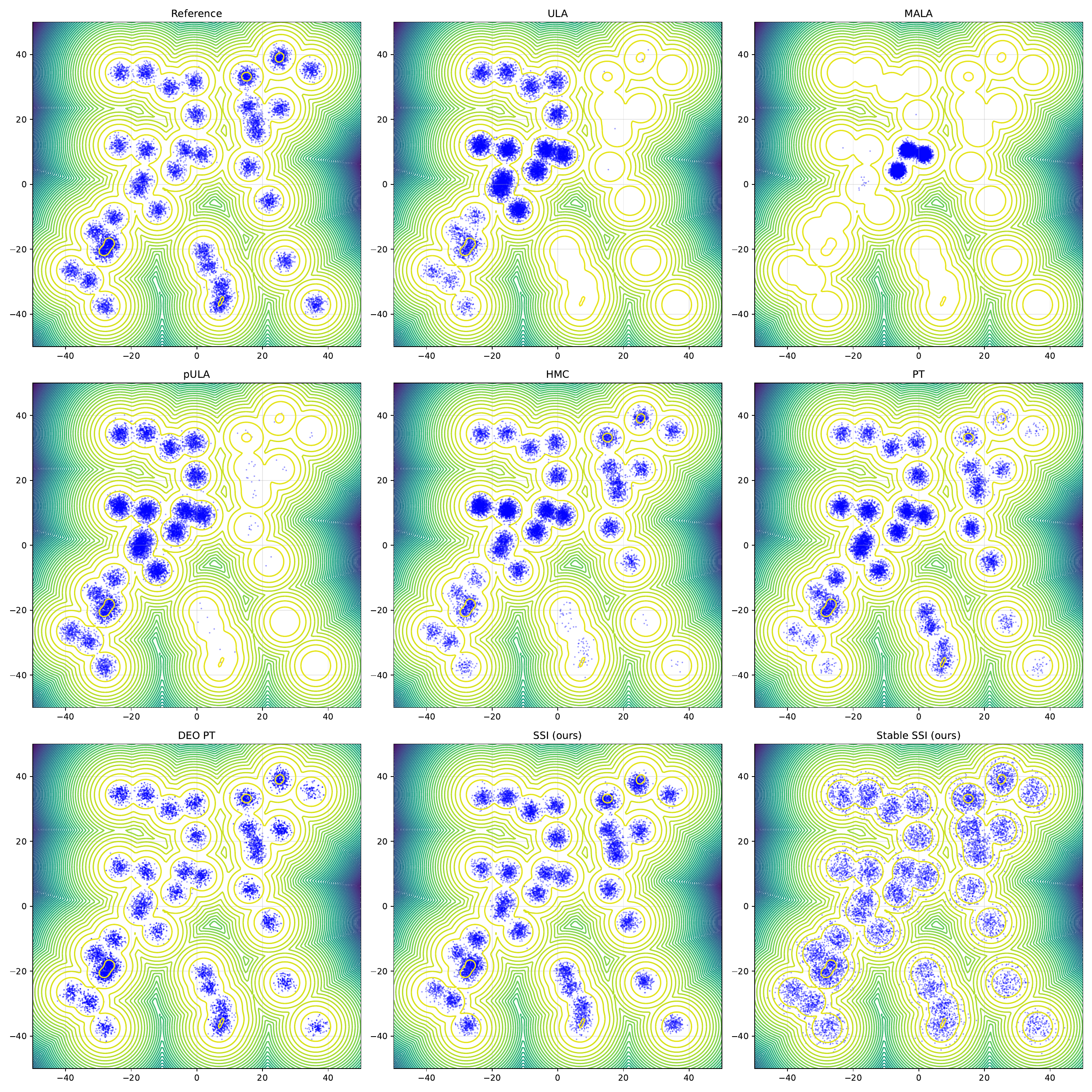}
\caption{Comparison of SSI against baseline methods for the \texttt{MoG40} distribution in Section~\ref{section:experiments:two:dim}. The blue points represent $10^{4}$ particles generated by each method.}
\label{fig:example2d:mog40}
\end{figure}

\par The results for the \texttt{MoG40} distribution are presented in Figure~\ref{fig:example2d:mog40}. Among the baseline methods, only PT successfully captures all 40 modes of the target distribution; however, it fails to recover their relative weights. In contrast, SSI not only captures every mode but also accurately preserves their respective weights.

\paragraph{Additional comparison with DEO-PT}
Table~\ref{tab:exp:results} shows that DEO-PT~\cite{Syed2022NonReversible} has comparable performances as our proposed SSI on \texttt{rings} and \texttt{MoG7x7}, while DEO-PT performs better than ours on \texttt{MoG40}. We give an example where DEO-PT fails while our SSI works. We consider the following two-dimensional Gaussian mixture target:
\begin{equation}
\label{eq:gmm:unequal}
p_{X_{1}}(x) \coloneqq \frac{1}{10}\mathcal{N}\left(x;(4,4),6.0^{2}I_{2}\right)+\frac{9}{10}\mathcal{N}\left(x;(-20,-20),0.2^{2}I_{2}\right).
\end{equation}
This distribution presents a particularly challenging sampling problem because its two components differ substantially in both mixing weight and scale. Specifically, the component with the larger probability mass is highly concentrated, whereas the component with the smaller probability mass is considerably more diffuse. Together with the large separation between the component means, this mismatch creates a narrow, high-density mode that is difficult to discover, as well as pronounced differences in the local geometry of the two modes.

The comparison between SSI and DEO-PT is presented in Figure~\ref{fig:example:pt}. The experimental settings are shown in Appendix~\ref{appendix:exp:pt}. Figure~\ref{fig:example:pt} shows DEO-PT predominantly explores the diffuse mode associated with
$\mathcal{N}\!\left((4,4),6.0^{2}I_{2}\right)$ and fails to adequately capture the narrow mode associated with
$\mathcal{N}\!\left((-20,-20),0.2^{2}I_{2}\right)$. In contrast, SSI successfully discovers and samples from both modes. Nevertheless, the empirical proportions of the samples assigned to the two modes do not yet accurately recover the ground-truth mixing weights.

\begin{figure}[htbp]
\centering
\includegraphics[width=\linewidth]{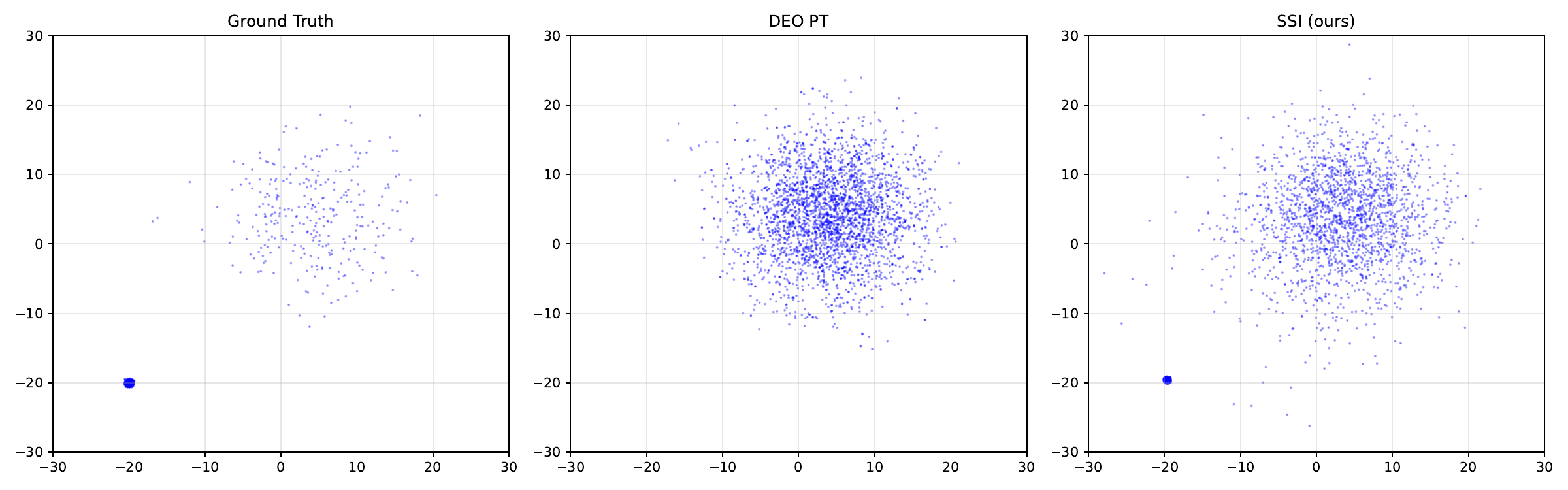}
\caption{Comparison of SSI and DEO-PT on the Gaussian mixture target in~\eqref{eq:gmm:unequal}. The blue points represent $3\times 10^{3}$ particles generated by each method. DEO-PT predominantly captures the diffuse mode, whereas SSI discovers both the diffuse and concentrated modes, although it does not perfectly recover their relative mixing weights.}
\label{fig:example:pt}
\end{figure}

\subsection{High-Dimensional Distributions}
\label{section:experiments:high:dim}

\par In this subsection, we demonstrate the efficacy of SSI on higher-dimensional distributions. We consider the eight-dimensional \texttt{Many Well} distribution~\cite{midgley2023flow} as a representative example, which is characterized by its highly multi-modal nature. Figure~\ref{fig:example:many:well} provides a visualization of these results, showing that SSI successfully captures all modes of the eight-dimensional \texttt{Many Well} distribution.

\begin{figure}[htbp]
\centering
\includegraphics[width=1.0\linewidth]{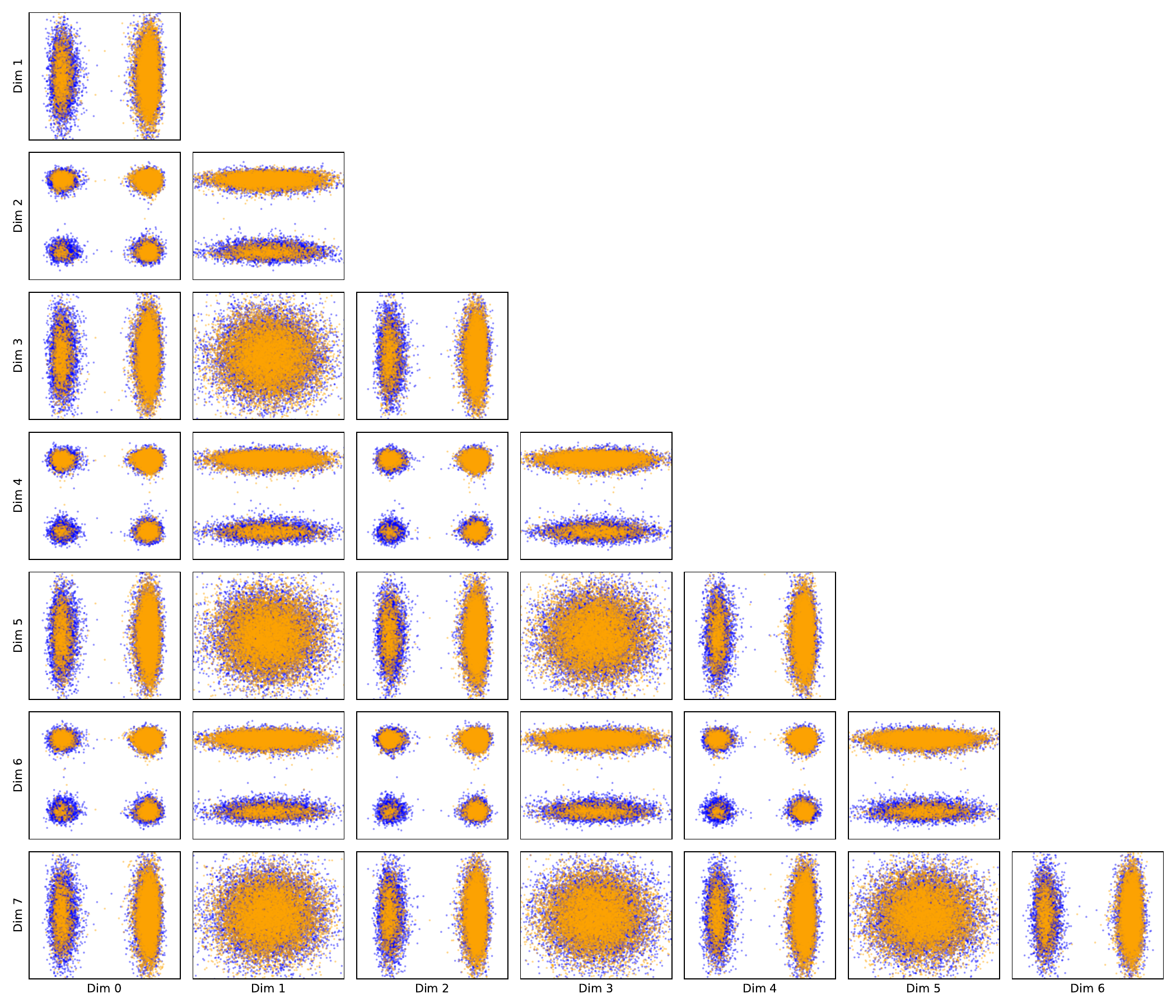}
\caption{Pairwise marginal distributions on $\bbR^{2}$ for the eight-dimensional \texttt{Many Well} distribution. The orange points represent samples from the ground truth distribution, while the blue points represent the particles generated by SSI.}
\label{fig:example:many:well}
\end{figure}

\subsection{Bayesian inference}
\label{Bayesian:inferrence:gmm}

\par In this subsection, we investigate the efficacy of SSI in the context of Bayesian inference~\cite{Lee2010Utility,Heng2021Gibbs,Durmus2019High}. We specifically consider the inference of cluster centers in a one-dimensional Gaussian mixture model~\cite{Lee2010Utility,Heng2021Gibbs}. Let $Y_{1:m}=\{Y_{i}\}_{i=1}^{m}$ denote a set of independent and identically distributed observations drawn from a one-dimensional Gaussian mixture with a known common variance $\sigma^{2}$:
\begin{equation*}
Y_{1},\ldots,Y_{m}\sim^{\iid}\frac{1}{K}\sum_{k=1}^{K}N(\theta_{k},\sigma^{2}),
\end{equation*}
where the vector of centers $\theta_{1:K}=(\theta_{1},\ldots,\theta_{K})^{\top}$ represents the unknown parameters to be estimated. By assuming a uniform prior over the parameter vector $\theta_{1:K}$, we obtain the following unnormalized posterior density given the measurements $Y_{1:m}$:
\begin{equation}\label{section:app:mixture:posterior}
p(\theta_{1:K}|Y_{1:m})\propto\underbrace{\prod_{i=1}^{m}\Big\{\frac{1}{K}\sum_{k=1}^{K}\gamma_{\sigma^{2}}(Y_{i}-\theta_{k})\Big\}}_{\text{likelihood}}\underbrace{\prod_{k=1}^{K}\unif_{[a,b]}(\theta_{k})}_{\text{prior}}.
\end{equation}
The goal of Bayesian inference is to sample from the posterior distribution defined in~\eqref{section:app:mixture:posterior}. 

\par As highlighted by~\cite[Section 5.1]{Heng2021Gibbs}, the posterior distribution in~\eqref{section:app:mixture:posterior} is invariant under any permutation of the parameter vector $\theta_{1:K}$. Specifically, for any permutation matrix $Q\in\bbR^{K\times K}$, the following holds:
\begin{equation*}
p(\theta_{1:K}|Y_{1:m})=p(Q\theta_{1:K}|Y_{1:m}), \quad \theta_{1:K}\in\bbR^{K}.
\end{equation*}
This permutation invariance implies that if the posterior $p(\cdot|Y_{1:m})$ possesses a mode centered at $\theta_{1:K}^{*}\in\bbR^{K}$, where each element $\theta_{k}^{*}$ is distinct (i.e., $\theta_{\ell}^{*}\neq\theta_{k}^{*}$ for all $\ell\neq k$), then there exist $(K\,!-1)$ additional distinct modes. These modes correspond to all possible permutations of $\theta_{1:K}^{*}$, depicted as red crosses in Figure~\ref{fig:bayesian:inference}. Such a multi-modal structure poses significant challenges for conventional sampling algorithms.

\par In this experiment, we consider a mixture of four one-dimensional Gaussians with centers located at $\{-3,0,3,6\}$. The prior is chosen as a uniform distribution on $(-10,10)^{4}$. Consequently, all permutations of $\theta_{1:4}^{*}=(-3,0,3,6)$ represent valid modes of the posterior. The particles generated by SSI are shown as blue points in Figure~\ref{fig:bayesian:inference}. These results demonstrate that SSI successfully captures all $4!=24$ modes of the posterior distribution~\eqref{section:app:mixture:posterior}.

\begin{figure}[htbp]
\centering
\includegraphics[width=0.95\linewidth]{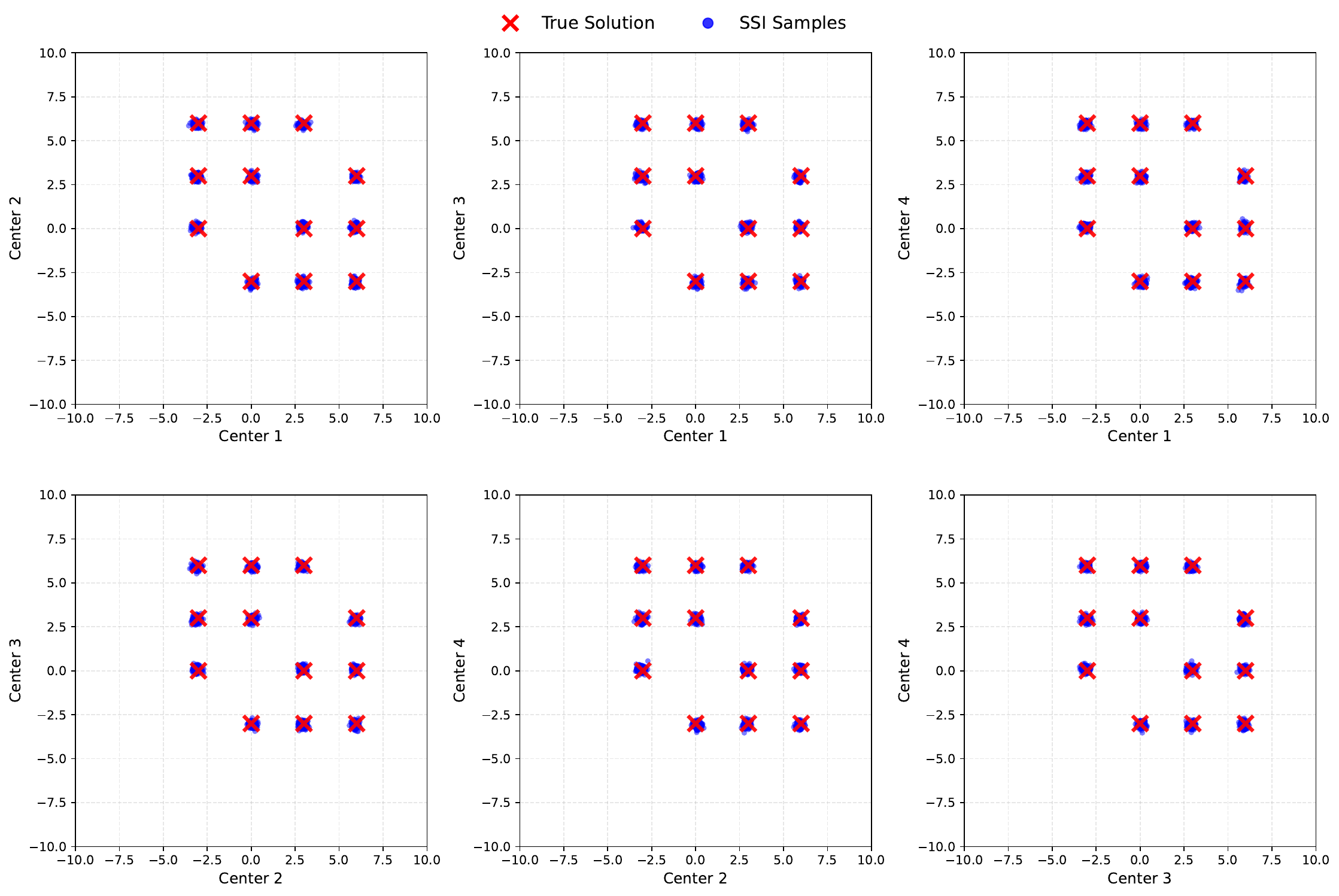}
\caption{Pairwise marginal posterior distributions on $\bbR^{2}$ for the Bayesian Gaussian model in Section~\ref{Bayesian:inferrence:gmm}. The red crosses indicate the true locations of the posterior modes, while the blue points represent the particles sampled by SSI.}
\label{fig:bayesian:inference}
\end{figure}

\begin{remark}[Constrained sampling using projected Langevin Monte Carlo]
Since the prior distribution in~\eqref{section:app:mixture:posterior} is a uniform distribution with compact support, the denoising distribution $p_{X_{1}\mid X_{t}=x_{t}}$ in~\eqref{eq:denoising:density} also has compact support. To sample from this distribution over a constrained domain, in practice we employ the projected Langevin Monte Carlo algorithm~\cite{Bubeck2015Finite}, projecting the particles onto the support of the distribution after each Langevin step~\eqref{eq:langevin:monte:carlo}. The error analysis of projected Langevin Monte Carlo under log-concavity assumptions has been established in the literature; see, e.g.,~\cite{Bubeck2015Finite,Bubeck2018Sampling}. We leave the extension of our analysis to constrained distributions for future work.
\end{remark}

\subsection{Ablation studies}
\label{section:experiments:ablation}

\par In this subsection, we provide an ablation study to show the influence of the choice of the initialization time $T_{0}$ in~\eqref{eq:PFODE}, and demonstrate the benefits of preconditioning introduced in Section~\ref{section:preconditioning}. We take \texttt{MoG7x7} as an example. The MMD and $W_{2}$ change against the initialization time $T_{0}$ are shown in Table~\ref{tab:exp:ablation}.

\begin{table}[htbp]
\centering\scriptsize
\setlength{\tabcolsep}{2pt}
\caption{Comparison of maximum mean discrepancy (MMD), and 2-Wasserstein distance ($W_{2}$) for different initialization time $T_{0}$.}
\label{tab:exp:ablation}
\begin{tabular}{crccccccc}
\toprule
precond. & & $T_{0}=0.01$ & $T_{0}=0.1$ & $T_{0}=0.2$ & $T_{0}=0.3$ & $T_{0}=0.4$ & $T_{0}=0.5$ & $T_{0}=0.6$ \\
\midrule
\ding{56} & MMD $\downarrow$ & $1.82\times 10^{-2}$ & $\bm{1.52\times 10^{-2}}$ & $4.22\times 10^{-2}$ & $3.33\times 10^{-1}$ & $1.51$ & $2.20$ & $2.23$ \\
\ding{56} & $W_{2}$ $\downarrow$ & $2.84$ & $\bm{2.81}$ & $4.04$ & $1.08\times 10^{1}$ & $2.14\times 10^{1}$ & $2.53\times 10^{1}$ & $2.59\times 10^{1}$ \\
\midrule 
\ding{51} & MMD $\downarrow$ & $1.88\times 10^{-2}$ & $1.67\times 10^{-2}$ & $7.24\times 10^{-3}$ & $\bm{1.69\times 10^{-3}}$ & $1.04\times 10^{-1}$ & $7.34\times 10^{-1}$ & $1.21$ \\
\ding{51} & $W_{2}$ $\downarrow$ & $5.48$ & $2.58$ & $1.96$ & $\bm{1.79}$ & $6.22$ & $1.54\times 10^{1}$ & $1.93\times 10^{1}$ \\
\bottomrule
\end{tabular}
\end{table}

\paragraph{The choice of the initialization time} 
As shown in Table~\ref{tab:exp:ablation}, the performance of SSI, both with and without preconditioning, is non-monotonic with respect to the initialization time. Specifically, errors are large for both small initialization times and times approaching one. This empirical observation aligns with our theoretical results in Lemmas~\ref{lemma:velocity:estimation},~\ref{lemma:velocity:estimation:stable}, and~\ref{lemma:error:initialization}. Specifically, for sufficiently small initialization times, the convergence of the Langevin Monte Carlo~\eqref{eq:langevin:monte:carlo} for velocity estimation cannot be guaranteed (Lemmas~\ref{lemma:velocity:estimation} and~\ref{lemma:velocity:estimation:stable}), leading to significant errors. Conversely, for sufficiently large initialization times, the initialization distribution $p_{X_{T_{0}}}$ is close to the target distribution $p_{X_{1}}$; it becomes highly multi-modal, preventing the Langevin Monte Carlo~\eqref{eq:langevin:monte:carlo:warmstart} from converging rapidly during initialization (Lemma~\ref{lemma:error:initialization}).

\begin{remark}[Theory--practice gap]
Note that a gap exists between our theoretical analysis and our numerical experiments, and this experiments indeed indicate that our method performs considerably better than the theory predicts. Consider, for instance, the target distribution \texttt{MoG7x7}, whose parameters are $R^{2}=1800$ and $\sigma=0.5$, so that $R\gg\sigma$. Proposition~\ref{lemma:denoising:langevin:convergence} then requires the initialization time to satisfy $T_{0}>T^{*}\approx 0.994$ in order to guarantee log-concavity of the denoising distribution. Since $T^{*}$ is so close to $1$, the initialization distribution $p_{X_{T_{0}}}$ is nearly indistinguishable from the target $p_{X_{1}}$; sampling from $p_{X_{T_{0}}}$ is therefore almost as difficult as sampling from $p_{X_{1}}$ itself, and the theoretical guarantee becomes essentially vacuous in this regime. In practice, however, we set $T_{0}=0.2$ and our method still performs well. This suggests that the velocity field can be estimated with acceptable accuracy for $t>T_{0}$ even when $T_{0}<T^{*}$, i.e., that the log-concavity requirement of Proposition~\ref{lemma:denoising:langevin:convergence} is sufficient but far from necessary.
\end{remark}

\paragraph{The benefits of preconditioning} 
Table~\ref{tab:exp:ablation} demonstrates that SSI with preconditioning outperforms the method without preconditioning for nearly every choice of initialization time, highlighting the advantages of the technique introduced in Section~\ref{section:preconditioning}. More importantly, the first two lines of Table~\ref{tab:exp:ablation} reveal that SSI without preconditioning is highly sensitive to the initialization time. In contrast, SSI with preconditioning exhibits a broader ``sweet spot,'' ranging from $T_{0}=0.2$ to $T_{0}=0.3$. This demonstrates that preconditioning significantly enhances robustness to the choice of initialization time. Furthermore, enabling the use of a larger initialization time reduces the number of iterations required to simulate the probability flow ODE~\eqref{eq:PFODE:velocity:ei} or~\eqref{eq:PFODE:velocity:ei:stable}, thereby lowering computational costs. This improvement arises because preconditioning simultaneously improves the sampling performance of Langevin diffusion for both velocity estimation and initialization. These results further highlight the advantages of our preconditioned approach over existing methods based on vanilla Langevin diffusion, such as~\cite{huang2024reverse,Grenioux2024Stochastic}.

\paragraph{The necessity of Langevin-based initialization}
For the initialization time $T_{0}=0.01$, Langevin-based initialization becomes unnecessary, and a simple Gaussian initialization suffices, since the initialization distribution $p_{X_{T_{0}}}$ remains close to the standard Gaussian. However, as shown in Table~\ref{tab:exp:ablation}, while choosing a small initialization time eliminates the need for this Langevin-based initialization step, it leads to a non-negligible accuracy loss in both the preconditioned and non-preconditioned cases. 

\subsection{Comparison with SLIPS and RDMC}

We compare SSI, both with and without preconditioning, with SLIPS~\cite{Grenioux2024Stochastic} and RDMC~\cite{huang2024reverse} on the \texttt{MoG40} target. We vary the number of integration steps, $M$, in~\eqref{eq:PFODE:velocity:ei:stable} or~\eqref{eq:PFODE:velocity:ei} from $100$ to $500$. To ensure a fair comparison, we use the same hyperparameter settings for velocity or score estimation across all methods.

\begin{table}[htbp]
\centering
\small
\setlength{\tabcolsep}{4pt}
\caption{Comparison of SSI with SLIPS and RDMC on the \texttt{MoG40} target for different number of integration steps $M$.}
\label{tab:exp:slips:rdmc}
\begin{tabular}{llcccc}
\toprule
Method & Metric & Reference & $M=100$ & $M=140$ & $M=180$ \\
\midrule
RDMC~\cite{huang2024reverse}
& NLL $\downarrow$
& $6.86$ & $8.17$ & $8.14$ & $8.10$ \\
& MMD $\downarrow$
& $3.92\times10^{-4}$ & $4.27\times10^{-2}$
& $4.29\times10^{-2}$ & $4.27\times10^{-2}$ \\
& $W_{2}$ $\downarrow$
& $0.79$ & $6.41$ & $6.43$ & $6.42$ \\
\midrule
SLIPS~\cite{Grenioux2024Stochastic}
& NLL $\downarrow$
& $6.86$ & $7.20$ & $7.17$ & $7.18$ \\
& MMD $\downarrow$
& $3.92\times10^{-4}$ & $3.80\times10^{-2}$
& $3.79\times10^{-2}$ & $3.74\times10^{-2}$ \\
& $W_{2}$ $\downarrow$
& $0.79$ & $5.52$ & $5.50$ & $5.46$ \\
\midrule
\textbf{SSI w/o precond.}
& NLL $\downarrow$
& $6.86$ & $7.49$ & $7.30$ & $7.23$ \\
& MMD $\downarrow$
& $3.92\times10^{-4}$ & $2.86\times10^{-2}$
& $1.69\times10^{-2}$ & $1.09\times10^{-2}$ \\
& $W_{2}$ $\downarrow$
& $0.79$ & $3.94$ & $3.03$ & $2.52$ \\
\midrule
\textbf{SSI}
& NLL $\downarrow$
& $6.86$ & $7.54$ & $7.35$ & $7.27$ \\
& MMD $\downarrow$
& $3.92\times10^{-4}$ & $2.26\times10^{-2}$
& $1.63\times10^{-2}$ & $1.27\times10^{-2}$ \\
& $W_{2}$ $\downarrow$
& $0.79$ & $4.00$ & $3.53$ & $3.13$ \\
\midrule
Method & Metric & $M=200$ & $M=300$ & $M=400$ & $M=500$ \\
\midrule
RDMC~\cite{huang2024reverse}
& NLL $\downarrow$
& $8.12$ & $8.12$ & $8.11$ & $8.12$ \\
& MMD $\downarrow$
& $4.27\times10^{-2}$ & $4.31\times10^{-2}$
& $4.34\times10^{-2}$ & $4.36\times10^{-2}$ \\
& $W_{2}$ $\downarrow$
& $6.43$ & $6.45$ & $6.48$ & $6.49$ \\
\midrule
SLIPS~\cite{Grenioux2024Stochastic}
& NLL $\downarrow$
& $7.18$ & $7.16$ & $7.15$ & $7.15$ \\
& MMD $\downarrow$
& $3.74\times10^{-2}$ & $3.68\times10^{-2}$
& $3.62\times10^{-2}$ & $3.60\times10^{-2}$ \\
& $W_{2}$ $\downarrow$
& $5.47$ & $5.39$ & $5.34$ & $5.33$ \\
\midrule
\textbf{SSI w/o precond.}
& NLL $\downarrow$
& $7.18$ & $7.11$ & $7.11$ & $7.08$ \\
& MMD $\downarrow$
& $1.23\times10^{-2}$ & $1.01\times10^{-2}$
& $7.03\times10^{-3}$ & $4.92\times10^{-3}$ \\
& $W_{2}$ $\downarrow$
& $2.67$ & $2.50$ & $2.15$ & $1.86$ \\
\midrule
\textbf{SSI}
& NLL $\downarrow$
& $7.23$ & $7.16$ & $7.17$ & $7.14$ \\
& MMD $\downarrow$
& $1.31\times10^{-2}$ & $1.38\times10^{-2}$
& $1.50\times10^{-2}$ & $1.42\times10^{-2}$ \\
& $W_{2}$ $\downarrow$
& $3.16$ & $3.18$ & $3.42$ & $3.26$ \\
\bottomrule
\end{tabular}
\end{table}

Table~\ref{tab:exp:slips:rdmc} shows that, across all values of $M$, SSI substantially outperforms both SLIPS and RDMC, with and without preconditioning. This advantage may stem from the deterministic probability flow ODE employed by SSI, which can incur less discretization error than the corresponding SDE under coarse step sizes, as discussed in Remark~\ref{remark:sde:ode}.

%% file: conclusion.tex

\section{Conclusions}
\label{section:conclusion}

\par We propose a novel framework for sampling from unnormalized Boltzmann densities driven by linear stochastic interpolants. Linear stochastic interpolants induce a probability flow ODE that transports an initialization distribution to the target distribution. The initialization distribution is given by a Gaussian convolution of the target density and is therefore substantially easier to sample from. Moreover, by Tweedie's formula, the velocity field of the probability flow ODE can be estimated via Monte Carlo methods, where the associated sampling distribution is likewise significantly simpler. As a result, our approach decomposes the challenging problem of sampling from a multi-modal and complex target density into a sequence of tractable subproblems, each of which can be addressed using Langevin diffusion. These sub-sampling tasks include (i) using Langevin diffusion to generate samples from the initialization distribution, and (ii) employing Langevin diffusion to estimate the velocity field. Simulating the resulting probability flow ODE with Langevin-based initialization and velocity estimation yields samples distributed according to the target density. 

In future work, we will apply our method to high-dimensional and complex Bayesian inverse problems, demonstrating its capability for uncertainty quantification in more realistic settings. We also aim to develop a rigorous theoretical analysis of preconditioned Langevin Monte Carlo, and integrate interacting particle samplers~\cite{Sprungk2025Metropolis,stein2025understanding} to further enhance the performance of the underlying Langevin components.

%% file: appendix.tex

\section{Auxiliary Results}

For the proofs in the following sections, we need some auxiliary computations.

\begin{lemma}\label{lemma:norm:Xt}
Let Assumption~\ref{assumption:Gaussian:convolution} be fulfilled. 
Then, for $\sigma \in (0,1]$ and any $t\in[0,1]$, it holds $\bbE[\|X_{t}\|_{2}^{2}]\leq d+R^{2}$.
\end{lemma}

\begin{proof}
By Assumption~\ref{assumption:Gaussian:convolution} and~\eqref{eq:stochastic:interporlant}, we obtain 
\begin{align*}
\bbE\big[\|X_{t}\|_{2}^{2}\big]
&=\bbE\big[\|tY+(1-t(1-\sigma))\Xi\|_{2}^{2}\big] \\
&=t^{2}\bbE\big[\|Y\|_{2}^{2}\big]+ 2 t(1-t(1-\sigma))\bbE\big[\langle Y,\Xi \rangle\big]+(1-t(1-\sigma))^{2}\bbE\big[\|\Xi\|_{2}^{2}\big] \\
&=t^{2}\bbE\big[\|Y\|_{2}^{2}\big]+(1-t(1-\sigma))^{2}\bbE\big[\|\Xi\|_{2}^{2}\big] \\
&\leq t^2 R^2 + (1-t(1-\sigma))^{2} d \leq
d+R^{2}.
\end{align*}
This completes the proof.
\end{proof}

\begin{lemma}\label{lemma:tweedies:second:order}
For every $t\in(0,1)$, we have
\begin{equation*}
\nabla\bbE[X_{1}|X_{t}=x_{t}]=\frac{t}{(1-t)^2}\Cov(X_{1}|X_{t}=x_{t}).
\end{equation*}
\end{lemma}

\begin{proof}
It follows from~\eqref{eq:denoising:density} that 
\begin{align*}
\nabla \bbE[X_{1}|X_{t}=x_{t}] 
&= \nabla \Big(\int x_{1} p_{X_1|X_t = x_t} (x_{1}) \d x_{1}\Big) \\
&=
\nabla \Big(\int x_{1} \gamma_{(1-t)^{2}}(x_{t}-tx_{1})
\frac{p_{X_1}(x_{1})}{p_{X_t}(x_{t})}\d x_{1}\Big) \\
&=
-\int x_{1} \Big(\frac{x_{t}-tx_{1}}{(1-t)^{2}}\Big)^{\top} \gamma_{(1-t)^{2}}(x_{t}-tx_{1})
\frac{p_{X_1}(x_1)}{p_{X_t}(x_t)} \d x_{1}
\\
&\quad 
- \int x_{1} \gamma_{(1-t)^{2}}(x_t-tx_1)
\frac{p_{X_1}(x_1)}{p_{X_t}(x_t)} \d x_{1} \nabla\log p_{X_t}(x_{t})^\top
\end{align*}
and further by~\eqref{eq:velocity} and~\eqref{eq:score:velocity} that
\begin{align*}
&\nabla \bbE[X_{1}|X_{t}=x_{t}] \\
&=-\frac{1}{(1-t)^{2}}\bbE[X_{1}|X_{t}=x_{t}]x_{t}^{\top}+\frac{t}{(1-t)^{2}}\bbE[X_{1}X_{1}^{\top}|X_{t}=x_{t}] \\
&\quad+\frac{1}{(1-t)^2}\bbE[X_{1}|X_{t}=x_{t}]x_t^{\top}-\frac{t}{(1-t)^2}\bbE[X_1|X_t=x_t]\bbE[X_{1}|X_{t}=x_{t}]^{\top} \\
&=\frac{t}{(1-t)^2}\Cov(X_{1}|X_{t}=x_{t}). \qedhere
\end{align*}
\end{proof}

\par The following lemma is an extension of Lemma~\ref{lemma:tweedies:second:order}.

\begin{lemma}[{\cite[Lemma C.3]{chang2026inference}}]\label{lemma:tweedies:second:order:f}
Let $f:\bbR^{d}\rightarrow\bbR$ be an integrable function. For every $t\in(0,1)$ and every $x_{t}\in\mathbb{R}^{d}$, we have
\begin{equation*}
\nabla\bbE[f(X_{1})|X_{t}=x_{t}]=\frac{t}{(1-t)^2}(\bbE[f(X_{1})X_{1}|X_{t}=x_{t}]-\bbE[f(X_1)|X_t=x_t]\bbE[X_{1}|X_{t}=x_{t}]).
\end{equation*}
\end{lemma}

\begin{proof}
It follows from~\eqref{eq:denoising:density} that 
\begin{align*}
\nabla \bbE[f(X_{1})|X_{t}=x_{t}] 
&= \nabla \Big(\int f(x_{1}) p_{X_1|X_t = x_t} (x_{1}) \d x_{1}\Big) \\
&=
\nabla \Big(\int f(x_{1}) \gamma_{(1-t)^{2}}(x_{t}-tx_{1})
\frac{p_{X_1}(x_{1})}{p_{X_t}(x_{t})}\d x_{1}\Big) \\
&=
-\int f(x_{1}) \Big(\frac{x_{t}-tx_{1}}{(1-t)^{2}}\Big) \gamma_{(1-t)^{2}}(x_{t}-tx_{1})
\frac{p_{X_1}(x_1)}{p_{X_t}(x_t)} \d x_{1}
\\
&\quad 
- \int f(x_{1}) \gamma_{(1-t)^{2}}(x_t-tx_1)
\frac{p_{X_1}(x_1)}{p_{X_t}(x_t)} \d x_{1} \nabla\log p_{X_t}(x_{t})
\end{align*}
and further by~\eqref{eq:velocity} and~\eqref{eq:score:velocity} that
\begin{align*}
&\nabla \bbE[f(X_{1})|X_{t}=x_{t}] \\
&=-\frac{1}{(1-t)^{2}}\bbE[f(X_{1})|X_{t}=x_{t}]x_{t}+\frac{t}{(1-t)^{2}}\bbE[f(X_{1})X_{1}|X_{t}=x_{t}] \\
&\quad+\frac{1}{(1-t)^2}\bbE[f(X_{1})|X_{t}=x_{t}]x_t-\frac{t}{(1-t)^2}\bbE[f(X_1)|X_t=x_t]\bbE[X_{1}|X_{t}=x_{t}]. \qedhere
\end{align*}
\end{proof}

By similar arguments as the proof of Lemma~\ref{lemma:tweedies:second:order}, we obtain the following lemma.

\begin{lemma}\label{lemma:tweedies:second:order:gaussian:convolution}
Let Assumption~\ref{assumption:Gaussian:convolution} be fulfilled. Then it holds
\begin{equation*}
\nabla\bbE[Y|X_{1}=x_{1}]=\frac{1}{\sigma^2}\Cov(Y|X_{1}=x_{1}).
\end{equation*}
\end{lemma}

Further, we will need the following relation. 

\begin{lemma}\label{lemma:representation:conditional}
Let  Assumption~\ref{assumption:Gaussian:convolution} be fulfilled. Then, we have for every $t\in(0,1)$ and $x_t \in \mathbb R ^d$ that
\begin{equation*}
( X_{1}| X_{t}= x_{t}) \stackrel{\d}
{=}
\frac{t}{(1-t)^{2}} \sigma_t^{2} x_{t}+\frac{\sigma_t^2}{\sigma^{2}} Y^{x_{t}}_t+\sigma_t \Xi,
\end{equation*}
where  $Y^{x_{t}}_t \sim \frac{\gamma_{\sigma^2 t^2 + (1-t)^2} (x_t - t \, \cdot) 
\, 
p_Y(\cdot)}{p_{X_t}(x_t)}$ and
$$\sigma_t^2 \coloneqq 
    \frac{\sigma^{2}(1-t)^{2}}{\sigma^{2}t^{2}+(1-t)^2}$$
    In particular
    \begin{equation}
    \bbE[f(X_1) \mid X_t = x_t] = \bbE \left[ f\left( \frac{t \sigma_t^2}{(1-t)^2} x_t + \frac{\sigma_t^2}{\sigma^2} Y_t^{x_t} + \sigma_t \Xi \right) \right]
\end{equation}
and thus 
$\supp(Y^{x_{t}}_t)=\supp(p_Y)$ and
\begin{align*}
\bbE[X_{1}|X_{t}=x_{t}]
&= \frac{t}{(1-t)^{2}} \sigma_t^{2} x_{t}
+\frac{\sigma_t^2}{\sigma^{2}} \bbE[Y^{x_{t}}_t],\\
\Cov(X_1|X_t=x_{t})
&=\
\frac{\sigma_t^4}{\sigma^{4}}\Cov(Y^{x_{t}}_t) + \sigma_t^2  I_d.
\end{align*}
\end{lemma}

\begin{proof}
By straightforward computation we obtain
\begin{align*}
    p_{X_1|X_t = x_t} (x_1)
    &=\frac{1}{p_{X_t}(x_{t})} \int \gamma_{(1-t)^{2}}(x_{t}-tx_{1})
    \gamma_{\sigma^{2}}(x_{1}-y) \, p_Y (y) \d y  \\
&=
\int \gamma_{\sigma_t^2} 
\Big(x_{1} - \frac{t}{(1-t)^{2}} \sigma_t^{2} x_t
-
\frac{\sigma_t^2}{\sigma^{2}} y\Big) \, 
\frac{\gamma_{\sigma^2 t^2 + (1-t)^2} (x_t - ty) \, p_Y (y)}{p_{X_t}(x_t)} \d y. 
    \end{align*}
Using~\eqref{eq:conv} yields the assertion.
\end{proof}

The following lemma provides a bound on the time partial derivative of the score, which is inspired by~\cite{beyler2025convergence} and~\cite[Appendix B]{saremi2024chain}.

\begin{lemma}\label{lemma:time:derivative:score}
Let Assumption~\ref{assumption:Gaussian:convolution} be fulfilled. Then, for any $t\in(0,1)$, it holds
\begin{equation*}
\|\partial_{t}\nabla\log p_{X_{t}}(x_{t})\|_{2}\leq \tilde H \frac{\|x_{t}\|_{2}+1}{(1-t)^{3}}, \quad x_{t}\in\bbR^{d},
\end{equation*}
where $\tilde H$ is a constant depending only on $R$ and $\sigma$.
\end{lemma}

\begin{proof}
By the continuity equation~\eqref{eq:transport}, we have 
\begin{align*}
\partial_{t}\log p_{X_{t}}(x_{t})
&=\frac{\partial_{t} p_{X_{t}}(x_{t})}{p_{X_{t}}(x_{t})} 
=-\nabla\cdot u(t,x_{t})-\langle u(t,x_{t}) , \nabla\log p_{X_{t}}(x_{t})\rangle \\
&=-\trace(\nabla u(t,x_{t}))-\langle u(t,x_{t}),  \nabla\log p_{X_{t}}(x_{t})\rangle.
\end{align*}
Taking gradient with respect to $x_{t}$ yields
\begin{equation}\label{eq:lemma:time:derivative:score:1}
\nabla\partial_{t}\log p_{X_{t}}(x_{t})=-\underbrace{\nabla\trace(\nabla u(t,x_{t}))}_{A_1}-\underbrace{\nabla u(t,x_{t})\nabla\log p_{X_{t}}(x_{t})}_{A_2}-\underbrace{\nabla^{2}\log p_{X_{t}}(x_{t})u(t,x_{t})}_{A_3}.
\end{equation}

\par\noindent\textbf{Estimation of $A_1$:}
Using~\eqref{eq:velocity} and Lemma~\ref{lemma:tweedies:second:order}, we conclude
\begin{align*}
\nabla u(t,x_{t})
&=-\frac{1}{1-t} I_{d}+\frac{1}{1-t}\nabla\bbE[X_{1}|X_{t}=x_{t}] \\
&=-\frac{1}{1-t} I_{d}+\frac{t}{(1-t)^{3}} \big(\bbE[X_{1}X_{1}^{\top}|X_{t}=x_{t}]-\bbE[X_{1}|X_{t}=x_{t}]\bbE[X_{1}|X_{t}=x_{t}]^{\top} \big),
\end{align*}
and consequently
\begin{equation*}
\trace(\nabla u(t,x_{t}))=-\frac{d}{1-t}+\frac{t}{(1-t)^{3}}(\bbE[\|X_{1}\|_{2}^{2}|X_{t}=x_{t}]-\|\bbE[X_{1}|X_{t}=x_{t}]\|_{2}^{2}).
\end{equation*}
Taking gradient with respect to $x_{t}$ and invoking Lemma~\ref{lemma:tweedies:second:order:f} yields
\begin{equation}\label{eq:lemma:time:derivative:score:2}
\begin{aligned}
\nabla\trace(\nabla u(t,x_{t}))
&=\frac{t}{(1-t)^{3}}(\nabla\bbE[\|X_{1}\|_{2}^{2}|X_{t}=x_{t}]-2\nabla\bbE[X_{1}|X_{t}=x_{t}]\bbE[X_{1}|X_{t}=x_{t}]) \\
&= \frac{t^{2}}{(1-t)^{5}}\underbrace{\bbE[\|X_{1}\|_{2}^{2}(X_{1}-\bbE[X_{1}|X_{t}=x_{t}])|X_{t}=x_{t}]}_{\text{(iv)}} \\
&\quad -\frac{2t}{(1-t)^{3}}\underbrace{\nabla\bbE[X_{1}|X_{t}=x_{t}]\bbE[X_{1}|X_{t}=x_{t}]}_{\text{(v)}}.
\end{aligned}
\end{equation}
 We consider the summands and start with (iv). Recall that $D(t,x_t)= \bbE[X_1 \mid X_t = x_t]$ and we will compute:
\begin{equation}
    \mathcal{T} \coloneqq \bbE \left[ \|X_1\|_2^2 (X_1 - D(t,x_t)) \mid X_t = x_t \right].
\end{equation}
Let $A = \frac{t \sigma_t^2}{(1-t)^2} x_t + \frac{\sigma_t^2}{\sigma^2} Y_t^{x_t}$ and $B = \sigma_t \Xi$. Using the identity $D(t,x_t) = \bbE[A]$ and using Lemma~\ref{lemma:representation:conditional}
\begin{equation}
    \mathcal{T} = \bbE \left[ \|A + B\|_2^2 (A + B - \bbE[A]) \right] = \bbE \left[ (\|A\|_2^2 + 2\langle A, B \rangle + \|B\|_2^2)(A - \bbE[A] + B) \right].
\end{equation}
Expanding this product gives several terms. Since $\Xi$ (and thus $B$) is independent of $Y_t^{x_t}$ (and thus $A$) and has a symmetric distribution ($\bbE[B] = 0$ and $\bbE[\|B\|_2^2 B] = 0$), the following terms vanish:
\begin{enumerate}
    \item $\bbE[\|A\|_2^2 B] = \bbE[\|A\|_2^2] \bbE[B] = 0$.
    \item $\bbE[2\langle A, B \rangle (A - \bbE[A])] = 2\sum_{i=1}^d\bbE[A_i(A - \bbE[A])] \bbE[B_i]  = 0$.
    \item $\bbE[\|B\|_2^2 (A - \bbE[A])] = \bbE[\|B\|_2^2] \bbE[A - \bbE[A]] = \sigma_t^2 d \cdot 0 = 0$.
    \item $\bbE[\|B\|_2^2 B] = 0$ (third-order moment of centered Gaussian).
\end{enumerate}
The remaining non-zero terms are:
\begin{equation}
    \mathcal{T} = \bbE[\|A\|_2^2 (A - \bbE[A])] + \bbE[2 \langle A, B \rangle B].
\end{equation}
For the second term, since $\bbE[B B^\top] = \sigma_t^2 I_d$, we have $\bbE[2 \langle A, B \rangle B] = 2 \bbE[B B^\top] \bbE[A] = 2\sigma_t^2 \bbE[A]$. For the first term, we write $A = \bbE[A] + \frac{\sigma_t^2}{\sigma^2}(Y_t^{x_t} - \bbE[Y_t^{x_t}])$. Expanding this similarly yields:
\begin{equation}
\begin{aligned}
    \mathcal{T} &= 2\frac{\sigma_t^{4}}{\sigma^{4}}\Cov(Y_{t}^{x_{t}})\Big(\frac{t\sigma_t^{2}}{(1-t)^{2}}x_{t}+\frac{\sigma_t^2}{\sigma^{2}}\bbE[Y_{t}^{x_{t}}]\Big) \\
    &\quad +\frac{\sigma_t^{6}}{\sigma^{6}}\bbE\Big[\|Y_{t}^{x_{t}}-\bbE[Y_{t}^{x_{t}}]\|_{2}^{2}(Y_{t}^{x_{t}}-\bbE[Y_{t}^{x_{t}}])\Big] \\
    &\quad +2\sigma_{t}^{2}\Big(\frac{t\sigma_t^{2}}{(1-t)^{2}}x_{t}+\frac{\sigma_t^2}{\sigma^{2}}\bbE[Y^{x_{t}}_{t}]\Big).
\end{aligned}
\end{equation}

Then using the triangular inequality, we have 
\begin{align}
&\|\bbE[\|X_{1}\|_{2}^{2}(X_{1}-\bbE[X_{1}|X_{t}=x_{t}])|X_{t}=x_{t}]\|_{2} \nonumber \\
&\leq 2\Big(\frac{\sigma_t^{4}}{\sigma^{4}}R^{2}+\sigma_{t}^{2}\Big)\Big(\frac{t\sigma_t^{2}}{(1-t)^{2}}\|x_{t}\|_{2}+\frac{\sigma_t^2}{\sigma^{2}}R\Big)+8\frac{\sigma_t^{6}}{\sigma^{6}}R^{3}, \label{eq:lemma:time:derivative:score:3}
\end{align}
where we used the fact that $\|Y^{x_{t}}_{t}\|_{2}\leq R$. For the term (v) in~\eqref{eq:lemma:time:derivative:score:2}, according to Lemmas~\ref{lemma:tweedies:second:order} and~\ref{lemma:representation:conditional}, we have 
\begin{align}
&\|\nabla\bbE[X_{1}|X_{t}=x_{t}]\bbE[X_{1}|X_{t}=x_{t}]\|_{2} \nonumber \\
&\leq\frac{t}{(1-t)^2}\Big(\frac{\sigma_{t}^{4}}{\sigma^{4}}R^{2}+\sigma_{t}^{2}\Big)\Big(\frac{t\sigma_t^{2}}{(1-t)^{2}}\|x_{t}\|_{2}+\frac{\sigma_t^2}{\sigma^{2}}R\Big), \label{eq:lemma:time:derivative:score:4}
\end{align}
where we used the fact that $\|Y^{x_{t}}_{t}\|_{2}\leq R$. Substituting~\eqref{eq:lemma:time:derivative:score:3} and~\eqref{eq:lemma:time:derivative:score:4} into~\eqref{eq:lemma:time:derivative:score:2} and using the triangular inequality, we get
\begin{align}
\|\nabla\trace(\nabla u(t,x_{t}))\|_{2} 
&\leq 4\frac{t^{2}}{(1-t)^{5}}\Big(\frac{\sigma_t^{4}}{\sigma^{4}}R^{2}+\sigma_{t}^{2}\Big)\Big(\frac{t\sigma_t^{2}}{(1-t)^{2}}\|x_{t}\|_{2}+\frac{\sigma_t^2}{\sigma^{2}}R\Big)+8\frac{t^{2}}{(1-t)^{5}}\frac{\sigma_t^{6}}{\sigma^{6}}R^{3} \nonumber \\
&\leq 4\frac{t^{2}}{(1-t)^{3}}\Big(\frac{(1-t)^{2}}{\zeta_{t}^{4}}R^{2}+\frac{\sigma^{2}}{\zeta_{t}^{2}}\Big)\Big(\frac{t\sigma^{2}}{\zeta_{t}^{2}}\|x_{t}\|_{2}+\frac{(1-t)^2}{\zeta_{t}^{2}}R\Big)+8\frac{t^{2}(1-t)}{\zeta_{t}^{6}}R^{3}, \label{eq:lemma:time:derivative:score:5}
\end{align}
where $\zeta_{t}^{2}\coloneq \sigma^{2}t^{2}+(1-t)^{2}$, and the last inequality invokes the definition of $\sigma_{t}^{2}$ in Lemma~\ref{lemma:representation:conditional}.

\vspace{1ex}

\par\noindent\textbf{Estimation of $A_2$:}
Taking gradient with respect to $x_{t}$ on both sides of~\eqref{eq:velocity} implies 
\begin{align}
\nabla u(t,x_{t})
&= -\frac{1}{1-t}\mathrm{I}_{d}+\frac{1}{1-t}\nabla\bbE[X_{1}|X_{t}=x_{t}] \\
&= \frac{t\sigma_{t}^{2}-(1-t)^{2}}{(1-t)^{3}} I_d+\frac{t}{(1-t)^{3}}\frac{\sigma_t^4}{\sigma^{4}}\Cov(Y^{x_{t}}_t) \nonumber  \\
&= \frac{t\sigma^{2}-1+t}{\zeta_{t}^{2}} I_{d}+\frac{t(1-t)}{\zeta_{t}^{4}}\Cov(Y_{t}^{x_{t}}), \label{eq:lemma:time:derivative:score:6:1}
\end{align}
where the second equality holds from Lemmas~\ref{lemma:tweedies:second:order} and~\ref{lemma:representation:conditional}, and the last equality used $\zeta_{t}^{2}\coloneq \sigma^{2}t^{2}+(1-t)^{2}$. Now~\eqref{eq:score:velocity} and Lemma~\ref{lemma:representation:conditional} yield 
\begin{align} 
\nabla\log p_{X_{t}}(x_{t})
&= -\frac{1}{(1-t)^{2}}x_{t}+\frac{t}{(1-t)^{2}}\bbE[X_{1}|X_{t}=x_{t}] \label{eins} \\
&= \frac{t^{2}\sigma_t^{2}-(1-t)^{2}}{(1-t)^{4}}x_{t}
+\frac{t}{(1-t)^{2}}\frac{\sigma_t^2}{\sigma^{2}}\bbE[Y^{x_{t}}_t] \nonumber \\
&= -\frac{1}{\zeta_{t}^{2}}x_{t}
+\frac{t}{\zeta_{t}^{2}}\bbE[Y^{x_{t}}_t]. \label{eq:lemma:time:derivative:score:6:0}
\end{align}
Therefore,  using the fact that $\|Y^{x_{t}}_{t}\|_{2}\leq R$, we conclude
\begin{equation}\label{eq:lemma:time:derivative:score:6}
\|\nabla u(t,x_{t})\nabla\log p_{X_{t}}(x_{t})\|_{2}\leq \Big(\frac{\sigma^{2}+2}{\zeta_{t}^{2}}+\frac{t(1-t)}{\zeta_{t}^{4}}R^{2}\Big)\Big(\frac{1}{\zeta_{t}^{2}}\|x_{t}\|_{2}
+\frac{t}{\zeta_{t}^{2}}R\Big).
\end{equation}

\vspace{1ex}

\par\noindent\textbf{Estimation of $A_3$:}
Taking gradient with respect to $x_{t}$ on both sides of~\eqref{eins} yields 
\begin{align*}
\nabla^{2}\log p_{X_{t}}(x_{t})
&= -\frac{1}{(1-t)^{2}} I_{d}+\frac{t}{(1-t)^{2}}\nabla\bbE[X_{1}|X_{t}=x_{t}] \\
&= \frac{t\sigma_t^{2}-(1-t)^{2}}{(1-t)^{4}} I_{d}
+\frac{t}{(1-t)^{4}}\frac{\sigma_{t}^{4}}{\sigma^{4}}\Cov(Y_{t}^{x_{t}}) \\
&= \frac{1}{1-t}\Big(\frac{t\sigma^{2}-1+t}{\zeta_{t}^{2}} I_{d}+\frac{t}{\zeta_{t}^{4}}\Cov(Y_{t}^{x_{t}})\Big),
\end{align*}
where the second equality holds from Lemmas~\ref{lemma:tweedies:second:order} and~\ref{lemma:representation:conditional}. Similarly, combining~\eqref{eq:velocity} and Lemma~\ref{lemma:representation:conditional} deduces
\begin{align}
u(t,x_{t})
&=-\frac{1}{1-t}x_{t}+\frac{1}{1-t}\bbE[X_{1}|X_{t}=x_{t}] \nonumber \\
&=\frac{t\sigma_t^{2}-(1-t)^{2}}{(1-t)^{3}}x_{t}+\frac{1}{1-t}\frac{\sigma_t^2}{\sigma^{2}}\bbE[Y_{t}^{x_{t}}] \nonumber \\
&=\frac{t\sigma^{2}-1+t}{\zeta_{t}^{2}}x_{t}+\frac{1-t}{\zeta_{t}^{2}}\bbE[Y_{t}^{x_{t}}], \label{eq:lemma:time:derivative:score:8}
\end{align}
As a result, we have 
\begin{equation}\label{eq:lemma:time:derivative:score:7}
\|\nabla^{2}\log p_{X_{t}}(x_{t})u(t,x_{t})\|_{2}\leq \frac{1}{1-t}\Big(\frac{\sigma^{2}+2}{\zeta_{t}^{2}}+\frac{t}{\zeta_{t}^{4}}R^{2}\Big)\Big(\frac{\sigma^{2}+2}{\zeta_{t}^{2}}\|x_{t}\|_{2}+\frac{1-t}{\zeta_{t}^{2}}R\Big), 
\end{equation}
where we used $\|Y^{x_{t}}_{t}\|_{2}\leq R$.

\vspace{1ex}

\par\noindent\textbf{Conclusion.}
Finally, substituting~\eqref{eq:lemma:time:derivative:score:5},~\eqref{eq:lemma:time:derivative:score:6} and~\eqref{eq:lemma:time:derivative:score:7} into~\eqref{eq:lemma:time:derivative:score:1} yields
\begin{align*}
&\|\nabla\partial_{t}\log p_{X_{t}}(x_{t})\|_{2} \\
&\leq 4\frac{t^{2}}{(1-t)^{3}}\Big(\frac{(1-t)^{2}}{\zeta_{t}^{4}}R^{2}+\frac{\sigma^{2}}{\zeta_{t}^{2}}\Big)\Big(\frac{t\sigma^{2}}{\zeta_{t}^{2}}\|x_{t}\|_{2}+\frac{(1-t)^2}{\zeta_{t}^{2}}R\Big)+8\frac{t^{2}(1-t)}{\zeta_{t}^{6}}R^{3} \\
&\quad +\Big(\frac{\sigma^{2}+2}{\zeta_{t}^{2}}+\frac{t(1-t)}{\zeta_{t}^{4}}R^{2}\Big)\Big(\frac{1}{\zeta_{t}^{2}}\|x_{t}\|_{2}
+\frac{t}{\zeta_{t}^{2}}R\Big) \\
&\quad +\frac{1}{1-t}\Big(\frac{\sigma^{2}+2}{\zeta_{t}^{2}}+\frac{t}{\zeta_{t}^{4}}R^{2}\Big)\Big(\frac{\sigma^{2}+2}{\zeta_{t}^{2}}\|x_{t}\|_{2}+\frac{1-t}{\zeta_{t}^{2}}R\Big),
\end{align*}
which completes the proof.
\end{proof}

\begin{lemma}\label{lemma:lipschitz:target:score}
Let Assumption~\ref{assumption:Gaussian:convolution} be fulfilled. Then we have 
\begin{equation}
\sup_{x_{1}\in\bbR^{d}}\|\nabla^{2}\log p_{X_{1}}(x_{1})\|_{\rm op}\leq \frac{\sigma^{2}+  R^{2}}{\sigma^{4}}.
\end{equation}
\end{lemma}

\begin{proof}
For $X_1 = Y +\sigma \Xi$, where $Y$ is independent of $\Xi$, we know
by Tweedie's formula, see, e.g.~\cite{Efron2011Tweedie} or~\cite[Lemma 3.2]{daras2024consistent}, that
\begin{equation} \label{lemma:tweedies:gaussian:convolution}
\nabla\log p_{X_1}(x_{1})=-\frac{1}{\sigma^2}x_{1}+\frac{1}{\sigma^2}\bbE[Y|X_{1}=x_{1}].
\end{equation}
This implies with Lemma~\ref{lemma:tweedies:second:order:gaussian:convolution} that
\begin{align}
\nabla^{2}\log p_{X_1}(x_{1})
&=-\frac{1}{\sigma^2} I_{d}+\frac{1}{\sigma^2}\nabla\bbE[Y|X_{1}=x_{1}] \\
&=-\frac{1}{\sigma^2} I_{d}+\frac{1}{\sigma^{4}}\Cov(Y|X_{1}=x_{1}).
\end{align}
Using triangular inequality yields 
\begin{align}
\|\nabla^{2}\log p_{X_1}(x_{1})\|_{\rm op} 
&\leq \frac{1}{\sigma^2}+\frac{1}{\sigma^4}\|\Cov(Y|X_{1}=x_{1})\|_{\rm op} \nonumber \\
&\leq \frac{1}{\sigma^2}+\frac{1}{\sigma^4}\trace\Cov(Y|X_{1}=x_{1}).\label{eq:lemma:lipschitz:target:score:1}
\end{align}
Further, we have 
\begin{align*}
\trace\Cov(Y|X_{1}=x_{1}) 
&= \trace\bbE[YY^{\top}|X_{1}=x_1]-\trace\bbE[Y|X_{1}=x_1]\bbE[Y|X_{1}=x_1]^{\top} \\
&\leq \bbE[\trace YY^{\top}|X_{1}=x_1] = \bbE[\trace Y^{\top}Y|X_{1}=x_1] \\
&= \bbE[Y^{\top}Y|X_{1}=x_1] = \bbE[\|Y\|_{2}^{2}|X_{1}=x_1] \leq R^{2},
\end{align*}
where the second equality holds from the cyclic property of the trace operator, and the last inequality is due to Assumption~\ref{assumption:Gaussian:convolution}. Substituting this into~\eqref{eq:lemma:lipschitz:target:score:1} completes the proof.
\end{proof}

\begin{lemma}\label{lemma:conditional:cov:score}
Let Assumption~\ref{assumption:Gaussian:convolution} be fulfilled. Then we have 
\begin{align*}
&\Cov(\nabla\log p_{X_{1}}(X_{1})|X_{t}=x_{t}) \\
&=\frac{(1-t)^{2}}{\sigma^{2}(\sigma^{2}t^{2}+(1-t)^{2})} I_{d}+\frac{t^{4}}{(\sigma^{2}t^{2}+(1-t)^{2})^{2}}\Cov( Y \mid X_{t})-\frac{1}{\sigma^{4}}\bbE[\Cov(Y|X_{1})|X_{t}=x_{t}],
\end{align*}
Further, it holds 
\begin{align*}
&\trace(\Cov(\nabla\log p_{X_{1}}(X_{1})|X_{t}=x_{t})) 
\le\frac{(1-t)^{2}d}{\sigma^{2}(\sigma^{2}t^{2}+(1-t)^{2})}+\frac{t^{4}R^{2}}{(\sigma^{2}t^{2}+(1-t)^{2})^{2}}.
\end{align*}
\end{lemma}

\begin{proof}
Recalling the Tweedie's formula~\cite{Efron2011Tweedie} in~\eqref{lemma:tweedies:gaussian:convolution}, we have 
\begin{align}
&\Cov(\nabla\log p_{X_{1}}(X_{1})|X_{t}=x_{t}) \\
&=\Cov\Big(\frac{1}{\sigma^2}\bbE[Y-X_{1}|X_{1}]|X_{t}=x_{t}\Big) \nonumber \\
&=\frac{1}{\sigma^{4}}\Cov(\bbE[Y-X_{1}|X_{1}]|X_{t}=x_{t}) \nonumber \\
&=\frac{1}{\sigma^{4}}\Cov(Y-X_{1}|X_{t}=x_{t})-\frac{1}{\sigma^{4}}\bbE[\Cov(Y-X_{1}|X_{1})|X_{t}=x_{t}] \nonumber \\
&=\frac{1}{\sigma^{4}}\Cov(Y-X_{1}|X_{t}=x_{t})-\frac{1}{\sigma^{4}}\bbE[\Cov(Y|X_{1})|X_{t}=x_{t}], \label{eq:lemma:conditional:cov:score:1}
\end{align}
where the first equality holds from~\eqref{lemma:tweedies:gaussian:convolution}, and the third equality is due to the law of total covariance. It remains to focus on $\Cov(Y-X_{1}|X_{t}=x_{t})$. Recall that 
\begin{equation*}
X_{1}\stackrel{\d}{=}Y+\sigma\Xi, \quad \text{and} \quad X_{t}\stackrel{\d}{=}(1-t)X_{0}+tX_{1},
\end{equation*}
where $\Xi\sim\calN(0, I_{d})$ and $X_{0}\sim\calN(0, I_{d})$ are independent. Then it holds that 
\begin{equation*}
(Y-X_{1} \mid X_{t},Y) \stackrel{\d}{=} (Y-X_{1}|X_{t}-tY,Y) \stackrel{\d}{=} (\sigma\Xi|X_{t}-tY) \stackrel{\d}{=} (\sigma\Xi \mid (1-t)X_{0}+t\sigma\Xi),
\end{equation*}
where the second equality invokes the independence between $\Xi$ and $Y$. Since $(\sigma\Xi, (1-t)X_{0}+t\sigma\Xi)$ follows a joint Gaussian distribution, and the marginal distribution of both $\sigma\Xi$ and $(1-t)X_{0}+t\sigma\Xi$ are Gaussian, the conditional distribution $(\sigma\Xi \mid (1-t)X_{0}+t\sigma\Xi)$ is also a Gaussian with mean and covariance as 
\begin{align}
\bbE[Y-X_{1} \mid X_{t},Y]
&=\bbE[\sigma\Xi \mid (1-t)X_{0}+\sigma\Xi]=-\frac{t\sigma^{2}}{\sigma^{2}t^{2}+(1-t)^{2}}(X_{t}-tY), \label{eq:lemma:conditional:cov:score:2} \\
\Cov(Y-X_{1} \mid X_{t},Y)
&=\Cov(\sigma\Xi \mid (1-t)X_{0}+\sigma\Xi)=\frac{(1-t)^{2}\sigma^{2}}{\sigma^{2}t^{2}+(1-t)^{2}} I_{d}. \label{eq:lemma:conditional:cov:score:3}
\end{align}
Then using the law of total covariance, we have 
\begin{align}
&\Cov(Y-X_{1} \mid X_{t}) \nonumber \\
&=\bbE[\Cov(Y-X_{1} \mid X_{t},Y) \mid X_{t}]+\Cov(\bbE[Y-X_{1} \mid X_{t},Y] \mid X_{t}) \nonumber \\
&=\frac{(1-t)^{2}\sigma^{2}}{\sigma^{2}t^{2}+(1-t)^{2}} I_{d}+\Cov\Big(-\frac{t\sigma^{2}}{\sigma^{2}t^{2}+(1-t)^{2}}(X_{t}-tY) \mid X_{t}\Big) \nonumber \\
&=\frac{(1-t)^{2}\sigma^{2}}{\sigma^{2}t^{2}+(1-t)^{2}} I_{d}+\frac{t^{4}\sigma^{4}}{(\sigma^{2}t^{2}+(1-t)^{2})^{2}}\Cov( Y \mid X_{t}), \label{eq:lemma:conditional:cov:score:4}
\end{align}
where the second equality follows from~\eqref{eq:lemma:conditional:cov:score:2} and~\eqref{eq:lemma:conditional:cov:score:3}. Substituting~\eqref{eq:lemma:conditional:cov:score:4} into~\eqref{eq:lemma:conditional:cov:score:1} completes the proof.
\end{proof}

\par Applying the method of variation-of-constants, we have the following solution of the semi-linear ODE.

\begin{lemma}[Variation-of-constants]\label{lemma:variation:constant}
Suppose $\phi$ satisfies the following semi-linear ODE: 
\begin{equation*}
\frac{\d}{\dt}\phi(t,x) = a(t)\phi(t,x)+b(t)N(t,\phi(t,x)), \quad t\in(t_{1},t_{2}).
\end{equation*}
Then we have 
\begin{equation*}
\phi(t_{2},x) = \exp\Big(\int_{t_{1}}^{t_{2}}a(s)\ds\Big)\phi(t_{1},x) + \int_{t_{1}}^{t_{2}}\exp\Big(\int_{t}^{t_{2}}a(s)\ds\Big)b(t)N(t,\phi(t,x))\dt.
\end{equation*}
\end{lemma}

\section{Proofs of Section~\ref{section:method:velocity} } \label{section:proof:method_1}

\begin{customproposition}{\ref{lemma:denoising:langevin:convergence}}
Let Assumption~\ref{assumption:Gaussian:convolution} be fulfilled, and let $X_t$ be given by the linear interpolant~\eqref{eq:stochastic:interporlant}. 
For $t \in (0,1)$ and $\sigma^2 >0$, set
$$\beta_t \coloneqq \frac{t^2}{(1-t)^{2}}+\frac{1}{\sigma^{2}} - \frac{R^2}{\sigma^4}.$$ 
Then, for every $t\in(0,1)$, it holds
\begin{equation*} 
\beta_t \,  I_{d}
\preceq 
-\nabla^2 \log p_{X_1|X_t = x_t}
\preceq
\Big(\beta_t + \frac{R^2}{\sigma^4} \Big)\,  I_{d},
\end{equation*}
and $p_{X_1|X_t = x_t}$ is $\beta_t$-strongly log-concave  for $t \in (T^*,1)$, where
\begin{equation}
T^* =
\left\{
\begin{array}{ll}
0& \text{if } R^{2}\leq\sigma^{2},
\\[1ex]
\frac{1}{2} & \text{if } R^{2}=\sigma^{2}+\sigma^{4},
\\[1ex]
\frac{R^{2}-\sigma^{2}-\sigma^{2}\sqrt{R^{2}-\sigma^{2}}}{R^{2}-\sigma^{2}-\sigma^{4}}&\text{ otherwise}.
\end{array}
\right.
\end{equation}
We have $T^*< \frac12$ if $\sigma^{2}<R^{2}<\sigma^{2}+\sigma^{4}$, and $T^*> \frac12$ if $R^{2}>\sigma^{2}+\sigma^{4}$.

Further, for $t \in (T^*,1)$, it holds
\begin{equation}
\kl \left(p_{Z_s}, p_{X_1|X_t = x_t} \right)
\leq
{\rm e}^{-2\beta_t s} \, 
\kl \left(p_{Z_0},p_{X_1|X_t = x_t} \right),
\end{equation}
where $Z_{s}$,  $s\geq 0$ is defined by the Langevin diffusion~\eqref{eq:langevin:score:estimation}. 
\end{customproposition}

\begin{proof}
Using~\eqref{eq:denoising_score_function},~\eqref{lemma:tweedies:gaussian:convolution}
as well as Lemma~\ref{lemma:tweedies:second:order:gaussian:convolution}, we obtain
\begin{align*} 
\nabla^2\log p_{X_1|X_t=x_t}(x_{1})
&=-\frac{t^2}{(1-t)^{2}} \,  I_{d}+\nabla^2\log p_{X_1}(x_{1}) \nonumber \\
&=-\frac{t^2}{(1-t)^{2}}\,  I_{d}-\frac{1}{\sigma^{2}} \,  I_{d}+\frac{1}{\sigma^{2}}\nabla\bbE[Y|X_{1}=x_{1}] 
\nonumber \\
&=- \Big(\beta_t + \frac{R^2}{\sigma^4} \Big) \,  I_{d}
+
\frac{1}{\sigma^{4}}\Cov(Y|X_{1}= x_{1}). 
\end{align*} 
Since $\|Y\|_{2}\leq R$ almost surely, we have
\begin{equation*}
0\preceq \Cov(Y|X_1=x_1)\preceq R^{2} \,  I_{d},
\end{equation*}
so that
\begin{equation*}
\beta_t \,  I_{d}
\preceq 
-\nabla^2 \log p_{X_1|X_t = x_t}
\preceq
\Big(\beta_t + \frac{R^2}{\sigma^4} \Big)\,  I_{d},
\end{equation*}
Next, we consider for $t\in (0,1)$ the expression
\begin{align*}
(1-t)^2 \sigma^4 \beta_t 
&= t^2 \sigma^4 + (1-t)^2 (\sigma^2 - R^2)\\
&= (\sigma^2 + \sigma^4 - R^2) t^2 - 2t (\sigma^2 - R^2) + \sigma^2 - R^2.
\end{align*}
If $\sigma^{2}\geq R^{2}$, then clearly $\beta_t > 0$.
If $R^{2} = \sigma^{2} + \sigma^4$, then $\beta_t > 0$  if and only if $t \in (\frac12,1)$.\\
Let $\sigma^{2}<R^{2}$. Then the zeros of the above quadratic equation
are
$$
t_{\pm} = \frac{R^2 - \sigma^2 \pm \sigma^2 \sqrt{R^2 - \sigma^2}}{ R^2 - \sigma^2 - \sigma^4} .
$$
If $R^{2} <\sigma^{2}+\sigma^{4}$, 
then $t_+ < 0$ and $t_- \in (0,\frac12)$, so that $\beta_t > 0$ if and only if
$t \in (t_-,1)$. 
If $R^{2}>\sigma^{2}+\sigma^{4}$, 
then $t_+ > 1$ and $t_- \in (\frac12,1)$, so that $\beta_t >0$ if and only if
$t \in (t_-,1)$.
This gives $T^*$ in~\eqref{eq:T}.
The final relation~\eqref{est} follows from the $\beta_t$-strong log-concavity, see Theorem~\ref{thm:conv}.
This completes the proof.
\end{proof}

\begin{customproposition}{\ref{proposition:rescaling:velocity}}
Let $(X_0,X_1)\sim\gamma\otimes p_{X_1}$ and consider the stochastic
interpolant
\[
X_t:=a_tX_0+b_tX_1,
\]
where $a_t$ and $b_t$ are continuously differentiable schedules satisfying
$a_t,b_t>0$ for $t\in(0,1)$. Then the velocity field
\[
u(t,x_t):=\mathbb{E}[\dot X_t\mid X_t=x_t]
\]
satisfies, for every $t\in(0,1)$ and $x_t\in\mathbb{R}^d$,
\[
u(t,x_t)
=
\frac{\dot b_t}{b_t}x_t
+
\frac{a_t(a_t\dot b_t-\dot a_tb_t)}{b_t^2}
\mathbb{E}\big[
    \nabla\log p_{X_1}(X_1)
    \,\big|\,X_t=x_t
\big].
\]
In particular, for $a_t=1-t$ and $b_t=t$, we obtain
\[
u(t,x_t)
=
\frac{1}{t}x_t
+
\frac{1-t}{t^2}
\mathbb{E}\big[
    \nabla\log p_{X_1}(X_1)
    \,\big|\,X_t=x_t
\big],
\]
which is the assertion of Proposition~4.2.
\end{customproposition}

\begin{proof}
Computing the time derivative of the interpolant gives
\[
\dot X_t=\dot a_tX_0+\dot b_tX_1.
\]
Conditionally on $X_t=x_t$, we have
\[
X_0=\frac{x_t-b_tX_1}{a_t},
\]
and consequently
\[
\begin{aligned}
u(t,x_t)
&=
\mathbb{E}\big[\dot X_t\mid X_t=x_t\big]
\\
&=
\frac{\dot a_t}{a_t}x_t
+
\left(
    \dot b_t-\frac{\dot a_tb_t}{a_t}
\right)
\mathbb{E}\big[X_1\mid X_t=x_t\big].
\end{aligned}
\]

By Bayes' theorem, the posterior density is given by
\[
p_{X_1\mid X_t=x_t}(x_1)
=
\frac{1}{p_{X_t}(x_t)}
\gamma_{a_t^2}(x_t-b_tx_1)p_{X_1}(x_1).
\]
Differentiating its logarithm with respect to $x_1$ yields
\[
\nabla_{x_1}\log p_{X_1\mid X_t=x_t}(x_1)
=
\frac{b_t}{a_t^2}(x_t-b_tx_1)
+
\nabla\log p_{X_1}(x_1).
\]
The conditional score has zero mean. Indeed, by Assumption~1, the
conditional density is continuously differentiable and has sufficient
decay at infinity, so that
\[
\begin{aligned}
&\mathbb{E}\big[
    \nabla_{x_1}\log p_{X_1\mid X_t=x_t}(X_1)
    \,\big|\,X_t=x_t
\big]
\\
&\qquad=
\int_{\mathbb{R}^d}
\nabla_{x_1}p_{X_1\mid X_t=x_t}(x_1)\,\mathrm{d}x_1
=0.
\end{aligned}
\]
Taking the conditional expectation in the preceding score identity
therefore gives
\[
0
=
\frac{b_t}{a_t^2}
\left(
    x_t-b_t\mathbb{E}[X_1\mid X_t=x_t]
\right)
+
\mathbb{E}\big[
    \nabla\log p_{X_1}(X_1)
    \,\big|\,X_t=x_t
\big].
\]
Hence,
\[
\mathbb{E}[X_1\mid X_t=x_t]
=
\frac{1}{b_t}x_t
+
\frac{a_t^2}{b_t^2}
\mathbb{E}\big[
    \nabla\log p_{X_1}(X_1)
    \,\big|\,X_t=x_t
\big].
\]
Substituting this expression into the representation of $u(t,x_t)$
above yields
\[
u(t,x_t)
=
\frac{\dot b_t}{b_t}x_t
+
\frac{a_t(a_t\dot b_t-\dot a_tb_t)}{b_t^2}
\mathbb{E}\big[
    \nabla\log p_{X_1}(X_1)
    \,\big|\,X_t=x_t
\big].
\]
Finally, for $a_t=1-t$ and $b_t=t$, this reduces to
\[
u(t,x_t)
=
\frac{1}{t}x_t
+
\frac{1-t}{t^2}
\mathbb{E}\big[
    \nabla\log p_{X_1}(X_1)
    \,\big|\,X_t=x_t
\big],
\]
which completes the proof.
\end{proof}

\section{Proofs of Section~\ref{section:computation:flow} } \label{section:proof:method}

\begin{customlemma}{\ref{lem1}}
Let Assumption~\ref{assumption:Gaussian:convolution} be fulfilled. For every $t\in(0,1)$, it holds 
\begin{equation*}
\|u(t,x_{t})\|_{2}\leq B(1+\|x_{t}\|_{2}), \quad x_{t}\in\bbR^{d},
\end{equation*}
where $B$ is a constant only depending on $\sigma$ and $R$.
\end{customlemma}

\begin{proof}
According to~\eqref{eq:velocity} and Lemma~\ref{lemma:representation:conditional}, we have 
\begin{align*}
u(t,x_{t})
&=-\frac{1}{1-t} x_{t}+\frac{1}{1-t}\bbE[X_{1}|X_{t}=x_{t}] \\
&=-\frac{1}{1-t} x_{t}+\frac{1}{1-t}\frac{t}{(1-t)^{2}} \sigma_t^{2} x_{t}
+\frac{1}{1-t}\frac{\sigma_t^2}{\sigma^{2}} \bbE[Y^{x_{t}}_t] \\
&=\frac{\sigma^{2}t-(1-t)}{\sigma^{2}t^2+(1-t)^{2}}x_{t}+\frac{1-t}{\sigma^{2}t^{2}+(1-t)^2}\bbE[Y^{x_{t}}_t].
\end{align*}
Then using triangular inequality yields
\begin{align*}
\|u(t,x_{t})\|_{2}
&\leq \frac{|\sigma^{2}t-(1-t)|}{\sigma^{2}t^2+(1-t)^{2}}\|x_{t}\|_{2}+\frac{1-t}{\sigma^{2}t^{2}+(1-t)^2}\|\bbE[Y^{x_{t}}_t]\|_{2} \\
&\leq \frac{|\sigma^{2}t-(1-t)|}{\sigma^{2}t^2+(1-t)^{2}}\|x_{t}\|_{2}+\frac{1-t}{\sigma^{2}t^{2}+(1-t)^2}\bbE[\|Y^{x_{t}}_t]\|_{2}] \\
&\leq \frac{|\sigma^{2}t-(1-t)|}{\sigma^{2}t^2+(1-t)^{2}}\|x_{t}\|_{2}+\frac{1-t}{\sigma^{2}t^{2}+(1-t)^2}R,
\end{align*}
where the second inequality holds from Jensen's inequality, and the last inequality is due to Lemma~\ref{lemma:representation:conditional}. This completes the proof.
Compare to~\cite[Proposition 3.2]{ding2024characteristic}.
\end{proof}

\begin{customtheorem}{\ref{thm}}
Let Assumption~\ref{assumption:Gaussian:convolution} be fulfilled. For every $t\in(0,1)$, it holds 
\begin{equation*}
\|\nabla u(t,x_t)\|_{\mathrm{op}}\leq G,
\end{equation*}
where $G$ is a constant only depending on $R$ and $\sigma$.
\end{customtheorem}

\begin{proof}
According to~\eqref{eq:velocity} and Lemmas~\ref{lemma:tweedies:second:order} and~\ref{lemma:representation:conditional}, we have 
\begin{align*}
\nabla u(t,x_{t})
&=-\frac{1}{1-t}\mathrm{I}_{d}+\frac{1}{1-t}\nabla\bbE[X_{1}|X_{t}=x_{t}] \\
&=-\frac{1}{1-t}\mathrm{I}_{d}+\frac{1}{1-t}\frac{t}{(1-t)^2}\Cov(X_{1}|X_{t}=x_{t}) \\
&=-\frac{1}{1-t}\mathrm{I}_{d}+\frac{1}{1-t}\frac{t}{(1-t)^2}\frac{\sigma_t^4}{\sigma^{4}}\Cov(Y^{x_{t}}_t) + \frac{1}{1-t}\frac{t}{(1-t)^2}\sigma_t^2  I_d \\
&=\frac{\sigma^{2}t-(1-t)}{\sigma^{2}t^2+(1-t)^{2}} I_{d}+\frac{(1-t)t}{(\sigma^{2}t^{2}+(1-t)^2)^{2}}\Cov(Y^{x_{t}}_t).
\end{align*}
Then using triangular inequality yields
\begin{align*}
\|\nabla u(t,x_{t})\|_{\mathrm{op}}
&\leq\frac{|\sigma^{2}t-(1-t)|}{\sigma^{2}t^2+(1-t)^{2}}+\frac{(1-t)t}{(\sigma^{2}t^{2}+(1-t)^2)^{2}}\|\Cov(Y^{x_{t}}_t)\|_{\mathrm{op}} \\
&\leq\frac{|\sigma^{2}t-(1-t)|}{\sigma^{2}t^2+(1-t)^{2}}+\frac{(1-t)t}{(\sigma^{2}t^{2}+(1-t)^2)^{2}}R^2,
\end{align*}
where the first inequality holds from the triangular inequality, and the last inequality is due to Lemma~\ref{lemma:representation:conditional}. This completes the proof. 
See also~\cite[Proposition 3.5]{ding2024characteristic}.
\end{proof}
\vspace{0.2cm}

\begin{remark}[Euler forward coincides with exponential integrator for vanilla velocity estimator]\label{remark:euler:ei}
For
$
D(s,y) := \mathbb{E}[X_1 \mid X_s = y],
$
we consider the ODE
\[
\frac{d}{dt}\psi(t,x)
= u(t,\psi(t,x))
= -\frac{1}{1-t}\psi(t,x)
+ \frac{1}{1-t}D\bigl(t,\psi(t,x)\bigr),
\]
i.e., we consider the vanilla velocity estimator.
Discretizing this via the explicit Euler method with step size \(h>0\), we obtain for
$
\hat{\psi}_m := \psi(t_m,x)
$
that
\begin{align*}
\hat{\psi}_{m+1}
&= \hat{\psi}_m + h\,u(t_m,\hat{\psi}_m) \\
&= \hat{\psi}_m
   - \frac{h}{1-t_m}\hat{\psi}_m
   + \frac{h}{1-t_m}D(t_m,\hat{\psi}_m) \\
&=
\left(
1-\frac{t_{m+1}-t_m}{1-t_m}
\right)\hat{\psi}_m
+ \frac{h}{1-t_m}D(t_m,\hat{\psi}_m) \\
&=
\frac{1-t_{m+1}}{1-t_m}\hat{\psi}_m
+ \frac{h}{1-t_m}D(t_m,\hat{\psi}_m).
\end{align*}
On the other hand, using the method of exponential integrators, we get
\[
\hat{\psi}_{m+1}
=
\exp\!\left(
\int_{t_m}^{t_{m+1}}
-\frac{1}{1-s}\, \d s
\right)
\hat{\psi}_m
+
\int_{t_m}^{t_{m+1}}
\exp\!\left(
\int_s^{t_{m+1}}
-\frac{1}{1-r}\, \d r
\right)
\frac{D(s,\psi(s,x))}{1-s}\, \d s.
\]
The first summand simplifies to
\begin{align*}
\exp\!\left(
\int_{t_m}^{t_{m+1}}
-\frac{1}{1-s}\, \d s
\right)
\hat{\psi}_m
&=
\exp\!\left(
\ln\!\left(
\frac{1-t_{m+1}}{1-t_m}
\right)\right)
\hat{\psi}_m =
\frac{1-t_{m+1}}{1-t_m}\hat{\psi}_m.
\end{align*}
For the second summand, approximating \(D\) by its value at the left boundary of the
time interval (as done in the paper), we obtain
\begin{align*}
&\int_{t_m}^{t_{m+1}}
\exp\!\left(
\int_s^{t_{m+1}}
-\frac{1}{1-r}\, \d r
\right)
\frac{D(s,\psi(s,x))}{1-s}\, \d s 
\approx
D(t_m,\hat{\psi}_m)
\int_{t_m}^{t_{m+1}}
\frac{1-t_{m+1}}{(1-s)^2}\, \d s \\
&\qquad=
(1-t_{m+1})D(t_m,\hat{\psi}_m)
\int_{t_m}^{t_{m+1}}
\frac{1}{(1-s)^2}\, \d s \\
&\qquad=
(1-t_{m+1})D(t_m,\hat{\psi}_m)
\left(
\frac{1}{1-t_{m+1}}
-
\frac{1}{1-t_m}
\right) \\
&\qquad=
\left(
1-\frac{1-t_{m+1}}{1-t_m}
\right)
D(t_m,\hat{\psi}_m) =
\frac{t_{m+1}-t_m}{1-t_m}
D(t_m,\hat{\psi}_m) 
=
\frac{h}{1-t_m}
D(t_m,\hat{\psi}_m).
\end{align*}
Hence, in both cases, we obtain the exact same expression.
\end{remark}

\section{Proofs in Section~\ref{section:convergence}}
\label{section:proof:convergence}

\subsection{Convergence of Langevin Monte Carlo}

We first show a convergence rate of Langevin Monte Carlo with a log-concave invariant distribution.

\begin{theorem}[Convergence rate in Wasserstein distance~{\cite[Theorem 4.1.2]{Chewi2025log}}]\label{lemma:lmc}
Let $p\in C^{2}(\mathbb{R}^{d})$ be a probability density. Suppose there exist $0<\alpha\leq L<\infty$ such that $\alpha I_{d}\preceq \nabla^{2}\log p(x) \preceq LI_{d}$ for any $x\in\mathbb{R}^{d}$. Let $\widehat{p}_{t}$ be the marginal density of Langevin Monte Carlo with exact score
\begin{equation*}
\what{X}_{(k+1)h}=\what{X}_{kh}+h\nabla\log p(\what{X}_{kh})+\sqrt{2h}\xi_{k+1}, \quad \xi_{k+1}\sim\mathcal{N}(0,I_d).
\end{equation*}
Suppose the step size $h>0$ satisfies $\alpha^{-1}L^{2}h\ll 1$, then we have
\begin{equation*}
W_{2}^{2}(\widehat{p}_{Kh},p) \leq \exp(-\alpha Kh)W_{2}^{2}(\widehat{p}_{0},p)+\frac{c}{\alpha^{2}}dL^{2}h,
\end{equation*}
where $c>0$ is an absolute constant.
\end{theorem}

\par Let $p\in C^{2}(\mathbb{R}^{d})$ be a probability density satisfying log-Sobolev inequality with constant $C_{\rm LSI}(p)$, and $\log p$ is $L$-smooth.  Consider the Langevin Monte Carlo with a score estimator $\what{s}$, and initialization $\what{X}_{0}\sim p_{0}$ given by
\begin{equation*}
\what{X}_{(k+1)h}=\what{X}_{kh}+h\what{s}(\what{X}_{kh})+\sqrt{2h}\xi_{k+1}, \quad \xi_{k+1}\sim\mathcal{N}(0,\rm{I}_d).
\end{equation*}
The associated interpolated process is defined as
\begin{equation}\label{eq:lmc:inexact:score}
\d\what{X}_{t}=\what{s}(\what{X}_{kh})\dt+\sqrt{2}\d B_{t}, \quad t\in[kh,(k+1)h],
\end{equation}
where $B_{t}$ represents a $d$-dimensinal Brownian motion. Denote by $\widehat{p}_{t}$ the marginal density of $\widehat{X}_{t}$, denote by $\widehat{p}_{kh|t}(\cdot|x_{t})$ the conditional density of $\widehat{X}_{kh}$ given $\widehat{X}_{t}=x_{t}$, and denote by $\widehat{p}_{kh,t}$ the joint density of $\widehat{X}_{kh}$ and $\widehat{X}_{t}$. The goal of this subsection is to establish the error between $\widehat{p}_{t}$ and $p$ in KL-divergence.

\begin{theorem}\label{lemma:lmc:score}
Let $p\in C^{2}(\mathbb{R}^{d})$ be a probability density satisfying log-Sobolev inequality with constant $C_{\rm LSI}(p)$, and $\log p$ is $L$-smooth. Let $\widehat{p}_{t}$ be the marginal density of Langevin Monte Carlo with inexact score~\eqref{eq:lmc:inexact:score}. Suppose the step size $h>0$ satisfies $16L^{2}hC_{\rm LSI}(p)\leq 1$, then we have
\begin{align}
\kl(\widehat{p}_{Kh},p)
&\leq \exp\Big(-\frac{Kh}{C_{\rm LSI}(p)}\Big)\kl(\widehat{p}_{0},p) + 32dL^{2}hC_{\rm LSI}(p) \\
&\quad +\sum_{j=0}^{K-1}\exp\Big(-\frac{K-1-j}{C_{\rm LSI}(p)}h\Big)12h\int\widehat{p}_{jh}(x_{jh})\|\what{s}(x_{jh})-\nabla\log p(x_{jh})\|_{2}^{2}\d x_{jh}.
\end{align}
\end{theorem}

\begin{proof}
The proof is divided into four steps.
\\[1ex]
\emph{Step 1. Differential equation of KL-divergence.}
According to~\cite[Lemma A.1]{Lee2022Convergence}, the Fokker-Planck equation of the Langevin Monte Carlo with score estimator~\eqref{eq:lmc:inexact:score} is given as 
\begin{equation}\label{eq:lemma:lmc:score:kl:1}
\partial_{t}\widehat{p}_{t}(x_t)=\nabla\cdot\Big\{\widehat{p}_{t}\Big(-\int\what{s}(x_{kh})\widehat{p}_{kh|t}(x_{kh}|x_{t})\d x_{kh}+\nabla\log\widehat{p}_{t}(x_t)\Big)\Big\},
\end{equation}
for any $x\in\mathbb{R}^{d}$ and $t\in[kh,(k+1)h]$. Then it follows from the Fokker-Planck equation~\eqref{eq:lemma:lmc:score:kl:1} and Gauss-Green formula that 
\begin{align}
&\partial_{t}\kl(\widehat{p}_{t},p) \\
&=\partial_{t}\int\widehat{p}_{t}(x_{t})\log\frac{\widehat{p}_{t}(x_{t})}{p(x_t)}\d x_{t} \\
&=\int \partial_{t}\widehat{p}_{t}(x_{t})\log\frac{\widehat{p}_{t}(x_{t})}{p(x_{t})}\d x_{t}+\int\partial_{t}\widehat{p}_{t}(x_{t})\d x_{t} \\
&=\int \partial_{t}\widehat{p}_{t}(x_{t})\log\frac{\widehat{p}_{t}(x_{t})}{p(x_{t})}\d x_{t} \\
&=\int\nabla\cdot\Big\{\widehat{p}_{t}(x_{t})\Big(-\int\what{s}(x_{kh})\widehat{p}_{kh|t}(x_{kh}|x_{t})\d x_{kh}+\nabla\log\widehat{p}_{t}(x_{t})\Big)\Big\}\log\frac{\widehat{p}_{t}(x_{t})}{p(x)}\d x_{t} \\
&=\int\nabla\cdot\Big\{\widehat{p}_{t}(x_{t})\Big(\nabla\log p(x_{t})-\int\what{s}(x_{kh})\widehat{p}_{kh|t}(x_{kh}|x_{t})\d x_{kh}\Big)\Big\}\log\frac{\widehat{p}_{t}(x_{t})}{p(x_{t})}\d x_{t} \\
&\quad +\int\nabla\cdot\Big\{\widehat{p}_{t}(x_{t})\nabla\log\frac{\widehat{p}_{t}(x_{t})}{p(x_{t})}\Big\}\log\frac{\widehat{p}_{t}(x_{t})}{p(x_{t})}\d x_{t} \\
&=-\int\widehat{p}_{t}(x_{t})\Big\langle\nabla\log p(x_{t})-\int\what{s}(x_{kh})\widehat{p}_{kh|t}(x_{kh}|x_{t})\d x_{kh},\nabla\log\frac{\widehat{p}_{t}(x_{t})}{p(x_{t})}\Big\rangle\d x_{t} \\
&\quad -\int\widehat{p}_{t}(x_{t})\Big\langle\nabla\log\frac{\widehat{p}_{t}(x_{t})}{p(x_{t})},\nabla\log\frac{\widehat{p}_{t}(x_{t})}{p(x_{t})}\Big\rangle\d x_{t} \\
&= \iint\widehat{p}_{kh,t}(x_{kh},x_{t})\Big\langle\what{s}(x_{kh})-\nabla\log p(x_{t}),\nabla\log\frac{\widehat{p}_{t}(x_{t})}{p(x_{t})}\Big\rangle\d x_{kh}\d x_{t} \\
&\quad -\int\widehat{p}_{t}(x_{t})\|\nabla\log\frac{\widehat{p}_{t}(x_{t})}{p(x_{t})}\|_{2}^{2}\d x_{t} \\
&\leq \underbrace{\iint\widehat{p}_{kh,t}(x_{kh},x_{t})\|\what{s}(x_{kh})-\nabla\log p(x_{t})\|_{2}^{2}\d x_{kh}\d x_{t}}_{A_{1}} \\
&\quad -\frac{3}{4}\underbrace{\int\widehat{p}_{t}(x_{t})\|\nabla\log\frac{\widehat{p}_{t}(x_{t})}{p(x_{t})}\|_{2}^{2}\d x_{t}}_{A_{2}}, \label{eq:lemma:lmc:score:kl:2}
\end{align}
where the third equality holds from $\int\partial_{t}\widehat{p}_{t}\d x=\partial_{t}\int \widehat{p}_{t}\d x=0$, the sixth equality is due to Gauss-Green formula, and the last inequality invokes $\langle a,b\rangle\leq\|a\|_{2}^{2}+\frac{1}{4}\|b\|_{2}^{2}$. Here the terms $A_{1}$ and $A_{2}$ in~\eqref{eq:lemma:lmc:score:kl:2} will be bounded in the rest of the proof.
\\[1ex]
\emph{Step 2. Discretization error and score estimation error.} 
In this step, we aim to bound $A_{1}$ in~\eqref{eq:lemma:lmc:score:kl:2}. Since $\log p$ is $L$-smooth, we have
\begin{align}
&\iint\widehat{p}_{kh,t}(x_{kh},x_{t})\|\what{s}(x_{kh})-\nabla\log p(x_{t})\|_{2}^{2}\d x_{kh}\d x_{t} \nonumber \\
&\leq 2\int\widehat{p}_{kh}(x_{kh})\|\what{s}(x_{kh})-\nabla\log p(x_{kh})\|_{2}^{2}\d x_{kh} \nonumber \\
&\quad +2\iint\widehat{p}_{kh,t}(x_{kh},x_{t})\|\nabla\log p(x_{kh})-\nabla\log p(x_{t})\|_{2}^{2}\d x_{kh}\d x_{t} \nonumber \\
&\leq 2\int\widehat{p}_{kh}(x_{kh})\|\what{s}(x_{kh})-\nabla\log p(x_{kh})\|_{2}^{2}\d x_{kh} \\
&\quad +2L^{2}\iint\widehat{p}_{kh,t}(x_{kh},x_{t})\|x_{kh}-x_{t}\|_{2}^{2}\d x_{kh}\d x_{t}, \label{eq:lemma:lmc:score:kl:3}
\end{align}
where the first inequality is due to the triangular inequality. We now turn to focus on the second term in~\eqref{eq:lemma:lmc:score:kl:3}. According to the interpolated process~\eqref{eq:lmc:inexact:score}, 
\begin{equation}\label{eq:lemma:lmc:score:kl:4}
\what{X}_{t}=\what{X}_{kh}+(t-kh)\what{s}(\what{X}_{kh})+\sqrt{2(t-kh)}\xi, \quad \xi\sim\mathcal{N}(0,\rm{I}_d),
\end{equation}
where $\xi$ is independent of $\what{X}_{kh}$. Consequently, we have 
\begin{align}
&\iint\widehat{p}_{kh,t}(x_{kh},x_{t})\|x_{kh}-x_{t}\|_{2}^{2}\d x_{kh}\d x_{t} \nonumber \\
&= \iint\widehat{p}_{kh}(x_{kh})\gamma(\xi)\|(t-kh)\what{s}(x_{kh})+\sqrt{2(t-kh)}\xi\|_{2}^{2}\d x_{kh}\d\xi \nonumber \\
&= (t-kh)^{2}\int\widehat{p}_{kh}(x_{kh})\|\what{s}(x_{kh})\|_{2}^{2}\d x_{kh} + 2(t-kh)\int\gamma(\xi)\|\xi\|_{2}^{2}\d\xi \nonumber \\
&\quad +2\iint\widehat{p}_{kh}(x_{kh})\gamma(\xi)\langle (t-kh)\what{s}(x_{kh}),\sqrt{2(t-kh)}\xi\rangle\d x_{kh}\d\xi \nonumber \\
&= (t-kh)^{2}\int\widehat{p}_{kh}(x_{kh})\|\what{s}(x_{kh})\|_{2}^{2}\d x_{kh} + 2(t-kh)\int\gamma(\xi)\|\xi\|_{2}^{2}\d\xi \nonumber \\
&\leq 2(t-kh)^{2}\int\widehat{p}_{kh}(x_{kh})\|\what{s}(x_{kh})-\nabla\log p(x_{kh})\|_{2}^{2}\d x_{kh} \nonumber \\
&\quad + 2(t-kh)^{2}\int\widehat{p}_{kh}(x_{kh})\|\nabla\log p(x_{kh})\|_{2}^{2}\d x_{kh} + 2d(t-kh) \nonumber \\
&\leq 2(t-kh)^{2}\int\widehat{p}_{kh}(x_{kh})\|\what{s}(x_{kh})-\nabla\log p(x_{kh})\|_{2}^{2}\d x_{kh} \nonumber \\
&\quad + 16L^{2}C_{\rm LSI}(p)(t-kh)^{2}\kl(\widehat{p}_{kh},p)+4dL(t-kh)^{2} + 2d(t-kh), \label{eq:lemma:lmc:score:kl:5}
\end{align}
where the first equality holds from~\eqref{eq:lemma:lmc:score:kl:4}, the third equality is due to the independence of $\xi$ and $\what{X}_{kh}$, the first inequality invokes the triangular inequality, and the last inequality is owing to~\cite[Lemma 10]{Vempala2019Rapid}. Substituting~\eqref{eq:lemma:lmc:score:kl:5} into~\eqref{eq:lemma:lmc:score:kl:3} yields
\begin{align*}
&\iint\widehat{p}_{kh,t}(x_{kh},x_{t})\|\what{s}(x_{kh})-\nabla\log p(x_{t})\|_{2}^{2}\d x_{kh}\d x_{t} \nonumber \\
&\leq (2+4L^{2}(t-kh)^{2})\int\widehat{p}_{kh}(x_{kh})\|\what{s}(x_{kh})-\nabla\log p(x_{kh})\|_{2}^{2}\d x_{kh} \\
&\quad + 32L^{4}C_{\rm LSI}(p)(t-kh)^{2}\kl(\widehat{p}_{kh},p)+8dL^{3}(t-kh)^{2} + 4dL^{2}(t-kh).
\end{align*}
Since~\cite[Lemma E.5]{Lee2022Convergence} shows that $LC_{\rm LSI}\geq 1$ and we assume $16L^{2}hC_{\rm LSI}(p)\leq 1$, we have $Lh\leq\frac{1}{LC_{\rm LSI}(p)}\leq 1$. Using this and $t-kh\leq h$ simplifies the above inequality as 
\begin{align}
&\iint\widehat{p}_{kh,t}(x_{kh},x_{t})\|\what{s}(x_{kh})-\nabla\log p(x_{t})\|_{2}^{2}\d x_{kh}\d x_{t} \nonumber \\
&\leq 32L^{4}h^{2}C_{\rm LSI}(p)\kl(\widehat{p}_{kh},p)+12dL^{2}h \nonumber \\
&\quad +6\int\widehat{p}_{kh}(x_{kh})\|\what{s}(x_{kh})-\nabla\log p(x_{kh})\|_{2}^{2}\d x_{kh}. \label{eq:lemma:lmc:score:kl:6}
\end{align}

\par\noindent
\emph{Step 3. Relations between KL-divergence and Fisher information.}
In this step, we aim to bound $A_{2}$ in~\eqref{eq:lemma:lmc:score:kl:2}, which is known as relative Fisher information of $\widehat{p}_{t}$ with respect to $p$. Since the stationary density $p$ satisfies log-Sobolev inequality with constant $C_{\rm LSI}(p)$, setting $f^{2}=\widehat{p}_{t}/p$ in~\eqref{eq:lsi} implies that the relative Fisher information upper bounds the KL-divergence, i.e., 
\begin{equation}\label{eq:lemma:lmc:score:kl:7}
\kl(\widehat{p}_{t},p) = \mathrm{Ent}_{p}\Big(\frac{\widehat{p}_{t}}{p}\Big)\leq \frac{C_{\rm LSI}(p)}{2}\int\widehat{p}_{t}(x_{t})\|\nabla\log\frac{\widehat{p}_{t}(x_{t})}{p(x_{t})}\|_{2}^{2}\d x_{t}.
\end{equation}

\par\noindent
\emph{Step 4. Recursion and conclusion.}
Substituting~\eqref{eq:lemma:lmc:score:kl:6} and~\eqref{eq:lemma:lmc:score:kl:7} into~\eqref{eq:lemma:lmc:score:kl:2} implies a linear differential inequality of KL-divergence:
\begin{align*}
\partial_{t}\kl(\widehat{p}_{t},p)
&\leq -\frac{3}{2C_{\rm LSI}(p)}\kl(\widehat{p}_{t},p)+32L^{4}h^{2}C_{\rm LSI}(p)\kl(\widehat{p}_{kh},p)+12dL^{2}h \\
&\quad +6\int\widehat{p}_{kh}(x_{kh})\|\what{s}(x_{kh})-\nabla\log p(x_{kh})\|_{2}^{2}\d x_{kh}.
\end{align*}
Multiplying both sides of the inequality by the integrating factor $\exp(\frac{3}{2C_{\rm LSI}(p)}t)$ yields
\begin{align*}
&\partial_{t}\Big(\exp\Big(\frac{3}{2C_{\rm LSI}(p)}t\Big)\kl(\widehat{p}_{t},p)\Big) \\
&\leq \exp\Big(\frac{3}{2C_{\rm LSI}(p)}t\Big)32L^{4}h^{2}C_{\rm LSI}(p)\kl(\widehat{p}_{kh},p) \\
&\quad +\exp\Big(\frac{3}{2C_{\rm LSI}(p)}t\Big)\Big\{12dL^{2}h+6\int\widehat{p}_{kh}(x_{kh})\|\what{s}(x_{kh})-\nabla\log p(x_{kh})\|_{2}^{2}\d x_{kh}\Big\}.
\end{align*}
Integrating both sides of the inequality from $kh$ to $(k+1)h$ gives
\begin{align*}
&\exp\Big(\frac{3(k+1)h}{2C_{\rm LSI}(p)}\Big)\kl(\widehat{p}_{(k+1)h},p)-\exp\Big(\frac{3kh}{2C_{\rm LSI}(p)}\Big)\kl(\widehat{p}_{kh},p) \\
&\leq \frac{2C_{\rm LSI}(p)}{3}\Big\{\exp\Big(\frac{3(k+1)h}{2C_{\rm LSI}(p)}\Big)-\exp\Big(\frac{3kh}{2C_{\rm LSI}(p)}\Big)\Big\}32L^{4}h^{2}C_{\rm LSI}(p)\kl(\widehat{p}_{kh},p) \\
&\quad +\frac{2C_{\rm LSI}(p)}{3}\Big\{\exp\Big(\frac{3(k+1)h}{2C_{\rm LSI}(p)}\Big)-\exp\Big(\frac{3kh}{2C_{\rm LSI}(p)}\Big)\Big\} \\
&\quad \times\Big\{12dL^{2}h+6\int\widehat{p}_{kh}(x_{kh})\|\what{s}(x_{kh})-\nabla\log p(x_{kh})\|_{2}^{2}\d x_{kh}\Big\},
\end{align*}
which implies 
\begin{align*}
&\kl(\widehat{p}_{(k+1)h},p)-\exp\Big(-\frac{3h}{2C_{\rm LSI}(p)}\Big)\kl(\widehat{p}_{kh},p) \\
&\leq \exp\Big(-\frac{3h}{2C_{\rm LSI}(p)}\Big)\frac{2C_{\rm LSI}(p)}{3}\Big\{\exp\Big(\frac{3h}{2C_{\rm LSI}(p)}\Big)-1\Big\}32L^{4}h^{2}C_{\rm LSI}(p)\kl(\widehat{p}_{kh},p) \\
&\quad +\exp\Big(-\frac{3h}{2C_{\rm LSI}(p)}\Big)\frac{2C_{\rm LSI}(p)}{3}\Big\{\exp\Big(\frac{3h}{2C_{\rm LSI}(p)}\Big)-1\Big\} \\
&\quad \times\Big\{12dL^{2}h+6\int\widehat{p}_{kh}(x_{kh})\|\what{s}(x_{kh})-\nabla\log p(x_{kh})\|_{2}^{2}\d x_{kh}\Big\} \\
&\leq 64\exp\Big(-\frac{3h}{2C_{\rm LSI}(p)}\Big)L^{4}h^{3}C_{\rm LSI}(p)\kl(\widehat{p}_{kh},p) \\
&\quad +\Big\{24dL^{2}h^{2} + 12h\int\widehat{p}_{kh}(x_{kh})\|\what{s}(x_{kh})-\nabla\log p(x_{kh})\|_{2}^{2}\d x_{kh}\Big\},
\end{align*}
where the second inequality holds from $\exp(z)-1\leq 2z$ for $z\in[0,1]$ and $\exp(-z)\leq 1$ for $z\geq 0$. Since $1+64L^{4}h^{3}C_{\rm LSI}(p)\leq 1+\frac{h}{2C_{\rm LSI}(p)}\leq\exp(\frac{h}{2C_{\rm LSI}(p)})$, we have the following recursion 
\begin{align*}
\kl(\widehat{p}_{(k+1)h},p)
&\leq \exp\Big(-\frac{3h}{2C_{\rm LSI}(p)}\Big)(1+64L^{4}h^{3}C_{\rm LSI}(p))\kl(\widehat{p}_{kh},p) \\
&\quad +\Big\{24dL^{2}h^{2} + 12h\int\widehat{p}_{kh}(x_{kh})\|\what{s}(x_{kh})-\nabla\log p(x_{kh})\|_{2}^{2}\d x_{kh}\Big\} \\
&\leq \exp\Big(-\frac{h}{C_{\rm LSI}(p)}\Big)\kl(\widehat{p}_{kh},p) \\
&\quad +\Big\{24dL^{2}h^{2} + 12h\int\widehat{p}_{kh}(x_{kh})\|\what{s}(x_{kh})-\nabla\log p(x_{kh})\|_{2}^{2}\d x_{kh}\Big\}.
\end{align*}
Applying this recursion yields
\begin{align*}
&\kl(\widehat{p}_{Kh},p) \\
&\leq \exp\Big(-\frac{Kh}{C_{\rm LSI}(p)}\Big)\kl(\widehat{p}_{0},p) + \Big(1-\exp\Big(-\frac{h}{C_{\rm LSI}(p)}\Big)\Big)^{-1}24dL^{2}h^{2} \\
&\quad +\sum_{j=0}^{K-1}\exp\Big(-\frac{K-1-j}{C_{\rm LSI}(p)}h\Big)12h\int\widehat{p}_{jh}(x_{jh})\|\what{s}(x_{jh})-\nabla\log p(x_{jh})\|_{2}^{2}\d x_{jh} \\
&\leq \exp\Big(-\frac{Kh}{C_{\rm LSI}(p)}\Big)\kl(\widehat{p}_{0},p) + 32dL^{2}hC_{\rm LSI}(p) \\
&\quad +\sum_{j=0}^{K-1}\exp\Big(-\frac{K-1-j}{C_{\rm LSI}(p)}h\Big)12h\int\widehat{p}_{jh}(x_{jh})\|\what{s}(x_{jh})-\nabla\log p(x_{jh})\|_{2}^{2}\d x_{jh},
\end{align*}
where the last inequality holds from $1-\exp(-\frac{h}{C_{\rm LSI}(p)})\geq\frac{3}{4}\frac{h}{C_{\rm LSI}(p)}$ as $\frac{h}{C_{\rm LSI}(p)}\in(0,\frac{1}{4}]$. This completes the proof.
\end{proof}

\subsection{Error Bounds of Velocity Estimation}
\label{section:proof:convergence:velocity}

\begin{customlemma}{\ref{lemma:velocity:estimation}}
Let Assumption~\ref{assumption:Gaussian:convolution} be fulfilled. For $t\in(T^{*},1)$ with $T^{*}$ in~\eqref{eq:T}, let $\what{U}(t,\cdot)$ be the velocity estimator defined in~\eqref{eq:velocity:estimator}, and let $\widehat{D}(t,\cdot)$ be the denoiser estimator defined in~\eqref{eq:denoiser:estimator}. Then, for any $\delta\in(0,1)$ and every $x_{t}\in\bbR^{d}$, we have
\begin{equation}
\bbE\Big[\|D(t,x_{t})-\what{D}(t,x_{t})\|_{2}^{2}\Big]
\lesssim\delta^{2},
\end{equation}
and, as a consequence, 
\begin{equation}
\bbE\Big[\|u(t,x_{t})-\what{U}(t,x_{t})\|_{2}^{2}\Big]
\lesssim
\frac{\delta^{2}}{(1-t)^{2}},
\end{equation}
provided that, with $\beta_t$ in~\eqref{beta_t}  the following relations are fulfilled
\begin{align*}
&n=\Theta\Big(\frac{(1-t)^{2}((1-t)^{2}R^2 +d\sigma^{4}t^{2}+d\sigma^{2}(1-t)^{2})}{(\sigma^{2}t^{2}+(1-t)^{2})^{2}\delta^{2}}\Big), \quad \eta=\Theta\Big(\frac{\sigma^{8}\beta_{t}^{2}\delta^{2}}{d(\sigma^{4}\beta_{t}+R^{2})^{2}}\Big), \\
&K=\Theta\Big(\frac{d}{\beta_{t}\delta^{2}}\Big(1+\frac{R^{2}}{\sigma^{4}\beta_{t}}\Big)^{2}\log\Big(\frac{W_{2}^{2}(p_{Z_{0}},p_{X_{1}|X_{t}=x_{t}})}{\delta^{2}}\vee 1\Big)\Big).
\end{align*}
\end{customlemma}

\begin{proof} Let $\bar{Z}_{K\eta}^{i}$, $i = 1,\ldots,n$ be computed by~\eqref{eq:langevin:monte:carlo}.
Let $Z^{1},\ldots,Z^{n}$  be i.i.d. random variables with law $p_{X_{1}|X_{t}=x_{t}}$, such that $(Z^{i},\bar{Z}_{K\eta}^{i})$ is an optimal coupling of $p_{X_{1}|X_{t}=x_{t}}$ and $p_{\bar{Z}_{K\eta}}$ for each $i = 1,\ldots,n$. This means 
\begin{equation}\label{eq:lemma:velocity:estimation:1}
\bbE\Big[\|Z^{i}-\bar{Z}_{K\eta}^{i}\|_{2}^{2}\Big]= W_{2}^{2}\Big(p_{X_{1}|X_{t}=x_{t}},p_{\bar{Z}_{K\eta}}\Big), \quad 1\leq i\leq n.
\end{equation}
According to the definition of the velocity estimator~\eqref{eq:velocity:estimator}, we obtain
\begin{equation*}
\bbE\Big[\|u(t,x_{t})-\what{U}(t,x_{t})\|_{2}^{2}\Big]
= \frac{1}{(1-t)^2} \bbE\Big[\|\bbE[X_{1}|X_{t}=x_{t}]-\frac{1}{n}\sum_{i=1}^{n}\bar{Z}_{K\eta}^{i}\|_{2}^{2}\Big],
\end{equation*}
so that it suffices to estimate
\begin{align}
&\bbE\Big[\|\bbE[X_{1}|X_{t}=x_{t}]-\frac{1}{n}\sum_{i=1}^{n}\bar{Z}_{K\eta}^{i}\|_{2}^{2}\Big] \nonumber \\
&=\bbE\Big[\|\bbE[X_{1}|X_{t}=x_{t}]-\frac{1}{n}\sum_{i=1}^{n}Z^{i}+\frac{1}{n}\sum_{i=1}^{n}(Z^{i}-\bar{Z}_{K\eta}^{i})\|_{2}^{2}\Big] \nonumber \\
&\leq 
2 \bbE\Big[\|\bbE[X_{1}|X_{t}=x_{t}]-\frac{1}{n}\sum_{i=1}^{n}Z^{i}\|_{2}^{2}\Big] 
+ 2 \bbE\Big[\|\frac{1}{n}\sum_{i=1}^{n}(Z^{i}-\bar{Z}_{K\eta}^{i})\|_{2}^{2}\Big]. \label{eq:lemma:velocity:estimation:2}
\end{align}
For the first summand in~\eqref{eq:lemma:velocity:estimation:2}, by a similar argument as the proof of~\cite[Proposition 3.1]{he2024zeroth} and~\cite[Theorem 1]{huang2024reverse}, we obtain 
by independence of the $Z^i$, $i=1,\ldots,n$ that
\begin{align}
\bbE\Big[\|\bbE[X_{1}|X_{t}=x_{t}]-\frac{1}{n}\sum_{i=1}^{n}Z^{i}\|_{2}^{2}\Big] 
&=\frac{1}{n^{2}}\bbE\Big[\|\sum_{i=1}^{n}(\bbE[X_{1}|X_{t}=x_{t}]-Z^{i})\|_{2}^{2}\Big] \nonumber \\
&=\frac{1}{n^{2}}\sum_{i=1}^{n}\trace(\Cov(X_{1}|X_{t}=x_{t}))\\
&\leq\frac{1}{n}\Big(\frac{\sigma_t^4}{\sigma^{4}}R^2+d\sigma_{t}^{2}\Big), \label{eq:lemma:velocity:estimation:3}
\end{align}
where the last inequality follows from Lemma~\ref{lemma:representation:conditional}. 
For the second summand in~\eqref{eq:lemma:velocity:estimation:2}, we conclude
\begin{align}
\bbE\Big[\|\frac{1}{n}\sum_{i=1}^{n}(Z^{i}-\bar{Z}_{K\eta}^{i})\|_{2}^{2}\Big] 
&=\frac{1}{n^{2}}\sum_{i=1}^{n}\sum_{j=1}^{n}\bbE\Big[\langle Z^{i}-\bar{Z}_{K\eta}^{i},Z^{j}-\bar{Z}_{K\eta}^{j} \rangle\Big] \nonumber \\
&\leq \frac{1}{n^{2}}\sum_{i=1}^{n}\sum_{j=1}^{n}\bbE^{\frac{1}{2}}\Big[\|Z^{i}-\bar{Z}_{K\eta}^{i}\|_{2}^{2}\Big]\bbE^{\frac{1}{2}}\Big[\|Z^{j}-\bar{Z}_{K\eta}^{j}\|_{2}^{2}\Big] \nonumber \\
&\leq W_{2}^{2}\Big(p_{X_{1}|X_{t}=x_{t}},p_{\bar{Z}_{K\eta}}\Big). \label{eq:lemma:velocity:estimation:4}
\end{align}
It remains to estimate the Wasserstein-2 distance in~\eqref{eq:lemma:velocity:estimation:4}. By Theorem~\ref{lemma:lmc} with
$\nabla \log p = \widehat s$, $p= p_{X_1|X_t = x_t}$, $\alpha=\beta_{t}$, $L=\beta_t + \frac{R^2}{\sigma^4}$ and $h =\eta$
that
\begin{align}
W_{2}^{2}\Big(p_{X_{1}|X_{t}=x_{t}},p_{\bar{Z}_{K\eta}}\Big)
&\lesssim 
\exp\Big(-\beta_{t}K \eta\Big)W_{2}^{2}\Big(p_{Z_{0}},p_{X_{1}|X_{t}=x_{t}}\Big)+d\Big(1+\frac{R^{2}}{\sigma^{4}\beta_{t}}\Big)^{2}\eta \label{eq:lemma:velocity:estimation:5}
\end{align}
for $\eta \le \frac{\sigma^{8}\beta_{t}^{2}}{d(\sigma^{4}\beta_{t}+R^{2})^{2}}$. Substituting~\eqref{eq:lemma:velocity:estimation:3},~\eqref{eq:lemma:velocity:estimation:4}, and~\eqref{eq:lemma:velocity:estimation:5} into~\eqref{eq:lemma:velocity:estimation:2} yields
\begin{align*}
&\bbE\Big[\|\bbE[X_{1}|X_{t}=x_{t}]-\frac{1}{n}\sum_{i=1}^{n}\bar{Z}_{K\eta}^{i}\|_{2}^{2}\Big] \nonumber \\
&\lesssim 
\frac{1}{n}\Big(\frac{\sigma_t^4}{\sigma^{4}}R^2+d\sigma_{t}^{2}\Big)+\exp\Big(-\beta_{t}K\eta\Big)W_{2}^{2}\Big(p_{Z_{0}},p_{X_{1}|X_{t}=x_{t}}\Big)+d\Big(1+\frac{R^{2}}{\sigma^{4}\beta_{t}}\Big)^{2}\eta.
\end{align*} 
Setting the first and third summand to  $\delta^{2}$, we obtain
\begin{align*}
&n=\frac{\sigma_{t}^{4}R^2+d\sigma_{t}^{2}\sigma^{4}}{\sigma^{4}\delta^{2}}, \quad 
\eta=\frac{\sigma^{8}\beta_{t}^{2}\delta^{2}}{d(\sigma^{4}\beta_{t}+R^{2})^{2}}. 
\end{align*}
For the second term, if $W_{2}^{2}(p_{Z_{0}},p_{X_{1}|X_{t}=x_{t}})\leq\delta^{2}$, which means the initial distribution $p_{Z_{0}}$ has satisfied the accuracy tolerance, then we set $K=0$. 
On the other hand, if $W_{2}^{2}(p_{Z_{0}},p_{X_{1}|X_{t}=x_{t}})>\delta^{2}$, then setting the second term to $\delta^2$, gives
\begin{equation*}
K = \frac{ d}{\beta_{t}\delta^{2}}\Big(1+\frac{R^{2}}{\sigma^{4}\beta_{t}}\Big)^{2}\log\Big(\frac{W_{2}^{2}(p_{Z_{0}},p_{X_{1}|X_{t}=x_{t}})}{\delta^{2}}\Big)>0.
\end{equation*}
This completes the proof.
\end{proof}

\begin{customlemma}{\ref{lemma:velocity:estimation:stable}}
Let Assumption~\ref{assumption:Gaussian:convolution} be fulfilled. For $t\in(T^{*},1)$ with $T^{*}$ in~\eqref{eq:T}, let $\what{U}_{\mathrm{stab}}(t,\cdot)$ be the velocity estimator defined in~\eqref{eq:stable:velocity:estimator}, and let $\widehat{F}(t,\cdot)$ defined as~\eqref{eq:stable:F:estimator}. Then, for any $\delta\in(0,1)$ and every $x_{t}\in\bbR^{d}$, we have 
\begin{equation}
\bbE\Big[\|F(t,x_{t})-\what{F}(t,x_{t})\|_{2}^{2}\Big]
\lesssim (1-t)^{2}\delta^{2},
\end{equation}
and, as a consequence, 
\begin{equation}
\bbE\Big[\|u(t,x_{t})-\what{U}_{\mathrm{stab}}(t,x_{t})\|_{2}^{2}\Big]
\lesssim
\frac{(1-t)^{2}}{t^{4}}\delta^{2},
\end{equation}
provided that, with $\beta_t$ in~\eqref{beta_t}  the following relations are fulfilled 
\begin{align*}
&n=\Theta\Big(\frac{d(1-t)^{2}(\sigma^{2}t^{2}+(1-t)^{2})+\sigma^{2}t^{4}R^{2}}{\sigma^{2}(\sigma^{2}t^{2}+(1-t)^{2})^{2}\delta^{2}}\Big), \quad \eta=\Theta\Big(\frac{\sigma^{8}\beta_{t}^{2}\delta^{2}}{d(\sigma^{4}\beta_{t}+R^{2})^{2}}\Big(\frac{\sigma^{4}}{\sigma^{2}+R^{2}}\Big)^2\Big), \\
&K=\Theta\Big(\frac{d}{\beta_{t}\delta^{2}}\Big(\frac{\sigma^{2}+R^{2}}{\sigma^{4}}\Big)^{2}\Big(1+\frac{R^{2}}{\sigma^{4}\beta_{t}}\Big)^{2}\log\Big(\Big(\frac{\sigma^{2}+R^{2}}{\sigma^{4}}\Big)^{2}\frac{W_{2}^{2}(p_{Z_{0}},p_{X_{1}|X_{t}=x_{t}})}{\delta^{2}}\vee 1\Big)\Big).
\end{align*}
\end{customlemma}

\begin{proof}
Let $Z^{1},\ldots,Z^{n}$  i.i.d. random variables with law $p_{X_{1}|X_{t}=x_{t}}$, such that $(Z^{i},\bar{Z}_{K\eta}^{i})$ is an optimal coupling of $p_{X_{1}|X_{t}=x_{t}}$ and $p_{\bar{Z}_{K\eta}}$ for each $1\leq i\leq n$.
By definition of the stable velocity estimator~\eqref{eq:stable:velocity:estimator}, and Proposition~\ref{proposition:rescaling:velocity}, we obtain
\begin{align*}
&\bbE\Big[\|u(t,x_{t})-\what{U}_{\mathrm{stab}}(t,x_{t})\|_{2}^{2}\Big] \\
&= \frac{(1-t)^2}{t^4} \bbE\Big[\|\bbE[\nabla\log p_{X_{1}}(X_{1})|X_{t}=x_{t}]-\frac{1}{n}\sum_{i=1}^{n}\nabla\log p_{X_{1}}(\bar{Z}_{K\eta}^{i})\|_{2}^{2}\Big].
\end{align*}
Skipping the factor  $\frac{(1-t)^2}{t^4}$, the right-hand side can be estimated as 
\begin{align}
&\bbE\Big[\|\bbE[\nabla\log p_{X_{1}}(X_{1})|X_{t}=x_{t}]-\frac{1}{n}\sum_{i=1}^{n}\nabla\log p_{X_{1}}(Z^{i}) \\
 &\quad +\frac{1}{n}\sum_{i=1}^{n}(\nabla\log p_{X_{1}}(Z^{i})-\nabla\log p_{X_{1}}(\bar{Z}_{K\eta}^{i}))\|_{2}^{2}\Big] \nonumber \\
&\leq 
2 \underbrace{\bbE\Big[\|\bbE[\nabla\log p_{X_{1}}(X_{1})|X_{t}=x_{t}]-\frac{1}{n}\sum_{i=1}^{n}\nabla\log p_{X_{1}}(Z^{i})\|_{2}^{2}\Big]}_{A_1} \\
&\quad 
+2 \underbrace{\bbE\Big[\|\frac{1}{n}\sum_{i=1}^{n}(\nabla\log p_{X_{1}}(Z^{i})-\nabla\log p_{X_{1}}(\bar{Z}_{K\eta}^{i}))\|_{2}^{2}\Big]}_{A_2}. \label{eq:lemma:velocity:estimation:2:2}
\end{align}
For the first term, we obtain by independency of the $Z^i \sim p_{X_1|X_t = x_t}$, $i=1,\ldots,n$ that
\begin{align}
A_1
&=\frac{1}{n^{2}}\bbE\Big[\|\sum_{i=1}^{n}(\bbE[\nabla\log p_{X_{1}}(X_{1})|X_{t}=x_{t}]-\nabla\log p_{X_{1}}(Z^{i}))\|_{2}^{2}\Big] \nonumber \\
&=\frac{1}{n^{2}}\sum_{i=1}^{n}\sum_{j=1}^{n}\bbE\Big[\Big\langle \bbE[\nabla\log p_{X_{1}}(X_{1})|X_{t}=x_{t}]-\nabla\log p_{X_{1}}(Z^{i}), \\
&\qquad\qquad\qquad 
\bbE[\nabla\log p_{X_{1}}(X_{1})|X_{t}=x_{t}]-\nabla\log p_{X_{1}}(Z^{j}) \Big\rangle\Big] \nonumber \\
&=\frac{1}{n^{2}}\sum_{i=1}^{n}\bbE\Big[\|\bbE[\nabla\log p_{X_{1}}(X_{1})|X_{t}=x_{t}]-\nabla\log p_{X_{1}}(Z^{i})\|_{2}^{2}\Big] \nonumber \\
&\quad+\frac{1}{n^{2}}\sum_{i\neq j}\Big\langle \bbE[\nabla\log p_{X_{1}}(X_{1})|X_{t}=x_{t}]-\bbE[\nabla\log p_{X_{1}}(Z^{i})], \\
&\qquad\qquad\qquad \bbE[\nabla\log p_{X_{1}}(X_{1})|X_{t}=x_{t}]-\bbE[\nabla\log p_{X_{1}}(Z^{j})] \Big\rangle \\
&=\frac{1}{n^{2}}\sum_{i=1}^{n}\bbE\Big[\|\bbE[\nabla\log p_{X_{1}}(X_{1})|X_{t}=x_{t}]-\nabla\log p_{X_{1}}(Z^{i})\|_{2}^{2}\Big] \nonumber \\
&=\frac{1}{n^{2}}\sum_{i=1}^{n}\trace(\Cov(\nabla\log p_{X_{1}}(X_{1})|X_{t}=x_{t})) 
\end{align}
and by Lemma~\ref{lemma:conditional:cov:score} then
\begin{align}
A_1&\leq\frac{1}{n}\Big(\frac{(1-t)^{2}d}{\sigma^{2}(\sigma^{2}t^{2}+(1-t)^{2})}+\frac{t^{4}R^{2}}{(\sigma^{2}t^{2}+(1-t)^{2})^{2}}\Big), \label{eq:lemma:velocity:estimation:2:3}
\end{align}
For second term, we get  
\begin{align}
&A_2
=\frac{1}{n^{2}}\sum_{i=1}^{n}\sum_{j=1}^{n}\bbE\Big[\langle \nabla\log p_{X_{1}}(Z^{i})-\nabla\log p_{X_{1}}(\bar{Z}_{K\eta}^{i}), \nabla\log p_{X_{1}}(Z^{j})-\nabla\log p_{X_{1}}(\bar{Z}_{K\eta}^{j}) \rangle\Big] \nonumber \\
&\leq \frac{1}{n^{2}}\sum_{i,j=1}^{n}\bbE^{\frac{1}{2}}\Big[\|\nabla\log p_{X_{1}}(Z^{i})-\nabla\log p_{X_{1}}(\bar{Z}_{K\eta}^{i})\|_{2}^{2}\Big]
\, \bbE^{\frac{1}{2}}\Big[\|\nabla\log p_{X_{1}}(Z^{j})-\nabla\log p_{X_{1}}(\bar{Z}_{K\eta}^{j})\|_{2}^{2}\Big] .
\end{align}
and further by Lemma~\ref{lemma:lipschitz:target:score},~\eqref{eq:lemma:velocity:estimation:1} and~\eqref{eq:lemma:velocity:estimation:5} that
\begin{align}
A_2 &\leq \Big(\frac{\sigma^{2}+R^{2}}{\sigma^{4}}\Big)^{2}\frac{1}{n^{2}}\sum_{i,j=1}^{n}\bbE^{\frac{1}{2}}\Big[\|Z^{i}-\bar{Z}_{K\eta}^{i}\|_{2}^{2}\Big]\bbE^{\frac{1}{2}}\Big[\|Z^{j}-\bar{Z}_{K\eta}^{j}\|_{2}^{2}\Big] \\
&\leq 
\Big(\frac{\sigma^{2}+R^{2}}{\sigma^{4}}\Big)^{2}W_{2}^{2}\Big(p_{X_{1}|X_{t}=x_{t}},p_{\bar{Z}_{K\eta}}\Big)
\\
&\lesssim \Big(\frac{\sigma^{2}+R^{2}}{\sigma^{4}}\Big)^{2} \Big(\exp(-\beta_{t}Kh)W_{2}^{2}(p_{Z_{0}},p_{X_{1}|X_{t}=x_{t}})+d\Big(1+\frac{R^{2}}{\sigma^{4}\beta_{t}}\Big)^{2}\eta \Big). 
\end{align}
Adding $A_1$ and  $A_2$ yields
\begin{align*}
&\bbE\Big[\|\bbE[\nabla\log p_{X_{1}}(X_{1})|X_{t}=x_{t}]-\frac{1}{n}\sum_{i=1}^{n}\nabla\log p_{X_{1}}(\bar{Z}_{K\eta}^{i})\|_{2}^{2}\Big] \nonumber \\
&\lesssim \frac{1}{n}\Big(\frac{(1-t)^{2}d}{\sigma^{2}(\sigma^{2}t^{2}+(1-t)^{2})}+\frac{t^{4}R^{2}}{(\sigma^{2}t^{2}+(1-t)^{2})^{2}}\Big) \nonumber \\
&\quad +\Big(\frac{\sigma^{2}+R^{2}}{\sigma^{4}}\Big)^{2}\exp(-\beta_{t}K\eta)W_{2}^{2}(p_{Z_{0}},p_{X_{1}|X_{t}=x_{t}})+d\Big(\frac{\sigma^{2}+R^{2}}{\sigma^{4}}\Big)^{2}\Big(1+\frac{R^{2}}{\sigma^{4}\beta_{t}}\Big)^{2}\eta.
\end{align*} 
Setting $K=0$ if $W_{2}^{2}(p_{Z_{0}},p_{X_{1}|X_{t}=x_{t}})\leq(\frac{\sigma^{4}}{\sigma^{2}+R^{2}})^{2}\delta^{2}$ and otherwise each terms in the right-hand side to $\delta^{2}$ yields
\begin{align*}
&n=\frac{d(1-t)^{2}(\sigma^{2}t^{2}+(1-t)^{2})+\sigma^{2}t^{4}R^{2}}{\sigma^{2}(\sigma^{2}t^{2}+(1-t)^{2})^{2}\delta^{2}}, \quad \eta=\frac{\sigma^{8}\beta_{t}^{2}\delta^{2}}{d(\sigma^{4}\beta_{t}+R^{2})^{2}}\Big(\frac{\sigma^{4}}{\sigma^{2}+R^{2}}\Big)^2, \\
&K=\frac{d}{\beta_{t}\delta^{2}}\Big(\frac{\sigma^{2}+R^{2}}{\sigma^{4}}\Big)^{2}\Big(1+\frac{R^{2}}{\sigma^{4}\beta_{t}}\Big)^{2}\log\Big(\Big(\frac{\sigma^{2}+R^{2}}{\sigma^{4}}\Big)^{2}\frac{W_{2}^{2}(p_{Z_{0}},p_{X_{1}|X_{t}=x_{t}})}{\delta^{2}}\Big).
\end{align*}
This completes the proof.
\end{proof}

\subsection{Error bounds of flow initialization}
\label{section:proof:convergence:initialization}

\begin{customlemma}{\ref{lemma:error:initialization}}
Let Assumption~\ref{assumption:Gaussian:convolution} be fulfilled. Let $\what{X}_{T_{0}}\coloneq\what{U}_{L\tau}$ be the terminal state of the Langevin Monte Carlo in~\eqref{eq:langevin:monte:carlo:warmstart}. Then, for any $\kappa \in(0,1)$, we have 
\begin{equation}
\bbE\Big[W_{2}^{2}\Big(p_{\what{X}_{T_{0}}},p_{X_{T_{0}}}\Big)\Big] \lesssim \kappa^{2},    
\end{equation}
provided that the error of the velocity estimation $\delta^{2}$ in~\eqref{eq:lemma:velocity:estimation:0}, the step size $\tau$, and the number of steps $L$ fulfill
\begin{align*}
&\delta^{2}=\Theta\Big(\exp\Big\{-\frac{16T_0^{2}R^{2}}{\big(1- T_0(1- \sigma)\big)^2}\Big\}\frac{(1-T_{0})^{4}}{T_{0}^{2}}\kappa^{2}\Big), \\
&\tau=\Theta\Big(\exp\Big\{-\frac{16T_0^{2}R^{2}}{\big(1- T_0(1- \sigma)\big)^2}\Big\}\frac{(1-T_{0})^{2}}{d(G+1)^{2}}\kappa^{2}\Big), \\
&L=\Theta\Big(\exp\Big\{\frac{24T_0^{2}R^{2}}{\big(1- T_0(1- \sigma)\big)^2}\Big\}\frac{d(G+1)^{2}}{(1-T_{0})^{2}\kappa^{2}}\Big(\frac{8T_0^{2}R^{2}}{\big(T_0^2\sigma^2+(1-T_0)^2\big)^2}+\log\frac{\kl(p_{U_{0}},p_{X_{T_{0}}})}{\kappa^{2}}\Big)\vee 0\Big).
\end{align*}
Here $G$ is a constant from Theorem~\ref{thm} only depending on $R$ and $\sigma$, the expectation in~\eqref{eq:lemma:error:initialization:0} is taken with respect to particles of Monte Carlo approximation~\eqref{eq:velocity:estimator} to the velocity field.
\end{customlemma}

\begin{proof}
By~\eqref{eq:score:velocity}, we have 
\begin{equation}
\nabla^{2}\log p_{X_{T_{0}}}(x)=\frac{T_{0}}{1-T_{0}}\nabla u(T_{0},x)-\frac{1}{1-T_{0}} I_{d}.
\end{equation}
Then it follows from Theorem~\ref{thm} that 
\begin{equation}\label{eq:lip:T0}
\|\nabla^{2}\log p_{X_{T_{0}}}(x)\|_{\rm op} \leq \frac{G  \, T_0 +1}{1-T_{0}} \le \frac{G +1}{1-T_{0}}.
\end{equation}
Further, by applying Lemma~\ref{lemma:velocity:estimation}, we get
\begin{equation*}
\mathbb{E}\Big[\|\nabla\log p_{T_0}(x)-\widehat{S}(T_0,x)\|_2^2\Big]\lesssim 
\frac{\delta^2 T_0^2}{(1-T_0)^4},
\end{equation*}
for any $x\in\mathbb{R}^{d}$. Applying Theorem~\ref{lemma:lmc:score} with $p=p_{T_0}$, $\widehat s$ as in~\eqref{eq:score:estimation}, Lipschitz constant $L$ in~\eqref{eq:lip:T0}, and $h = \tau$ implies 
\begin{align}
&\bbE\Big[\kl\Big(p_{\what{U}_{L\tau}},p_{X_{T_{0}}}\Big)\Big] \\
&\leq \exp\Big(-\frac{L\tau}{C_{\rm LSI}(p_{X_{T_{0}}})}\Big)\kl\Big(p_{\widehat{U}_{0}},p_{X_{T_{0}}}\Big) + 32d\frac{(G+1)^{2}}{(1-T_{0})^{2}}C_{\rm LSI}(p_{X_{T_{0}}})\tau \\
&\quad +\sum_{\ell=0}^{L-1}\exp\Big(-\frac{L-1-\ell}{C_{\rm LSI}(p_{X_{T_{0}}})}\tau\Big)12\tau\int\widehat{p}_{\ell\tau}(x_{\ell\tau})\mathbb{E}\Big[\|\what{s}(x_{\ell\tau})-\nabla\log p(x_{\ell\tau})\|_{2}^{2}\Big]\d x_{\ell\tau} \\
&\lesssim \exp\Big(-\frac{L\tau}{C_{\rm LSI}(p_{X_{T_{0}}})}\Big)\kl\Big(p_{\widehat{U}_{0}},p_{X_{T_{0}}}\Big) + 32d\frac{(G+1)^{2}}{(1-T_{0})^{2}}\tau C_{\rm LSI}(p_{X_{T_{0}}}) \\
&\quad +\sum_{j=0}^{L-1}\exp\Big(-\frac{L-1-\ell}{C_{\rm LSI}(p_{X_{T_{0}}})}\tau\Big)12\tau\frac{\delta^2 T_0^2}{(1-T_0)^4}  \\
&= \exp\Big(-\frac{L\tau}{C_{\rm LSI}(p_{X_{T_{0}}})}\Big)\kl\Big(p_{\widehat{U}_{0}},p_{X_{T_{0}}}\Big) + 32d\frac{(G+1)^{2}}{(1-T_{0})^{2}}C_{\rm LSI}(p_{X_{T_{0}}})\tau \\
&\quad +\Big(1-\exp\Big(-\frac{\tau}{C_{\rm LSI}(p_{X_{T_{0}}})}\Big)\Big)^{-1}12\tau\frac{\delta^2 T_0^2}{(1-T_0)^4}  \\
&\leq \exp\Big(-\frac{L\tau}{C_{\rm LSI}(p_{X_{T_{0}}})}\Big)\kl\Big(p_{\widehat{U}_{0}},p_{X_{T_{0}}}\Big) + 32d\frac{(G+1)^{2}}{(1-T_{0})^{2}}C_{\rm LSI}(p_{X_{T_{0}}})\tau \\
&\quad +16C_{\rm LSI}(p_{X_{T_{0}}})\frac{\delta^2 T_0^2}{(1-T_0)^4},
\end{align}
where the last inequality holds from $1-\exp(-\frac{\tau}{C_{\rm LSI}(p_{X_{T_{0}}})})\geq\frac{3}{4}\frac{\tau}{C_{\rm LSI}(p_{X_{T_{0}}})}$ as $\frac{\tau}{C_{\rm LSI}(p_{X_{T_{0}}})}\in(0,\frac{1}{4}]$. Then Otto--Villani theorem~\cite{Gentil2020entropc} shows that $p_{X_{T_{0}}}$ satisfies Talagrand's $T_{2}$ inequality, which yields 
\begin{align*}
\bbE\Big[W_{2}^{2}\Big(p_{\what{U}_{L\tau}},p_{X_{T_{0}}}\Big)\Big]
&\lesssim 
C_{\rm LSI}(p_{X_{T_{0}}})\exp\Big(-\frac{L\tau}{C_{\rm LSI}(p_{X_{T_{0}}})}\Big)\kl\Big(p_{\widehat{U}_{0}},p_{X_{T_{0}}}\Big) \\
&\quad +32d\frac{(G+1)^{2}}{(1-T_{0})^{2}}C_{\rm LSI}(p_{X_{T_{0}}})^{2}\tau+16C_{\rm LSI}(p_{X_{T_{0}}})^{2}\frac{\delta^2 T_0^2}{(1-T_0)^4}.
\end{align*}
Setting the second and third terms in the right-hand side as $\kappa^{2}$ yields
\begin{equation*}
\tau=\frac{(1-T_{0})^{2}}{d(G+1)^{2}}\frac{\kappa^{2}}{C_{\rm LSI}(p_{X_{T_{0}}})^{2}}, \quad\text{and}\quad \delta^{2}=\frac{(1-T_{0})^{4}}{T_{0}^{2}}\frac{\kappa^{2}}{C_{\rm LSI}(p_{X_{T_{0}}})^{2}}.
\end{equation*}
Then setting the first term as $\kappa^{2}$ implies
\begin{equation*}
L=\frac{d(G+1)^{2}}{(1-T_{0})^{2}}\frac{C_{\rm LSI}(p_{X_{T_{0}}})^{3}}{\kappa^{2}}\log\Big(\frac{C_{\rm LSI}(p_{X_{T_{0}}})\kl(p_{\widehat{U}_{0}},p_{X_{T_{0}}})}{\kappa^{2}}\vee 1\Big).
\end{equation*}
Substituting $C_{\text{LSI}}(p_{X_{T_0}})$ from~\eqref{eq:Clsi} completes the proof.
\end{proof}

\subsection{Main results}\label{section:error:ode}

\begin{customtheorem}{\ref{theorem:error:ode}}[Flow ODEs with vanilla velocity estimation]
Let Assumption~\ref{assumption:Gaussian:convolution} be fulfilled. Let $\what{\psi}$ be the discrete-time flow map defined by the exponential integrator in~\eqref{eq:PFODE:velocity:ei}. Assume $T^{*}<T_{0}<T_{\rm end}<1$. Then we have 
\begin{align*}
&\bbE\Big[W_{2}^{2}\Big(p_{X_{1}},(T_{\rm end}^{-1}\what{\psi}(T_{\rm end},\cdot))_{\sharp}p_{\what{X}_{T_{0}}}\Big)\Big]  \\
&\lesssim \underbrace{\frac{d(1-T_{\rm end})^{2}}{T_{\rm end}^{2}}}_{\text{early-stopping error}}
+ \underbrace{\frac{\exp(2G(T_{\rm end}-T_{0}))}{T_{\rm end}^{2}}\kappa^{2}}_{\text{initialization error}} \\
&\quad + \underbrace{\frac{\exp(2G(T_{\rm end}-T_{0}))}{T_{\rm end}^{2}}\frac{H^{2}}{T_{0}^{4}}\frac{d+R^{2}+1}{(1-T_{\rm end})^{2}}h^{2}}_{\text{discretization error}} + \underbrace{\frac{\exp(2G(T_{\rm end}-T_{0}))}{T_{\rm end}^{2}}\frac{(T_{\rm end}-T_0)^2}{(1-T_0)(1-T_{\rm end})}\delta^{2}}_{\text{velocity estimation error}},
\end{align*}
where $G$ the a constant from Theorem~\ref{thm} only depending on $R$ and $\sigma$, the constant $H$ from Lemma~\ref{lemma:derivative:denoiser} only depends on $R$ and $\sigma$. The error of the velocity estimation $\delta$ is defined in Lemma~\ref{lemma:velocity:estimation}, and the error of flow initialization $\kappa$ is defined in Lemma~\ref{lemma:error:initialization}.
\end{customtheorem}

\begin{proof}
We first estimate the early-stopping error. It is apparent that $(X_{1},X_{1}+(1-T_{\rm end})T_{\rm end}^{-1}X_{0})$ is a coupling of $(p_{X_{1}},(T_{\rm end}^{-1})_{\sharp}p_{X_{T_{\rm end}}})$, where $(X_{0},X_{1})\sim\gamma\otimes p_{X_{1}}$. Then
\begin{align}
W_{2}^{2}(p_{X_{1}},(T_{\rm end}^{-1})_{\sharp}p_{X_{T_{\rm end}}})
&\leq \bbE\Big[\|X_{1}-\Big(X_{1}+\frac{1-T_{\rm end}}{T_{\rm end}}X_{0}\Big)\|_{2}^{2}\Big] \\
&=\frac{(1-T_{\rm end})^{2}}{T_{\rm end}^{2}}\bbE\big[\|X_{0}\|_{2}^{2}\big]=\frac{d(1-T_{\rm end})^{2}}{T_{\rm end}^{2}}.
\end{align}
On the other hand,
\begin{align*}
W_{2}^{2}\Big((T_{\rm end}^{-1})_{\sharp}p_{X_{T_{\rm end}}},(T_{\rm end}^{-1}\, \what{\psi}(T_{\rm end},\cdot))_{\sharp}p_{\what{X}_{T_{0}}}\Big)
\leq 
\frac{1}{T_{\rm end}^{2}} W_{2}^{2}\Big(p_{X_{T_{\rm end}}},(\what{\psi}(T_{\rm end},\cdot))_{\sharp}p_{\what{X}_{T_{0}}}\Big),
\end{align*}
so that
\begin{equation*}
\bbE\Big[W_{2}^{2}\Big(p_{X_{1}},(T_{\rm end}^{-1})\, \what{\psi}(T_{\rm end},\cdot)_{\sharp}p_{\what{X}_{T_{0}}}\Big)\Big] 
\le 2 \frac{d(1-T_{\rm end})^{2}}{T_{\rm end}^{2}} + \frac{2}{T_{\rm end}^{2}} 
W_{2}^{2}\Big(p_{X_{T_{\rm end}}},(\what{\psi}(T_{\rm end},\cdot))_{\sharp}p_{\what{X}_{T_{0}}}\Big).
\end{equation*}
Using Lemma~\ref{lemma:error:ode:decomposition} completes the proof. 
\end{proof}

Recall that we defined $D(t,x_t)\coloneqq \bbE[X_1|X_t=x_t]$ and $\what{D}(t,x_{t}) \coloneq \frac{1}{n}\sum_{i=1}^n\bar{Z}_{K\eta}^i$, see Section~\ref{subsect:vel_approx}. 

\begin{remark}[Measurability of $\what{X}_{T_0}$]
For making sense of the term \[\bbE\Big[W_{2}^{2}\Big(p_{X_{1}},(T_{\rm end}^{-1}\what{\psi}(T_{\rm end},\cdot))_{\sharp}p_{\what{X}_{T_{0}}}\Big)\Big],\] we view $\what{X}_{T_{0}}$ as a random variable on a product space, where one factor is used to define the associated measure on $\bbR^d$ and the other is used for the expectation. For details, see below. 
We operate on the product probability space $(\Omega_S \times \Omega_I, \mathcal{F}_S \otimes \mathcal{F}_I, \mathbb{P}_S \otimes \mathbb{P}_I)$. Here, the space $\Omega_S$ carries the $n$ independent standard Brownian motions $(B^1_s, \dots, B^n_s)$ used to simulate the Langevin dynamics for estimating the score, see~\eqref{eq:langevin:score:estimation} or~\eqref{eq:langevin:score:estimation_wied}, while the integration space $\Omega_I$ carries the initial state $U_0$ and the sequence of integration noises $\zeta_1, \dots, \zeta_L \sim \calN(0, I_d)$ of~\eqref{eq:langevin:monte:carlo:warmstart}. Recall that $\what{X}_{T_0}=\what{U}_{L\tau}$ from~\eqref{eq:langevin:monte:carlo:warmstart}.\\
\textbf{Joint measurability of $\what{S}(t,x_t)(\omega_S)$ on $\mathcal{F}_S \otimes \mathcal{B}(\bbR^d)$:}
Recall, see~\eqref{eq:langevin:score:estimation}, that for a fixed diffusion time $t \in (0,1)$, the score estimator is given by the empirical average $\what{S}(t, x_t)(\omega_S) = \frac{t}{(1-t)^2n}\sum_{i=1}^n Z^{i, t, x_t}_T(\omega_S)-\frac{1}{(1-t)^2}x_t$, where each $Z^{i, t, x_t}_T$ is the solution evaluated at time $T$ of the SDE:
\begin{equation} \label{eq:langevin:score:estimation_wied}
\d Z_{s}^{i, t,x_t} = \left( \frac{t(x_t - tZ_s^{i, t,x_t})}{(1-t)^2} + \nabla\log p_{X_1}(Z_s^{i, t,x_t}) \right) \ds + \sqrt{2}\d B^i_{s}, \quad s\in[0,T].
\end{equation}
Since the marginal score $\nabla\log p_{X_1}(z)$ is uniformly Lipschitz continuous in $z$, the drift of this SDE satisfies global Lipschitz conditions and is affine with respect to the spatial parameter $x_t$. By continuous dependence of SDEs on the initial condition, see e.g.~\cite[Theorem 3.4.1]{kunita2019stochastic}, the mapping $x_t \mapsto Z^{i, t, x_t}_T(\omega_S)$ is continuous for fixed $\omega_S \in \Omega_S$. Note that we can use~\cite[Theorem 3.4.1]{kunita2019stochastic} by augmenting the space and set the change in the $x_t$ variable to constant. Furthermore, the solution to an SDE driven by the Brownian motions on $\Omega_S$ is $\mathcal{F}_S$-measurable. Thus, the map $(\omega_S, x_t) \mapsto \what{S}(t, x_t)(\omega_S)$ is measurable in $\omega_S$ and continuous in $x_t$. This defines $\what{S}$ as a \textit{Carath\'{e}odory function}, guaranteeing that it is jointly $\mathcal{F}_S \otimes \mathcal{B}(\bbR^d)$-measurable, see~\cite[Lemma 4.51]{aliprantis2006infinite}.\\
\textbf{Joint measurability of $\what{U}_{l\tau}$ on $\mathcal{F}_S \otimes \mathcal{F}_I$:}
We can now establish the joint measurability of the sequence $\what{U}_{\ell\tau}$, defined in~\eqref{eq:langevin:monte:carlo:warmstart},  by induction. The initial condition $U_0: \Omega_I \to \bbR^d$ trivially extends to a jointly measurable function on $\Omega_S \times \Omega_I$ that is constant with respect to $\omega_S$. Assuming that the $\ell$-th iterate $\what{U}_{\ell\tau}$ is jointly $\mathcal{F}_S \otimes \mathcal{F}_I$-measurable, the evaluated score term $\what{S}\big(t, \what{U}_{\ell\tau}(\omega_S, \omega_I)\big)(\omega_S)$ is the composition of the jointly measurable random state and the Carath\'{e}odory function $\what{S}$. Because the composition of measurable functions remains measurable, this term is $\mathcal{F}_S \otimes \mathcal{F}_I$-measurable. The subsequent update,
\[
\what{U}_{(\ell+1)\tau}(\omega_S, \omega_I) = \what{U}_{\ell\tau}(\omega_S, \omega_I) + \tau \what{S}\big(t, \what{U}_{\ell\tau}(\omega_S, \omega_I)\big)(\omega_S) + \sqrt{2\tau}\zeta_{\ell+1}(\omega_I),
\]
is a linear combination of jointly measurable components and is therefore jointly measurable itself. By induction, every iterate, including the final state $\what{X}_{T_0} = \what{U}_{L\tau}$, is jointly measurable.\\
\textbf{Conclusion:} First note that the joint measurability implies that the map $\Omega_S\to L^2(\bbP_I,\bbR^d)$ is measurable, see e.g.~\cite[Proposition 1.2.7]{hytonen2016analysis}. They use strong measurability in their assumptions, but since all our probability space can be chosen to be separable, measurability and strong measurability coincide.  Furthermore, the map from square integrable random variables to its law is continuous, if we view it as a map from $L^2(\mathbb P_I,\bbR^d)$ to $(\mathcal P_2,W_2)$. Consequently we obtain that the mapping $\omega_S \mapsto W_2(\mu, p_{\what{X}_{T_0}})$ is also $\mathcal{F}_S$-measurable.
\end{remark}

\begin{lemma}\label{lemma:error:ode:decomposition}
Let Assumption~\ref{assumption:Gaussian:convolution} be fulfilled. Suppose that there exists $\delta\in(0,1)$, such that
\begin{equation}\label{eq:lemma:error:ode:decomposition:0}
\bbE\big[\|D(t_{m},x)-\what{D}(t_{m},x)\|_{2}^{2}\big] \leq \delta^{2},
\end{equation}
for each $0\leq m\leq M-1$ and $x\in\mathbb{R}^{d}$. Then it follows 
\begin{align*}
&\bbE\Big[ W_{2}^{2}\Big(p_{X_{T_{\rm end}}},(\what{\psi}(T_{\rm end},\cdot))_{\sharp}p_{\what{X}_{T_{0}}}\Big)\Big]  \\
&\leq 3\exp(2G(T_{\rm end}-T_{0})) \Big\{ \bbE\Big[ W_{2}^{2}\Big(p_{X_{T_{0}}},p_{\what{X}_{T_{0}}}\Big)\Big] + \frac{3H^{2}}{T_{0}^{4}}\frac{d+R^{2}+1}{(1-T_{\rm end})^{2}}h^{2} + \frac{(T_{\rm end}-T_0)^2}{(1-T_0)(1-T_{\rm end})}\delta^{2} \Big\},
\end{align*}
where $G$ and $H$ are constants only depending on $R$ and $\sigma$, defined in Theorem~\ref{thm} and Lemma~\ref{lemma:derivative:denoiser}, respectively.
\end{lemma}

\begin{proof}
Let $\pi_{T_0}\in\mathcal{P}_2(\bbR^{d}\times\bbR^{d})$ be an optimal coupling of $p_{X_{T_{0}}}$ and $p_{\what{X}_{T_{0}}}$ realized by $(X_{T_0},\widehat{X}_{T_0})$, which implies
\begin{equation}\label{eq:lemma:error:ode:0}
 W_{2}^{2}\Big(p_{X_{T_{0}}},p_{\what{X}_{T_{0}}}\Big)=\bbE\Big[\|X_{T_{0}}-\what{X}_{T_{0}}\|_{2}^{2}\Big].
\end{equation}
We define the $L^{2}(\pi_{T_{0}})$-norm of  $f:\bbR^{d}\times\bbR^{d}\rightarrow\bbR^{d}$ by
\begin{equation*}
\|f\|_{L^{2}(\pi_{T_{0}})}^{2}\coloneq \int\|f(x_{T_{0}},\what{x}_{T_{0}})\|_{2}^{2} \, \d \pi_{T_{0}}(x_{T_{0}},\widehat x_{T_{0}}).
\end{equation*}
Applying the method of integrating factors the flow map of the probability flow ODE~\eqref{eq:PFODE} becomes
\begin{equation*}
\psi(t_{m+1},X_{T_{0}}) = \frac{1-t_{m+1}}{1-t_{m}}\psi(t_{m},X_{T_{0}})+(1-t_{m+1})\int_{t_{m}}^{t_{m+1}}\frac{1}{(1-t)^{2}}D(t,\psi(t,X_{T_{0}}))\dt.
\end{equation*}
Recall that the exponential integrator in~\eqref{eq:PFODE:velocity:ei} reads as
\begin{align*}
\what{\psi}(t_{m+1},\what{X}_{T_{0}}) 
&= \frac{1-t_{m+1}}{1-t_{m}}\what{\psi}(t_{m},\what{X}_{T_{0}})+(1-t_{m+1})\int_{t_{m}}^{t_{m+1}}\frac{1}{(1-t)^{2}}\what{D}(t_{m},\widehat \psi(t_{m},\what{X}_{T_{0}}))\dt \\
&= \frac{1-t_{m+1}}{1-t_{m}}\what{\psi}(t_{m},\what{X}_{T_{0}})+\Big(1-\frac{1-t_{m+1}}{1-t_{m}}\Big)\what{D}(t_{m},\what{\psi}(t_{m},\what{X}_{T_{0}})).
\end{align*}
Then it follows from the triangular inequality that 
\begin{equation}\label{eq:lemma:error:ode:1}
\begin{aligned}
&\|\psi(t_{m+1},\cdot)-\what{\psi}(t_{m+1},\cdot)\|_{L^{2}(\pi_{T_{0}})}  
\leq \frac{1-t_{m+1}}{1-t_{m}}\|\psi(t_{m},\cdot)-\what{\psi}(t_{m},\cdot)\|_{L^{2}(\pi_{T_{0}})} \\
&\quad +\underbrace{(1-t_{m+1})\int_{t_{m}}^{t_{m+1}}\frac{1}{(1-t)^{2}}\|D(t,\psi(t,\cdot))-D(t_{m},\psi(t_{m},\cdot))\|_{L^{2}(\pi_{T_{0}})} \dt}_{A_1} \\
&\quad +\underbrace{\frac{h}{1-t_{m}}\|D(t_{m},\psi(t_{m},\cdot))-D(t_{m},\what{\psi}(t_{m},\cdot))\|_{L^{2}(\pi_{T_{0}})}}_{A_2} \\
&\quad +\underbrace{\frac{h}{1-t_{m}}\|D(t_{m},\what{\psi}(t_{m},\cdot))-\what{D}(t_{m},\what{\psi}(t_{m},\cdot))\|_{L^{2}(\pi_{T_{0}})}}_{A_3}.
\end{aligned}
\end{equation}
\noindent{\textbf{Estimation of $A_1$}:}
By Lemma~\ref{lemma:derivative:denoiser} and  Lemma~\ref{lemma:norm:Xt}, we obtain
\begin{align*}
\|\frac{\d}{\ds}D(s,\psi(s,\cdot))\|_{L^{2}(\pi_{T_{0}})}^{2} 
&= \int\|\frac{\d}{\ds}D(s,\psi(s,x_{T_{0}}))\|_{2}^{2}p_{X_{T_{0}}}(x_{T_{0}}) \d x_{T_{0}} \\
&\leq \frac{2H^{2}}{T_{0}^{4}}\frac{\bbE[\|\psi(s,X_{T_{0}})\|_{2}^{2}]+1}{(1-s)^{2}} = \frac{2H^{2}}{T_{0}^{4}}\frac{\bbE[\|X_{s}\|_{2}^{2}]+1}{(1-s)^{2}} \\
&\leq \frac{2H^{2}}{T_{0}^{4}}\frac{d+R^{2}+1}{(1-s)^{2}}.
\end{align*}
As a consequence, 
\begin{equation}\label{eq:lemma:error:ode:1:1}
\|\frac{\d}{\ds}D(s,\psi(s,\cdot))\|_{L^{2}(\pi_{T_{0}})} \leq \frac{2H}{T_{0}^{2}}\frac{\sqrt{d}+R+1}{1-s}.
\end{equation}
For the term (i) in~\eqref{eq:lemma:error:ode:1}, we have 
\begin{align}
&(1-t_{m+1})\int_{t_{m}}^{t_{m+1}}\frac{1}{(1-t)^{2}}\|D(t,\psi(t,\cdot))-D(t_{m},\psi(t_{m},\cdot))\|_{L^{2}(\pi_{T_{0}})} \dt \nonumber \\
&= (1-t_{m+1})\int_{t_{m}}^{t_{m+1}}\frac{1}{(1-t)^{2}}\left\|\int_{t_{m}}^{t}\frac{\d}{\ds}D(s,\psi(s,\cdot))\ds\right\|_{L^{2}(\pi_{T_{0}})} \dt \nonumber \\
&\leq (1-t_{m+1})\int_{t_{m}}^{t_{m+1}}\frac{1}{(1-t)^{2}}\int_{t_{m}}^{t}\|\frac{\d}{\ds}D(s,\psi(s,\cdot))\|_{L^{2}(\pi_{T_{0}})}\ds\dt \nonumber \\
&\leq (1-t_{m+1})\frac{2H}{T_{0}^{2}}(\sqrt{d}+R+1)\int_{t_{m}}^{t_{m+1}}\frac{1}{(1-t)^{2}}\int_{t_{m}}^{t}\frac{\ds}{1-s}\dt \nonumber \\
&\leq (1-t_{m+1})\frac{2H}{T_{0}^{2}}(\sqrt{d}+R+1)\int_{t_{m}}^{t_{m+1}}\frac{t-t_{m}}{(1-t)^{3}}\dt \nonumber \\
&= \frac{H}{T_{0}^{2}}\frac{(\sqrt{d}+R+1)h^{2}}{(1-t_{m})(1-t_{m+1})}, \label{eq:lemma:error:ode:1:2}
\end{align}
where the first inequality holds from the Jensen's inequality, the second inequality invokes~\eqref{eq:lemma:error:ode:1:1}, and the last inequality used the fact that $1-s\geq 1-t$. 

\noindent{\textbf{Estimation of $A_2$}:}
It is straightforward that 
\begin{align}
&\frac{h}{1-t_{m}}\|D(t_{m},\psi(t_{m},\cdot))-D(t_{m},\what{\psi}(t_{m},\cdot))\|_{L^{2}(\pi_{T_{0}})} \nonumber \\
&\leq \frac{h}{1-t_{m}}(1+(1-t_{m})G)\|\psi(t_{m},\cdot)-\what{\psi}(t_{m},\cdot)\|_{L^{2}(\pi_{T_{0}})} \nonumber \\
&= \Big(\frac{h}{1-t_{m}}+Gh\Big)\|\psi(t_{m},\cdot)-\what{\psi}(t_{m},\cdot)\|_{L^{2}(\pi_{T_{0}})}, \label{eq:lemma:error:ode:2:1}
\end{align}
where the inequality is owing to~\eqref{eq:lipschitz:denoiser}.

\noindent{\textbf{Conclusion:}}
Substituting~\eqref{eq:lemma:error:ode:1:2} and~\eqref{eq:lemma:error:ode:2:1} into~\eqref{eq:lemma:error:ode:1} yields
\begin{align*}
&\|\psi(t_{m+1},\cdot)-\what{\psi}(t_{m+1},\cdot)\|_{L^{2}(\pi_{T_{0}})}  \\
&\leq \frac{1-t_{m+1}}{1-t_{m}}\|\psi(t_{m},\cdot)-\what{\psi}(t_{m},\cdot)\|_{L^{2}(\pi_{T_{0}})} + \frac{H}{T_{0}^{2}}\frac{(\sqrt{d}+R+1)h^{2}}{(1-t_{m})(1-t_{m+1})} \\
&\quad +\Big(\frac{h}{1-t_{m}}+Gh\Big)\|\psi(t_{m},\cdot)-\what{\psi}(t_{m},\cdot)\|_{L^{2}(\pi_{T_{0}})} + \frac{1}{1-t_{m}}h\varepsilon_{m} \\
&= (1+Gh)\|\psi(t_{m},\cdot)-\what{\psi}(t_{m},\cdot)\|_{L^{2}(\pi_{T_{0}})} + \frac{H}{T_{0}^{2}}\frac{(\sqrt{d}+R+1)h^{2}}{(1-t_{m})(1-t_{m+1})} + \frac{h\varepsilon_{m}}{1-t_{m}} \\
&\leq \exp(Gh)\|\psi(t_{m},\cdot)-\what{\psi}(t_{m},\cdot)\|_{L^{2}(\pi_{T_{0}})} + \frac{H}{T_{0}^{2}}\frac{(\sqrt{d}+R+1)h^{2}}{(1-t_{m})(1-t_{m+1})} + \frac{h\varepsilon_{m}}{1-t_{m}},
\end{align*}
where the last inequality holds from $1+z\leq \exp(z)$, and $\varepsilon_{m}$ is defined as 
\begin{equation}\label{eq:lemma:error:ode:3:1}
\varepsilon_{m} \coloneq \|D(t_{m},\what{\psi}(t_{m},\cdot))-\what{D}(t_{m},\what{\psi}(t_{m},\cdot))\|_{L^{2}(\pi_{T_{0}})}.
\end{equation}
\noindent{\textbf{Recursion over $m$:}}
Applying this recursion sequentially yields
\begin{equation}\label{eq:lemma:error:ode:3:2}
\begin{aligned}
&\|\psi(T_{\rm end},\cdot)-\what{\psi}(T_{\rm end},\cdot)\|_{L^{2}(\pi_{T_{0}})} = \|\psi(t_{M},\cdot)-\what{\psi}(t_{M},\cdot)\|_{L^{2}(\pi_{T_{0}})} \\
&\leq \exp(G(T_{\rm end}-T_{0})) \Big\{\|\psi(T_{0},\cdot)-\what{\psi}(T_{0},\cdot)\|_{L^{2}(\pi_{T_{0}})} \\
&\quad + \sum_{m=0}^{M-1}\frac{H}{T_{0}^{2}}\frac{(\sqrt{d}+R+1)h^{2}}{(1-t_{m})(1-t_{m+1})} + \sum_{m=0}^{M-1}\frac{h\varepsilon_{m}}{1-t_{m}}\Big\}.
\end{aligned}
\end{equation}
Note that 
\begin{align}
\sum_{m=0}^{M-1}\frac{h}{(1-t_{m})(1-t_{m+1})} 
&= \sum_{m=0}^{M-1}\Big(\frac{1}{1-t_{m+1}}-\frac{1}{1-t_{m}}\Big) \nonumber \\
&= \frac{1}{1-T_{\rm end}}-\frac{1}{1-T_{0}} \leq \frac{1}{1-T_{\rm end}} .\label{eq:lemma:error:ode:3:3}
\end{align}
Since the function $t \mapsto (1-t)^{-1}$ is increasing, it follows that 
\begin{align}
\sum_{m=0}^{M-1}\frac{h\varepsilon_{m}}{1-t_{m}}
&\leq \Bigg(\sum_{m=0}^{M-1}h\varepsilon_{m}^{2}\Bigg)^{\frac{1}{2}}\Bigg(\sum_{m=0}^{M-1}\frac{h}{(1-t_{m})^{2}}\Bigg)^{\frac{1}{2}} \nonumber \\
&= \Bigg(\sum_{m=0}^{M-1}h\varepsilon_{m}^{2}\Bigg)^{\frac{1}{2}}\Bigg(\sum_{m=0}^{M-1}\int_{t_{m}}^{t_{m+1}}\frac{\dt}{(1-t_{m})^{2}}\Bigg)^{\frac{1}{2}} \nonumber \\
&\leq \Bigg(\sum_{m=0}^{M-1}h\varepsilon_{m}^{2}\Bigg)^{\frac{1}{2}}\Bigg(\sum_{m=0}^{M-1}\int_{t_{m}}^{t_{m+1}}\frac{\dt}{(1-t)^{2}}\Bigg)^{\frac{1}{2}} \nonumber \\
&= \Bigg(\sum_{m=0}^{M-1}h\varepsilon_{m}^{2}\Bigg)^{\frac{1}{2}}\Bigg(\int_{T_{0}}^{T_{\rm end}}\frac{\dt}{(1-t)^{2}}\Bigg)^{\frac{1}{2}} \nonumber \\
&= \Bigg(\sum_{m=0}^{M-1}h\varepsilon_{m}^{2}\Bigg)^{\frac{1}{2}}\left(\frac{T_{\rm end}-T_0}{(1-T_0)(1-T_{\rm end})}\right)^{\frac{1}{2}}. \label{eq:lemma:error:ode:3:4}
\end{align}
Substituting~\eqref{eq:lemma:error:ode:3:3} and~\eqref{eq:lemma:error:ode:3:4} into~\eqref{eq:lemma:error:ode:3:2} and taking square on both sides yield
\begin{align*}
&\|\psi(T_{\rm end},\cdot)-\what{\psi}(T_{\rm end},\cdot)\|_{L^{2}(\pi_{T_{0}})}^{2} \\
&\leq 3\exp(2G(T_{\rm end}-T_{0})) \Bigg\{\|\psi(T_{0},\cdot)-\what{\psi}(T_{0},\cdot)\|_{L^{2}(\pi_{T_{0}})}^{2} + \Big(\frac{H}{T_{0}^{2}}\frac{\sqrt{d}+R+1}{1-T_{\rm end}}h\Big)^{2}  \\
&\quad + \frac{T_{\rm end}-T_0}{(1-T_0)(1-T_{\rm end})}\sum_{m=0}^{M-1}h\|D(t_{m},\what{\psi}(t_{m},\cdot))-\what{D}(t_{m},\what{\psi}(t_{m},\cdot))\|_{L^{2}(\pi_{T_{0}})}^{2}\Bigg\},
\end{align*}
where we used~\eqref{eq:lemma:error:ode:3:1}. 
    
As a consequence, using 
\begin{equation*}
\|\psi(T_0,\cdot)-\widehat{\psi}(T_0,\cdot)\|_{L^2(\pi_{T_0})}=\|x_{T_0}-\widehat{x}_{T_0}\|_{L^2(\pi_{T_0})}=W_2\Big(p_{X_{T_0}},p_{\widehat{X}_{T_0}}\Big),
\end{equation*}
we obtain
\begin{align*}
& W_{2}^{2}\Big(p_{X_{T_{\rm end}}},(\what{\psi}(T_{\rm end},\cdot))_{\sharp}p_{\what{X}_{T_{0}}}\Big) \leq \|\psi(T_{\rm end},\cdot)-\what{\psi}(T_{\rm end},\cdot)\|_{L^{2}(\pi_{T_{0}})}^{2} \\
&\leq 3\exp(2G(T_{\rm end}-T_{0})) \Bigg\{\|\psi(T_{0},\cdot)-\what{\psi}(T_{0},\cdot)\|_{L^{2}(\pi_{T_{0}})}^{2} + \Big(\frac{H}{T_{0}^{2}}\frac{\sqrt{d}+R+1}{1-T_{\rm end}}h\Big)^{2} \\
&\quad  + \frac{T_{\rm end}-T_0}{(1-T_0)(1-T_{\rm end})}\sum_{m=0}^{M-1}h\|D(t_{m},\what{\psi}(t_{m},\cdot))-\what{D}(t_{m},\what{\psi}(t_{m},\cdot))\|_{L^{2}(\pi_{T_{0}})}^{2}\Bigg\} \\
&\leq 3\exp(2G(T_{\rm end}-T_{0})) \Bigg\{ W_{2}^{2}\Big(p_{X_{T_{0}}},p_{\what{X}_{T_{0}}}\Big) + \frac{3H^{2}}{T_{0}^{4}}\frac{d+R^{2}+1}{(1-T_{\rm end})^{2}}h^{2}  \\
&\quad  +\frac{T_{\rm end}-T_0}{(1-T_0)(1-T_{\rm end})}\sum_{m=0}^{M-1}h\|D(t_{m},\what{\psi}(t_{m},\cdot))-\what{D}(t_{m},\what{\psi}(t_{m},\cdot))\|_{L^{2}(\pi_{T_{0}})}^{2} \Bigg\},
\end{align*}
By taking expectation with respect to $\what{X}_{T_{0}}$ and $\what{D}$, we have 
\begin{align*}
&\bbE\Big[ W_{2}^{2}\Big(p_{X_{T_{\rm end}}},(\what{\psi}(T_{\rm end},\cdot))_{\sharp}p_{\what{X}_{T_{0}}}\Big)\Big]  \\
&\leq 3\exp(2G(T_{\rm end}-T_{0})) \Bigg\{ \bbE\Big[ W_{2}^{2}\Big(p_{X_{T_{0}}},p_{\what{X}_{T_{0}}}\Big)\Big] + \frac{3H^{2}}{T_{0}^{4}}\frac{d+R^{2}+1}{(1-T_{\rm end})^{2}}h^{2} \\
&\quad  + \frac{T_{\rm end}-T_0}{(1-T_0)(1-T_{\rm end})}\sum_{m=0}^{M-1}h\bbE\Big[\|D(t_{m},\what{\psi}(t_{m},\cdot))-\what{D}(t_{m},\what{\psi}(t_{m},\cdot))\|_{L^{2}(\pi_{T_{0}})}^{2}\Big] \Bigg\} \\
&= 3\exp(2G(T_{\rm end}-T_{0})) \Bigg\{ \bbE\Big[ W_{2}^{2}\Big(p_{X_{T_{0}}},p_{\what{X}_{T_{0}}}\Big)\Big] + \frac{3H^{2}}{T_{0}^{4}}\frac{d+R^{2}+1}{(1-T_{\rm end})^{2}}h^{2}  \\
&\quad  + \frac{T_{\rm end}-T_0}{(1-T_0)(1-T_{\rm end})}\sum_{m=0}^{M-1}h\int\bbE\big[\|D(t_{m},\what{\psi}(t_{m},x_{T_{0}}))-\what{D}(t_{m},\what{\psi}(t_{m},x_{T_{0}}))\|_{2}^{2}\big]p_{X_{T_{0}}}(x_{T_{0}})\d x_{T_{0}} \Bigg\} \\
&\leq 3\exp(2G(T_{\rm end}-T_{0})) \Big\{ \bbE\Big[ W_{2}^{2}\Big(p_{X_{T_{0}}},p_{\what{X}_{T_{0}}}\Big)\Big] + \frac{3H^{2}}{T_{0}^{4}}\frac{d+R^{2}+1}{(1-T_{\rm end})^{2}}h^{2} + \frac{(T_{\rm end}-T_0)^2}{(1-T_0)(1-T_{\rm end})}\delta^{2} \Big\},
\end{align*}
where the last inequality invokes the condition~\eqref{eq:lemma:error:ode:decomposition:0}. This completes the proof.
\end{proof}

\begin{customtheorem}{\ref{theorem:error:ode:stable}}
[Flow ODEs with stable velocity estimation]
Let Assumption~\ref{assumption:Gaussian:convolution} be fulfilled. Let $\what{\psi}_{\rm stab}$ be the discrete-time flow map defined by the exponential integrator in~\eqref{eq:PFODE:velocity:ei:stable}. Assume $T^{*}<T_{0}<T_{\rm end}<1$. Then we have 
\begin{align*}
&\bbE\Big[W_{2}^{2}\Big(p_{X_{1}},(T_{\rm end}^{-1}\what{\psi}_{\rm stab}(T_{\rm end},\cdot))_{\sharp}p_{\what{X}_{T_{0}}}\Big)\Big]  \\
&\lesssim \underbrace{\frac{d(1-T_{\rm end})^{2}}{T_{\rm end}^{2}}}_{\text{early-stopping error}} \!
+ \underbrace{\exp\Big(2\Big(G+\frac{2}{T_{0}}\Big)(T_{\rm end}-T_{0})\Big)\frac{\kappa^{2}}{T_{\rm end}^{2}}}_{\text{initialization error}} \\
&\quad + \underbrace{\exp\Big(2\Big(G+\frac{2}{T_{0}}\Big)(T_{\rm end}-T_{0})\Big)\frac{V^{2}}{T_{0}^{6}}\frac{d+R^{2}+1}{(1-T_{\rm end})^{2}}h^{2}}_{\text{discretization error}} \\ 
&\quad + \underbrace{\exp\Big(2\Big(G+\frac{2}{T_{0}}\Big)(T_{\rm end}-T_{0})\Big)\frac{(T_{\rm end}-T_{0})^{2}}{T_{\rm end}^{2}T_{0}^{4}}\delta^{2}}_{\text{velocity estimation error}},
\end{align*}
where $G$ the a constant from Theorem~\ref{thm} only depending on $R$ and $\sigma$, the constant $V$ from Lemma~\ref{lemma:derivative:stable} only depends on $R$ and $\sigma$. The error of the velocity estimation $\delta$ is defined in Lemma~\ref{lemma:velocity:estimation:stable}, and the error of flow initialization $\kappa$ is defined in Lemma~\ref{lemma:error:initialization}.
\end{customtheorem}

\begin{proof}
Using Lemma~\ref{lemma:error:ode:decomposition:stable} and the same arguments as Theorem~\ref{theorem:error:ode} provides the proof.
\end{proof}

\begin{lemma}\label{lemma:error:ode:decomposition:stable}
Let Assumption~\ref{assumption:Gaussian:convolution} be fulfilled. Suppose there exists $\delta\in(0,1)$, such that
\begin{equation}\label{eq:lemma:error:ode:decomposition:stable:0}
\bbE\big[\|F(t_{m},x)-\what{F}(t_{m},x)\|_{2}^{2}\big] \leq \delta^{2},
\end{equation}
for each $0\leq m\leq M-1$ and $x\in\mathbb{R}^{d}$. Then it follows 
\begin{align*}
&\bbE\Big[ W_{2}^{2}\Big(p_{X_{T_{\rm end}}},(\what{\psi}_{\rm stab}(T_{\rm end},\cdot))_{\sharp}p_{\what{X}_{T_{0}}}\Big)\Big]  \\
&\leq 3\exp\Big(2\Big(G+\frac{2}{T_{0}}\Big)(T_{\rm end}-T_{0})\Big) \Big\{ \bbE\Big[ W_{2}^{2}\Big(p_{X_{T_{0}}},p_{\what{X}_{T_{0}}}\Big)\Big] + \frac{3T_{\rm end}^{2}V^{2}}{T_{0}^{6}}\frac{d+R^{2}+1}{(1-T_{\rm end})^{2}}h^{2} \\
&\quad + \frac{(T_{\rm end}-T_{0})^{2}}{T_{0}^{4}}\delta^{2} \Big\},
\end{align*}
where $G$ and $V$ are constants only depending on $R$ and $\sigma$, defined in Theorem~\ref{thm} and Lemma~\ref{lemma:derivative:stable}, respectively.
\end{lemma}

\begin{proof}
Let $\pi_{T_0}\in\mathcal{P}_2(\bbR^{d}\times\bbR^{d})$ be an optimal coupling of $p_{X_{T_{0}}}$ and $p_{\what{X}_{T_{0}}}$ realized by $(X_{T_0},\widehat{X}_{T_0})$, which implies
\begin{equation}
 W_{2}^{2}\Big(p_{X_{T_{0}}},p_{\what{X}_{T_{0}}}\Big)=\bbE\Big[\|X_{T_{0}}-\what{X}_{T_{0}}\|_{2}^{2}\Big].
\end{equation}
We define the $L^{2}(\pi_{T_{0}})$-norm of  $f:\bbR^{d}\times\bbR^{d}\rightarrow\bbR^{d}$ by
\begin{equation*}
\|f\|_{L^{2}(\pi_{T_{0}})}^{2}\coloneq \int\|f(x_{T_{0}},\what{x}_{T_{0}})\|_{2}^{2} \, \d \pi_{T_{0}}(x_{T_{0}},\widehat x_{T_{0}}).
\end{equation*}
According to the variation-of-constants in Lemma~\ref{lemma:variation:constant} and Proposition~\ref{proposition:rescaling:velocity}, the flow map of probability flow ODE~\eqref{eq:PFODE} satisfies
\begin{equation*}
\psi(t_{m+1},X_{T_{0}}) = \frac{t_{m+1}}{t_{m}}\psi(t_{m},X_{T_{0}})+t_{m+1}\int_{t_{m}}^{t_{m+1}}\frac{1}{t^{3}}F(t,\psi(t,X_{T_{0}}))\dt.
\end{equation*}
Recall the exponential integrator in~\eqref{eq:PFODE:velocity:ei:stable}
\begin{align*}
\what{\psi}_{\rm stab}(t_{m+1},\what{X}_{T_{0}}) 
&= \frac{t_{m+1}}{t_{m}}\what{\psi}_{\rm stab}(t_{m},\what{X}_{T_{0}})+t_{m+1}\int_{t_{m}}^{t_{m+1}}\frac{1}{t^{3}}\what{F}(t_{m},\what{\psi}_{\rm stab}(t_{m},\what{X}_{T_{0}}))\dt \\
&= \frac{t_{m+1}}{t_{m}}\what{\psi}_{\rm stab}(t_{m},\what{X}_{T_{0}})+\frac{t_{m}+t_{m+1}}{2t_{m}^{2}t_{m+1}} h\what{F}(t_{m},\what{\psi}_{\rm stab}(t_{m},\what{X}_{T_{0}})).
\end{align*}
Then it follows from the triangular inequality that 
\begin{equation}\label{eq:lemma:error:ode:stable:1}
\begin{aligned}
&\|\psi(t_{m+1},\cdot)-\what{\psi}_{\rm stab}(t_{m+1},\cdot)\|_{L^{2}(\pi_{T_{0}})}  \\
&\leq \frac{t_{m+1}}{t_{m}}\|\psi(t_{m},\cdot)-\what{\psi}_{\rm stab}(t_{m},\cdot)\|_{L^{2}(\pi_{T_{0}})} \\
&\quad +\underbrace{t_{m+1}\int_{t_{m}}^{t_{m+1}}\frac{1}{t^{3}}\|F(t,\psi(t,\cdot))-F(t_{m},\psi(t_{m},\cdot))\|_{L^{2}(\pi_{T_{0}})} \dt}_{A_1} \\
&\quad +\underbrace{\frac{t_{m}+t_{m+1}}{2t_{m}^{2}t_{m+1}} h\|F(t_{m},\psi(t_{m},\cdot))-F(t_{m},\what{\psi}_{\rm stab}(t_{m},\cdot))\|_{L^{2}(\pi_{T_{0}})}}_{A_2} \\
&\quad +\underbrace{\frac{t_{m}+t_{m+1}}{2t_{m}^{2}t_{m+1}} h\|F(t_{m},\what{\psi}_{\rm stab}(t_{m},\cdot))-\what{F}(t_{m},\what{\psi}_{\rm stab}(t_{m},\cdot))\|_{L^{2}(\pi_{T_{0}})}}_{A_3}.
\end{aligned}
\end{equation}

\noindent{\textbf{Estimation of $A_1$:}}
By a direct calculation, we have 
\begin{align*}
\|\frac{\d}{\ds}F(s,\psi(s,\cdot))\|_{L^{2}(\pi_{T_{0}})}^{2} 
&= \int\|\frac{\d}{\ds}F(s,\psi(s,x_{T_{0}}))\|_{2}^{2}p_{X_{T_{0}}}(x_{T_{0}}) \d x_{T_{0}} \\
&\leq 2V^{2}\frac{\bbE[\|\psi(s,X_{T_{0}})\|_{2}^{2}]+1}{(1-s)^{4}} = 2V^{2}\frac{\bbE[\|X_{s}\|_{2}^{2}]+1}{(1-s)^{4}} \\
&\leq 2V^{2}\frac{d+R^{2}+1}{(1-s)^{4}},
\end{align*}
where the first inequality invokes Lemma~\ref{lemma:derivative:stable}, and the last inequality is due to Lemma~\ref{lemma:norm:Xt}. As a consequence, 
\begin{equation}\label{eq:lemma:error:ode:stable:1:1}
\|\frac{\d}{\ds}F(s,\psi(s,\cdot))\|_{L^{2}(\pi_{T_{0}})} \leq 2V\frac{\sqrt{d}+R+1}{(1-s)^{2}}.
\end{equation}
For the term $A_1$ in~\eqref{eq:lemma:error:ode:stable:1}, we have 
\begin{align}
&t_{m+1}\int_{t_{m}}^{t_{m+1}}\frac{1}{t^{3}}\|F(t,\psi(t,\cdot))-F(t_{m},\psi(t_{m},\cdot))\|_{L^{2}(\pi_{T_{0}})} \dt \nonumber \\
&= t_{m+1}\int_{t_{m}}^{t_{m+1}}\frac{1}{t^{3}}\left\|\int_{t_{m}}^{t}\frac{\d}{\ds}F(s,\psi(s,\cdot))\ds\right\|_{L^{2}(\pi_{T_{0}})} \dt \nonumber \\
&\leq t_{m+1}\int_{t_{m}}^{t_{m+1}}\frac{1}{t^{3}}\int_{t_{m}}^{t}\|\frac{\d}{\ds}F(s,\psi(s,\cdot))\|_{L^{2}(\pi_{T_{0}})}\ds\dt \nonumber \\
&\leq 2t_{m+1}V(\sqrt{d}+R+1)\int_{t_{m}}^{t_{m+1}}\frac{1}{t^{3}}\int_{t_{m}}^{t}\frac{\ds}{(1-s)^{2}}\dt \nonumber \\
&\leq \frac{2T_{\rm end}V}{T_{0}^{3}}\frac{\sqrt{d}+R+1}{(1-t_{m+1})^{2}}\int_{t_{m}}^{t_{m+1}}t-t_{m}\dt \nonumber \\
&= \frac{T_{\rm end}V}{T_{0}^{3}}\frac{(\sqrt{d}+R+1)h^{2}}{(1-t_{m+1})^{2}}, \label{eq:lemma:error:ode:stable:1:2}
\end{align}
where the first inequality holds from Jensen's inequality, the second inequality invokes~\eqref{eq:lemma:error:ode:stable:1:1}, and the last inequality used the fact that $1-s\geq 1-t$ and $t_{m}\geq T_{0}$. 

\noindent{\textbf{Estimation of $A_2$:}}
It follows from Proposition~\ref{proposition:rescaling:velocity} that 
\begin{equation*}
\nabla F(t,x_{t}) = t^{2}\nabla u(t,x_{t})-t I_{d},
\end{equation*}
which implies 
\begin{equation}
\|\nabla F(t,x_{t})\|_{\rm op} \leq Gt^{2}+t.
\end{equation}
Then we have 
\begin{align}
&\frac{t_{m}+t_{m+1}}{2t_{m}^{2}t_{m+1}} h\|F(t_{m},\psi(t_{m},\cdot))-F(t_{m},\what{\psi}_{\rm stab}(t_{m},\cdot))\|_{L^{2}(\pi_{T_{0}})} \nonumber \\
&\leq \frac{t_{m}+t_{m+1}}{2t_{m}t_{m+1}} h(Gt_{m}+1)\|\psi(t_{m},\cdot)-\what{\psi}_{\rm stab}(t_{m},\cdot)\|_{L^{2}(\pi_{T_{0}})}, \label{eq:lemma:error:ode:stable:2:1}
\end{align}
where the inequality holds from Theorem~\ref{thm}.

\noindent{\textbf{Conclusion:}}
Substituting~\eqref{eq:lemma:error:ode:stable:1:2} and~\eqref{eq:lemma:error:ode:stable:2:1} into~\eqref{eq:lemma:error:ode:stable:1} yields
\begin{align*}
&\|\psi(t_{m+1},\cdot)-\what{\psi}_{\rm stab}(t_{m+1},\cdot)\|_{L^{2}(\pi_{T_{0}})}  \\
&\leq \frac{t_{m+1}}{t_{m}}\|\psi(t_{m},\cdot)-\what{\psi}_{\rm stab}(t_{m},\cdot)\|_{L^{2}(\pi_{T_{0}})} + \frac{T_{\rm end}V}{T_{0}^{3}}\frac{(\sqrt{d}+R+1)h^{2}}{(1-t_{m+1})^{2}} \\
&\quad + \frac{t_{m}+t_{m+1}}{2t_{m}t_{m+1}} h(Gt_{m}+1)\|\psi(t_{m},\cdot)-\what{\psi}_{\rm stab}(t_{m},\cdot)\|_{L^{2}(\pi_{T_{0}})} + \frac{t_{m}+t_{m+1}}{2t_{m}^{2}t_{m+1}} h\varepsilon_{m} \\
&\leq \exp\Big(\Big(G+\frac{2}{T_{0}}\Big)h\Big)\|\psi(t_{m},\cdot)-\what{\psi}_{\rm stab}(t_{m},\cdot)\|_{L^{2}(\pi_{T_{0}})} + \frac{T_{\rm end}V}{T_{0}^{3}}\frac{(\sqrt{d}+R+1)}{(1-t_{m+1})^{2}}h^{2} + \frac{h\varepsilon_{m}}{T_{0}^{2}},
\end{align*}
where the last inequality holds from $1+z\leq \exp(z)$, and $\varepsilon_{m}$ is defined as 
\begin{equation}\label{eq:lemma:error:ode:stable:3:1}
\varepsilon_{m} \coloneq \|F(t_{m},\what{\psi}_{\rm stab}(t_{m},\cdot))-\what{F}(t_{m},\what{\psi}_{\rm stab}(t_{m},\cdot))\|_{L^{2}(\pi_{T_{0}})}.
\end{equation}
Here we also used
\begin{align*}
\frac{t_{m+1}}{t_{m}}+\frac{t_{m}+t_{m+1}}{2t_{m}t_{m+1}} h(Gt_{m}+1)
&\leq 1+\frac{h}{t_m}+\frac{1}{t_m}h(Gt_m+1)\\
&= 1+h\left(G+\frac{2}{t_m}\right)\\
&\leq \exp\left(\left(G+\frac{2}{T_0}\right)h\right),
\end{align*}
where we used $t_{m+1}>t_{m}\geq T_{0}$. 

\par\noindent{\textbf{Recursion over $m$:}}
Applying this recursion sequentially yields
\begin{equation}\label{eq:lemma:error:ode:stable:3:2}
\begin{aligned}
&\|\psi(T_{\rm end},\cdot)-\what{\psi}_{\rm stab}(T_{\rm end},\cdot)\|_{L^{2}(\pi_{T_{0}})} 
\leq \exp\Big(\Big(G+\frac{2}{T_{0}}\Big)(T_{\rm end}-T_{0})\Big) \\
& \times \Big\{\|\psi(T_{0},\cdot)-\what{\psi}_{\rm stab}(T_{0},\cdot)\|_{L^{2}(\pi_{T_{0}})} 
+ \sum_{m=0}^{M-1}\frac{T_{\rm end}V}{T_{0}^{3}}\frac{(\sqrt{d}+R+1)h^{2}}{(1-t_{m+1})^{2}} + \sum_{m=0}^{M-1}\frac{h\varepsilon_{m}}{T_{0}^{2}}\Big\}.
\end{aligned}
\end{equation}
Note that 
\begin{equation}\label{eq:lemma:error:ode:stable:3:3}
\sum_{m=0}^{M-1}\frac{h}{(1-t_{m+1})^{2}} = \sum_{m=0}^{M-1}\int_{t_{m}}^{t_{m+1}}\frac{\dt}{(1-t_{m+1})^{2}} \leq \sum_{m=0}^{M-1}\int_{t_{m}}^{t_{m+1}}\frac{\dt}{(1-t)^{2}} \leq \frac{1}{1-T_{\rm end}}.
\end{equation}
Using Cauchy-Schwarz, it follows that 
\begin{equation}\label{eq:lemma:error:ode:stable:3:4}
\sum_{m=0}^{M-1}h\varepsilon_{m} \leq \Bigg(\sum_{m=0}^{M-1}h\varepsilon_{m}^{2}\Bigg)^{\frac{1}{2}}\Bigg(\sum_{m=0}^{M-1}h\Bigg)^{\frac{1}{2}} =\Bigg(\sum_{m=0}^{M-1}h\varepsilon_{m}^{2}\Bigg)^{\frac{1}{2}}\sqrt{T_{\rm end}-T_{0}}. 
\end{equation}
Substituting~\eqref{eq:lemma:error:ode:stable:3:3} and~\eqref{eq:lemma:error:ode:stable:3:4} into~\eqref{eq:lemma:error:ode:stable:3:2} and taking squares on both sides yields
\begin{align*}
&\|\psi(T_{\rm end},\cdot)-\what{\psi}_{\rm stab}(T_{\rm end},\cdot)\|_{L^{2}(\pi_{T_{0}})}^{2} \\
&\leq 3\exp\Big(2\Big(G+\frac{2}{T_{0}}\Big)(T_{\rm end}-T_{0})\Big) \Bigg\{\|\psi(T_{0},\cdot)-\what{\psi}_{\rm stab}(T_{0},\cdot)\|_{L^{2}(\pi_{T_{0}})}^{2} + \Big(\frac{T_{\rm end}V}{T_{0}^{3}}\frac{\sqrt{d}+R+1}{1-T_{\rm end}}h\Big)^{2}  \\
&\quad + \frac{T_{\rm end}-T_{0}}{T_{0}^{4}}\sum_{m=0}^{M-1}h\|F(t_{m},\what{\psi}_{\rm stab}(t_{m},\cdot))-\what{F}(t_{m},\what{\psi}_{\rm stab}(t_{m},\cdot))\|_{L^{2}(\pi_{T_{0}})}^{2}\Bigg\},
\end{align*}
where we used~\eqref{eq:lemma:error:ode:stable:3:1}. As a consequence, 
\begin{align*}
& W_{2}^{2}\Big(p_{X_{T_{\rm end}}},(\what{\psi}_{\rm stab}(T_{\rm end},\cdot))_{\sharp}p_{\what{X}_{T_{0}}}\Big) \leq \|\psi(T_{\rm end},\cdot)-\what{\psi}_{\rm stab}(T_{\rm end},\cdot)\|_{L^{2}(\pi_{T_{0}})}^{2} \\
&\leq 3\exp\Big(2\Big(G+\frac{2}{T_{0}}\Big)(T_{\rm end}-T_{0})\Big) \Bigg\{\|\psi(T_{0},\cdot)-\what{\psi}_{\rm stab}(T_{0},\cdot)\|_{L^{2}(\pi_{T_{0}})}^{2} + \Big(\frac{T_{\rm end}V}{T_{0}^{3}}\frac{\sqrt{d}+R+1}{1-T_{\rm end}}h\Big)^{2} \\
&\quad  + \frac{T_{\rm end}-T_{0}}{T_{0}^{4}}\sum_{m=0}^{M-1}h\|F(t_{m},\what{\psi}_{\rm stab}(t_{m},\cdot))-\what{F}(t_{m},\what{\psi}_{\rm stab}(t_{m},\cdot))\|_{L^{2}(\pi_{T_{0}})}^{2}\Bigg\} \\
&= 3\exp\Big(2\Big(G+\frac{2}{T_{0}}\Big)(T_{\rm end}-T_{0})\Big) \Bigg\{ W_{2}^{2}\Big(p_{X_{T_{0}}},p_{\what{X}_{T_{0}}}\Big) + \frac{3T_{\rm end}^{2}V^{2}}{T_{0}^{6}}\frac{d+R^{2}+1}{(1-T_{\rm end})^{2}}h^{2}  \\
&\quad  + \frac{T_{\rm end}-T_{0}}{T_{0}^{4}}\sum_{m=0}^{M-1}h\|F(t_{m},\what{\psi}_{\rm stab}(t_{m},\cdot))-\what{F}(t_{m},\what{\psi}_{\rm stab}(t_{m},\cdot))\|_{L^{2}(\pi_{T_{0}})}^{2} \Bigg\},
\end{align*}
where the first inequality is due to the definition of the Wasserstein-2 distance, and the equality holds from the fact that $\pi_{T_{0}}$ is the joint density of the optimal coupling $(X_{T_{0}},\what{X}_{T_{0}})$. By taking expectation with respect to $\what{X}_{T_{0}}$ and $\what{S}$, we have 
\begin{align*}
&\bbE\Big[ W_{2}^{2}\Big(p_{X_{T_{\rm end}}},(\what{\psi}_{\rm stab}(T_{\rm end},\cdot))_{\sharp}p_{\what{X}_{T_{0}}}\Big)\Big]  \\
&\leq 3\exp\Big(2\Big(G+\frac{2}{T_{0}}\Big)(T_{\rm end}-T_{0})\Big) \Bigg\{ \bbE\Big[ W_{2}^{2}\Big(p_{X_{T_{0}}},p_{\what{X}_{T_{0}}}\Big)\Big] + \frac{3T_{\rm end}^{2}V^{2}}{T_{0}^{6}}\frac{d+R^{2}+1}{(1-T_{\rm end})^{2}}h^{2} \\
&\quad  + \frac{T_{\rm end}-T_{0}}{T_{0}^{4}}\sum_{m=0}^{M-1}h\bbE\Big[\|F(t_{m},\what{\psi}_{\rm stab}(t_{m},\cdot))-\what{F}(t_{m},\what{\psi}_{\rm stab}(t_{m},\cdot))\|_{L^{2}(\pi_{T_{0}})}^{2}\Big] \Bigg\} \\
&\leq 3\exp\Big(2\Big(G+\frac{2}{T_{0}}\Big)(T_{\rm end}-T_{0})\Big) \Big\{ \bbE\Big[ W_{2}^{2}\Big(p_{X_{T_{0}}},p_{\what{X}_{T_{0}}}\Big)\Big] + \frac{3T_{\rm end}^{2}V^{2}}{T_{0}^{6}}\frac{d+R^{2}+1}{(1-T_{\rm end})^{2}}h^{2} \\
&\quad + \frac{(T_{\rm end}-T_{0})^{2}}{T_{0}^{4}}\delta^{2} \Big\},
\end{align*}
where the last inequality invokes the condition~\eqref{eq:lemma:error:ode:decomposition:stable:0}. This completes the proof.
\end{proof}

\begin{lemma}\label{lemma:time:derivative:denoiser}
Let Assumption~\ref{assumption:Gaussian:convolution} be fulfilled. 
Then, for any $t\in(T_{0},1)$, it holds
\begin{equation*}
\|\partial_{t}D(t,x_{t})\|_{2}
\leq 
\frac{\tilde{\tilde{H}}}{T_{0}^{2}}\frac{\|x_{t}\|_{2}+1}{1-t}, \quad x_{t}\in\bbR^{d},
\end{equation*}
where $\tilde{\tilde{H}}$ is a constant only depending on $R$ and $\sigma$.
\end{lemma}

\begin{proof}
By~\eqref{eq:velocity} and~\eqref{eq:score:velocity}, we have
\begin{equation*}
D(t,x_{t}) = \frac{1}{t}x_{t}+\frac{(1-t)^{2}}{t}\nabla\log p_{X_{t}}(x_{t}).
\end{equation*}
Taking partial derivative with respect to $t$ yields
\begin{equation*}
\partial_{t}D(t,x_{t}) = -\frac{1}{t^{2}}x_{t}+\frac{t^{2}-1}{t^{2}}\nabla\log p_{X_{t}}(x_{t})+\frac{(1-t)^{2}}{t}\partial_{t}\nabla\log p_{X_{t}}(x_{t}).
\end{equation*}
Then, using the triangular inequality together with~\eqref{eq:lemma:time:derivative:score:6:0} and Lemma~\ref{lemma:time:derivative:score}, we obtain
\begin{align}\label{eq:lemma:time:derivative:denoiser:1}
\|\partial_{t}D(t,x_{t})\|_{2}
&\leq \frac{1}{t^{2}}\|x_{t}\|_{2}+\frac{1-t^{2}}{t^{2}}\|\nabla\log p_{X_{t}}(x_{t})\|_{2}+\frac{(1-t)^{2}}{t}\|\nabla\partial_{t}\log p_{X_{t}}(x_{t})\|_{2} \\
&\leq \frac{1}{t^{2}}\|x_{t}\|_{2}+\frac{1-t^{2}}{t^{2}}\Big(\frac{1}{\zeta_{t}^{2}}\|x_{t}\|_{2}+\frac{t}{\zeta_{t}^{2}}R\Big)+ \tilde H\frac{\|x_{t}\|_{2}+1}{t(1-t)}
\end{align}
where $\zeta_t^2 = \sigma^2 t^2 + (1-t)^2$. 
Finally, noting that $\zeta_t^2 \ge c = \sigma^2/(\sigma^2+1)$, we get the assertion by
\begin{align}
\|\partial_{t}D(t,x_{t})\|_{2} &=  \frac{1}{t^2(1-t)} 
\Big( \|x_t\|_2 (1-t + \frac{(1-t^2)(1-t)}{\zeta_t^2} + \tilde H t) + \frac{(1-t^2)(1-t)}{\zeta_t^2} R t + \tilde H t
\Big)\\
&\le 
\frac{1}{(1-t) T_0^2}
\Big( \|x_t\|_2 (1+ \frac{1}{c} + \tilde H) + \frac{1}{c} R + \tilde H
\Big) \le \frac{\tilde{\tilde H}}{T_{0}^{2}}\frac{\|x_{t}\|_{2}+1}{1-t}
\end{align}
with $\tilde{\tilde H} \coloneqq \max \{1+1/c + \tilde H, R/c + \tilde H\}$.
\end{proof}

\begin{lemma}\label{lemma:derivative:denoiser}
Let Assumption~\ref{assumption:Gaussian:convolution} be fulfilled. Then, for any $t\in(T_{0},1)$, it holds
\begin{equation*}
\|\frac{\d}{\d t}D(t,\psi(t,x))\|_{2}\leq \frac{H}{T_{0}^{2}}\frac{\|\psi(t,x)\|_{2}+1}{1-t},
\end{equation*}
where $H$ is a constant only depending on $R$ and $\sigma$.
\end{lemma}

\begin{proof}
We have
\begin{align*}
\frac{\d}{\d t}D(t,\psi(t,x))
&=\partial_{t}D(t,\psi(t,x))+\nabla D(t,\psi(t,x))\frac{\d}{\d t}\psi(t,x) \\
&=\partial_{t}D(t,\psi(t,x))+\nabla D(t,\psi(t,x))u(t,\psi(t,x)).
\end{align*}
By~\eqref{eq:velocity}, we obtain 
\begin{equation*}
\nabla D(t,\psi(t,x)) =  I_{d}+(1-t)\nabla u(t,\psi(t,x)),
\end{equation*}
which implies by the triangular inequality and Theorem~\ref{thm} that 
\begin{equation}\label{eq:lipschitz:denoiser}
\|\nabla D(t,\psi(t,x))\|_{\rm op} \leq 1+(1-t)\|\nabla u(t,\psi(t,x))\|_{\rm op} \leq 1+(1-t)G.
\end{equation}
Together with~\eqref{eq:lemma:time:derivative:score:8} and $\|Y^{x_{t}}_{t}\|_{2}\leq R$ this implies 
\begin{equation*}
\|\nabla D(t,\psi(t,x))u(t,\psi(t,x))\|_{2} 
\leq 
(1+(1-t)G)\Big(\frac{\max\{\sigma^{2},1\}}{\zeta_{t}^{2}}\|x_{t}\|_{2}+\frac{1-t}{\zeta_{t}^{2}}R\Big).
\end{equation*}
Combining this inequality with Lemma~\ref{lemma:time:derivative:denoiser}, we obtain
$$
\|\frac{\d}{\d t}D(t,\psi(t,x))\|_{2}
\leq 
\frac{\tilde{\tilde H}}{T_{0}^{2}}\frac{\|\psi(t,x)\|_{2}+1}{1-t}
+
\frac{1}{T_0^2(1-t)} \frac{1+G}{c} \max\{1, \sigma^2,R\}\Big(\|\psi(t,x)\|_{2}+ 1\Big) .
$$
This completes the proof with $H \coloneqq  \tilde{\tilde H}+ \frac{1+G}{c} \max\{1, \sigma^2,R\} $.
\end{proof}

\begin{lemma}\label{lemma:derivative:stable}
Let Assumption~\ref{assumption:Gaussian:convolution} be fulfilled. Then, for any $t\in(T_{0},1)$, it holds
\begin{equation*}
\|\frac{\d}{\d t}F(t,\psi(t,x))\|_{2}\leq V\frac{\|\psi(t,x)\|_{2}+1}{(1-t)^{2}},
\end{equation*}
where $V$ is a constant only depending on $R$ and $\sigma$.
\end{lemma}

\begin{proof}
Combining Proposition~\ref{proposition:rescaling:velocity} and~\eqref{eq:score:velocity} implies 
\begin{equation}\label{eq:lemma:derivative:stable:1}
S(t,\psi(t,x)) = t(1-t)\nabla\log p_{X_{t}}(\psi(t,x)),
\end{equation}
and consequently
\begin{align}
\frac{\d}{\dt}S(t,\psi(t,x)) 
&= \partial_{t}S(t,\psi(t,x))+\nabla S(t,\psi(t,x))\frac{\d}{\dt}\psi(t,x) \\
&= (1-2t)\nabla\log p_{X_{t}}(\psi(t,x))+t(1-t)\partial_{t}\nabla\log p_{X_{t}}(\psi(t,x)) \\
&\quad +t(1-t)\nabla^{2}\log p_{X_{t}}(\psi(t,x))u(t,\psi(t,x)).
\end{align}
Taking $\ell^{2}$-norm on both sides of the inequality, using the triangular inequality, and applying~\eqref{eq:lemma:time:derivative:score:6:0}, Lemma~\ref{lemma:time:derivative:score}, and~\eqref{eq:lemma:time:derivative:score:7} completes the proof.
\end{proof}

\section{Preconditioning}
\label{section:precondition:alg}

Algorithms~\ref{alg:velocity:estimation:preconditioned} and~\ref{alg:initialization:preconditioning} present the preconditioned counterparts of Algorithms~\ref{alg:velocity:estimation} and~\ref{alg:initialization}, respectively. They detail the procedures for velocity estimation and initialization using preconditioned Langevin algorithms, which are discrete-time approximations of the preconditioned Langevin diffusion~\eqref{eq:langevin:preconditioned}.

\begin{algorithm}[htbp]\label{alg:velocity:estimation:preconditioned}
\caption{Langevin-based velocity estimation with preconditioning}
\KwIn{Time and location of interest $(t,x_{t})$, the target score $\nabla\log p_{X_{1}}$, the step size $\eta$, the number of steps $K$, the smoothing constant $\alpha$, and the tolerance $\epsilon$ for numerical stability.}
\KwOut{Velocity estimator $\what{U}(t,x_{t})$ or $\widehat{U}_{\mathrm{stab}}(t,x_{t})$.}
Initialize particles $\bar{Z}_{0}^{1},\ldots,\bar{Z}_{0}^{n}$ via importance sampling~\eqref{eq:is:Z0}. \\  
Initialize the variance $v_{0}^{i}\leftarrow 0$ for $1\leq i\leq n$. \\
\For {$k\in\{0,\ldots,K-1\}$}{
\texttt{\# Calculate denoising score} \\
$S_{k+1}^{i}\leftarrow-\frac{t^{2}}{(1-t)^{2}}\bar{Z}_{k\eta}^{i}+\frac{t}{(1-t)^{2}}x_{t}+\nabla\log p_{X_{1}}(\bar{Z}_{k\eta}^{i})$ for $1\leq i\leq n$. \\
\texttt{\# Preconditioned Langevin update} \\
$v_{k+1}^{i}\leftarrow\alpha v_{k}^{i}+(1-\alpha)S_{k+1}^{i}\odot S_{k+1}^{i}$ for $1\leq i\leq n$. \Comment{Moving average} \\
$P_{k+1}^{i}\leftarrow\diag(1/(\sqrt{v_{k+1}^{i}}+\epsilon))$ for $1\leq i\leq n$. \Comment{Preconditioner} \\
$\bar{Z}_{(k+1)\eta}^{i}\sim\calN(\bar{Z}_{k\eta}^{i}+\eta P_{k+1}^{i}S_{k+1}^{i},2\eta P_{k+1}^{i})$ for $1\leq i\leq n$. \Comment{Langevin update} \\
} 
\texttt{\# Monte Carlo velocity estimator} \\
\eIf{\normalfont using stable estimation}{
$\widehat{U}_{\mathrm{stab}}(t,x_{t}) \leftarrow \frac{1}{t}x_{t} + \frac{1-t}{t^2}\frac{1}{n}\sum_{i=1}^{n}\nabla\log p_{X_1}(\bar{Z}_{K\eta}^{i})$.}{
$\what{U}(t,x_{t})=-\frac{1}{1-t}x_{t}+ \frac{1}{1-t}\frac{1}{n}\sum_{i=1}^{n}\bar{Z}_{K\eta}^{i}$.
}
\Return{\normalfont $\what{U}(t,x_{t})$ or $\widehat{U}_{\mathrm{stab}}(t,x_{t})$}
\end{algorithm}

\begin{algorithm}[htbp]\label{alg:initialization:preconditioning}
\caption{Initialization of probability flow ODE with preconditioning}
\KwIn{The step size $\tau$, the number of steps $L$, the smoothing constant $\alpha$, and the tolerance $\epsilon$ for numerical stability.}
\KwOut{Particle $\what{U}_{L\tau}$ approximately following $p_{X_{T_{0}}}$.}
Initialize particle: $\what{U}_{0}$. \\
Initialize the variance $v_{0}\leftarrow 0$. \\
\For {$\ell\in\{0,\ldots,L-1\}$}{
\texttt{\# Score estimation} \\
Calculate the velocity estimator $\what{U}(T_{0},\what{U}_{\ell\tau})$ using Algorithm~\ref{alg:velocity:estimation}. \\
Calculate the score estimator $\what{S}_{\ell+1}\coloneq\what{S}(T_{0},\what{U}_{\ell\tau})$ using~\eqref{eq:score:estimation}. \\
\texttt{\# Preconditioned Langevin update} \\
$v_{\ell+1}\leftarrow\alpha v_{\ell}+(1-\alpha)\what{S}_{\ell+1}\odot\what{S}_{\ell+1}$.\Comment{Moving average} \\
$P_{\ell+1}\leftarrow\diag(1/(\sqrt{v_{\ell+1}}+\epsilon))$. \Comment{Preconditioner} \\
$\what{U}_{(\ell+1)\tau}\sim\calN(\what{U}_{\ell\tau}+\tau P_{\ell+1}\what{S}_{\ell+1},2\tau P_{\ell+1})$. \Comment{Langevin update}} 
\Return{$\what{U}_{L\tau}$}
\end{algorithm}

\section{Experimental Details and Additional Experimental Results}
\label{section:exp:details}

\subsection{Experimental details}

\paragraph{Hyper-parameters of SSI}
The hyper-parameters of SSI are summarized in Table~\ref{tab:parameters:ssi}. The two-dimensinal target distributions \texttt{rings}, \texttt{MoG7x7}, and \texttt{MoG40} are studied in Section~\ref{section:experiments:two:dim}, the eight-dimensional distribution \texttt{Many Well} is investigated in Section~\ref{section:experiments:high:dim}, while  \texttt{Bayesian} represents Bayesian inference in Section~\ref{Bayesian:inferrence:gmm}. 

\begin{table}[htbp]
\centering
\caption{Hyper-parameters of SSI in Table~\ref{tab:exp:results}.}
\label{tab:parameters:ssi}
\begin{tabular}{lccccccccccc}
\toprule
 &  & \multicolumn{3}{c}{PF ODE} & \multicolumn{2}{c}{Initialization} & \multicolumn{3}{c}{Score estimation} \\
 \cmidrule(lr){3-5} \cmidrule(lr){6-7} \cmidrule(lr){8-10}
 & dim. & $T_{0}$ & $T_{\rm end}$ & $M$ & $\tau$ & $L$ & $\eta$ & $K$ & $n$ \\
\midrule
\texttt{rings} & 2 & 0.2 & 0.99 & 100 & 0.1 & 100 & 0.01 & 100 & 800 \\
\texttt{MoG7x7} & 2 & 0.2 & 0.99 & 100 & 0.1 & 100 & 0.01 & 100 & 800 \\
\texttt{MoG40} & 2 & 0.2 & 0.99 & 100 & 0.1 & 100 & 0.01 & 100 & 800 \\
\texttt{ManyWell} & 8 & 0.6 & 0.99 & 100 & 0.1 & 100 & 0.01 & 100 & 800 \\
\texttt{Bayesian} & 4 & 0.8 & 0.99 & 20 & 0.01 & 50 & 0.01 & 20 & 80 \\
\bottomrule
\end{tabular}
\end{table}

\par For the hyper-parameters of probability flow ODE (PF ODE), consistent with Algorithm~\ref{alg:pfode:sampling}, $T_0$ and $T_{\text{end}}$ denote the initialization and early-stopping times, respectively, while $M$ represents the number of discretization steps in~\eqref{eq:PFODE:velocity:ei}. Regarding the PF ODE initialization in Algorithm~\ref{alg:initialization}, $\tau$ and $L$ correspond to the step size and number of steps of the Langevin Monte Carlo in~\eqref{eq:langevin:monte:carlo:warmstart}. Similarly, for score estimation, $\eta$ and $K$ denote the step size and number of steps of the Langevin Monte Carlo in~\eqref{eq:langevin:score:estimation}. Furthermore, to reduce the computational cost, we execute Langevin Monte Carlo chains in parallel and utilize samples obtained after convergence. Consequently, $n$ denotes the total number of Monte Carlo particles in velocity estimation as~\eqref{eq:velocity:estimator}, calculated as the product of the number of chains and the sampling steps retained after the warm up period in each chain. For preconditioning of both velocity estimation and flow initialization, we use the hyperparameter $\alpha=0.999$ and $\epsilon=1.0\times 10^{-3}$ in Algorithms~\ref{alg:velocity:estimation:preconditioned} and~\ref{alg:initialization:preconditioning}.

\paragraph{Hyper-parameters of baselines} 
We compared our method against ULA, MALA, pULA, and HMC in Sections~\ref{section:experiments:two:dim} and~\ref{section:experiments:high:dim}.

\begin{table}[htbp]
\centering
\caption{Hyper-parameters of baselines in Table~\ref{tab:exp:results}.}
\label{tab:parameters:baselines}
\small
\setlength{\tabcolsep}{4pt}
\begin{tabular}{lccccccccc}
\toprule
 &  & \multicolumn{2}{c}{ULA} & \multicolumn{2}{c}{MALA} & \multicolumn{2}{c}{pULA} & \multicolumn{2}{c}{HMC} \\
 \cmidrule(lr){3-4} \cmidrule(lr){5-6} \cmidrule(lr){7-8} \cmidrule(lr){9-10}
 & dim. & $\eta$ & $N$ & $\eta$ & $N$ & $\eta$ & $N$ & $\eta$ & $N$ \\
\midrule
\texttt{Rings} & 2 & 0.05 & 10,000 & 0.1 & 10,000 & 0.5 & 10,000 & 0.05 & 100$\times$100 \\
\texttt{MoG7x7} & 2 & 0.1 & 10,000 & 0.1 & 10,000 & 0.1 & 10,000 & 1.0 & 100$\times$100 \\
\texttt{MoG40} & 2 & 0.5 & 50,000 & 0.5 & 50,000 & 0.5 & 50,000 & 1.0 & 100$\times$100 \\
\bottomrule
\end{tabular}
\end{table}

\par We report the step size $\eta$, and the number of steps $N$ for each method in Table~\ref{tab:parameters:baselines}. In particular, for HMC, we consistently used a trajectory of 100 transitions with 100 leapfrog steps per transition across all tasks, which means $N=100\times 100$ of steps in total. The hyperparameter and implementation of PT aligns with that in~\cite{Chen2024Diffusive}.

\paragraph{Ablation Studies}
We performed ablation studies on the \texttt{MoG7x7} task to analyze the impact of the initialization time $T_0$ and the benefits of preconditioning. The initialization time was varied across the set $T_0 \in \{0.01, 0.1, 0.2, 0.3, 0.4, 0.5, 0.6\}$. Consistent with our implementation, for $T_0=0.01$, the number of outer Langevin steps was set to $L=0$, effectively reducing the initialization to a standard Gaussian sample. For other $T_0$, we employed the hyper-parameters listed for \texttt{MoG7x7} in Table~\ref{tab:parameters:ssi}, while specifically comparing the standard preconditioned setting against a non-preconditioned baseline.

\subsection{Empirical ring weights}\label{appendix:exp:weights}

Table~\ref{tab:exp:rings:weight} reports the normalized empirical weights assigned to the individual rings. For each generated sample \(x\), we compute its Euclidean distance from the origin, \(\lVert x\rVert_{2}\), and assign it to the ring whose radius is closest to this distance. The empirical weight of each ring is then computed from the number of samples assigned to it and normalized such that its reference value is \(1\).

\begin{table}[htbp]
\centering
\small
\setlength{\tabcolsep}{2pt}
\caption{Normalized empirical ring weights of the sample sets generated by different methods. The reference weight is \(1\) for each ring.}
\label{tab:exp:rings:weight}
\begin{tabular}{lcccccccc}
\toprule
Method
& 1st ring
& 2nd ring
& 3rd ring
& 4th ring
& 5th ring
& 6th ring
& 7th ring
& 8th ring \\
\midrule
ULA
& 1.32 & 0.73 & 0.90 & 0.93 & 0.95 & 0.92 & 0.75 & 1.50 \\
MALA
& 2.83 & 0.05 & 0.02 & 0.02 & 0.02 & 0.01 & 1.45 & 3.61 \\
pULA
& 1.64 & 1.11 & 0.68 & 0.57 & 0.57 & 0.73 & 1.14 & 1.56 \\
HMC
& 1.74 & 0.65 & 0.40 & 0.95 & 1.52 & 1.82 & 0.42 & 0.51 \\
PT
& 1.20 & 1.82 & 1.38 & 1.10 & 0.83 & 0.64 & 0.55 & 0.48 \\
\textbf{SSI w/o precond.}
& 0.60 & 0.88 & 1.05 & 1.12 & 1.17 & 1.19 & 1.13 & 0.86 \\
\textbf{SSI}
& 0.62 & 0.86 & 1.00 & 1.13 & 1.14 & 1.24 & 1.12 & 0.89 \\
\textbf{Stab.\ SSI}
& 0.60 & 0.80 & 0.98 & 1.11 & 1.19 & 1.10 & 1.15 & 1.07 \\
Reference
& 1.00 & 1.00 & 1.00 & 1.00 & 1.00 & 1.00 & 1.00 & 1.00 \\
\bottomrule
\end{tabular}
\end{table}

The empirical weights produced by SSI and its variants remain close to the reference weights across all eight rings, indicating that these methods accurately recover the relative mass assigned to the different radial modes. In contrast, the baseline methods exhibit more pronounced over- or under-representation of particular rings. This imbalance is especially severe for MALA, which assigns almost no mass to several intermediate rings while substantially overrepresenting the innermost and outermost rings. Overall, these results demonstrate that SSI and its variants provide a more faithful recovery of the target distribution's relative ring weights.

\subsection{Computational costs}\label{appendix:exp:cost}

Table~\ref{tab:exp:comp:cost} compares the computational costs of the different methods on the \texttt{MoG40} target using the hyperparameters in Tables~\ref{tab:parameters:ssi} and~\ref{tab:parameters:baselines}. We report three complementary measures: the end-to-end wall-clock runtime, the number of target score evaluations (NTSE) per generated particle, and the peak allocated memory. Here, NTSE counts the total number of evaluations of the target score required to generate one particle and therefore provides an implementation-independent measure of computational effort. Peak allocated memory is defined as the maximum number of bytes held by live tensors at any point during a training or sampling step. All measurements were obtained using a single NVIDIA Tesla V100-SXM2 GPU with \(16\) GB of memory.

\begin{table}[htbp]
\centering
\small
\setlength{\tabcolsep}{4pt}
\caption{Computational costs on the \texttt{MoG40} target. Runtime denotes the end-to-end wall-clock time, and NTSE denotes the number of target score evaluations per generated particle.}
\label{tab:exp:comp:cost}
\begin{tabular}{lccccc}
\toprule
Cost metric & ULA & MALA & pULA & HMC & PT \\
\midrule
Runtime (s)  $\downarrow$
& 107.6 & 346.1 & 107.6 & 20.9 & 21.2 \\
NTSE per particle  $\downarrow$
& \(50{,}000\) & \(100{,}000\) & \(50{,}000\)
& \(10{,}100\) & \(10{,}100\) \\
Peak memory (MB)  $\downarrow$
& 30.2 & 30.5 & 30.3 & 30.4 & 30.9 \\
\midrule
Cost metric  
& RDMC & SLIPS & SSI w/o precond. & SSI & Stab.\ SSI \\
\midrule
Runtime (s)  $\downarrow$
& 1771.7 & 246.1 & 1823.6 & 1756.6 & 3810.5 \\
NTSE per particle  $\downarrow$
& \(480{,}000\) & \(605{,}632\) & \(480{,}000\)
& \(480{,}000\) & \(561{,}000\) \\
Peak memory (MB)  $\downarrow$
& 2382.2 & 371.4 & 2382.0 & 2382.1 & 11098.4 \\
\bottomrule
\end{tabular}
\end{table}

The computational costs of conventional MCMC baselines is significantly smaller than those of RDMC, SLIPS and SSI; however, be interpreted jointly with the sampling-quality results, since the least expensive methods do not necessarily recover the multimodal target distribution accurately.

Among RDMC, SLIPS, and the SSI variants, SLIPS has the lowest runtime and memory footprint. It requires \(246.1\) seconds and \(371.4\) MB of peak memory, despite using the largest NTSE among the non-stabilized methods. In particular, SLIPS performs approximately \(26\%\) more target score evaluations than SSI, because SLIPS uses MALA for score estimation. SLIPS is about \(7.1\) times faster and uses about \(6.4\) times less memory. This discrepancy is because SLIPS does not reply on the importance sampling for initialization of Langevin Monte Carlo for score estimation, while both RDMC and SSI use.

Preconditioning introduces essentially no additional computational cost within SSI. Preconditioned and unpreconditioned SSI use the same NTSE, \(481{,}000\), and have nearly identical peak memory requirements of approximately \(2.38\) GB. Therefore, the substantial improvements in sample quality obtained through preconditioning are not accompanied by an increase in either the number of target score evaluations or memory consumption. This makes preconditioned SSI preferable to its unpreconditioned counterpart from both statistical and computational perspectives.

The computational profile of SSI is also nearly identical to that of RDMC. Relative to RDMC, SSI uses the same target score evaluations, and has essentially the same peak memory footprint. Consequently, the improvements of SSI over RDMC in MMD and \(\mathcal{W}_{2}\), reported in Table~\ref{tab:exp:slips:rdmc}, are achieved at approximately the same computational cost. In contrast, although SLIPS is substantially more efficient in terms of runtime and memory, SSI provides considerably better MMD and \(\mathcal{W}_{2}\) values. The comparison therefore reveals a clear trade-off between computational efficiency and sampling accuracy.

Stabilized SSI incurs the largest computational overhead. Relative to standard SSI, it increases the NTSE by approximately \(16.6\%\), from \(481{,}000\) to \(561{,}000\), while increasing the runtime by a factor of approximately \(2.2\). Its peak memory usage rises from \(2.38\) GB to approximately \(11.1\) GB, corresponding to a factor of approximately \(4.7\). Although this memory requirement remains within the \(16\)-GB capacity of the GPU used in our experiments, it substantially restricts the available memory for larger particle batches or higher-dimensional problems. Thus, stabilization provides an additional robustness mechanism at a non-negligible computational cost, whereas preconditioning yields improved sampling accuracy with essentially no additional overhead.

\subsection{Experimental setting of comparisons with DEO-PT}
\label{appendix:exp:pt}

For DEO-PT, we use 16 replicas and choose
$\mathcal{N}\!\left((0,0),8.0^{2}I_{2}\right)$ as the reference distribution. The hyperparameters used for preconditioned SSI are reported in Table~\ref{tab:parameters:ssi:pt}. We employ the vanilla velocity estimator for $t \leq 0.8$ and switch to the stable velocity estimator for $t>0.8$.

\begin{table}[htbp]
\centering
\caption{Hyperparameters of preconditioned SSI for sampling from~\eqref{eq:gmm:unequal}.}
\label{tab:parameters:ssi:pt}
\begin{tabular}{cccccccc}
\toprule
\multicolumn{3}{c}{PF ODE}
&
\multicolumn{2}{c}{Initialization}
&
\multicolumn{3}{c}{Score estimation}
\\
\cmidrule(lr){1-3}
\cmidrule(lr){4-5}
\cmidrule(lr){6-8}
$T_{0}$ & $T_{\mathrm{end}}$ & $M$
&
$\tau$ & $L$
&
$\eta$ & $K$ & $n$
\\
\midrule
0.1 & 0.98 & 200
&
0.1 & 1000
&
0.01 & 100 & 800
\\
\bottomrule
\end{tabular}
\end{table}

\section{Neural Sampler and Additional Experimental Results}

\subsection{Related work}

\par Unlike our method that estimates velocity fields on-the-fly, a distinct line of research parameterizes these fields using neural networks, leveraging their high expressive power. These approaches, often termed neural samplers~\cite{wu2020Stochastic,Arbel2021Annealed,wang2026energy}, typically minimize a divergence measure between the parameterized and target distributions.

\par A natural choice for the training objective is the reverse Kullback-Leibler (KL) divergence, primarily due to its computational tractability. However, training via reverse KL can be unstable when the target distribution is multimodal with high-energy barriers, often resulting in mode collapse. To mitigate this, several alternative objectives have been proposed. For instance, Flow Annealed Importance Sampling Bootstrap (FAB)~\cite{midgley2023flow} minimizes the $\alpha$-2 divergence.~\cite{mate2023learning} and~\cite{chemseddine2025neural} utilize loss functions derived from the residuals of the continuity equation, while~\cite{wu2025annealing} employ an optimal transport formulation.~\cite{he2025training} proposed minimizing the reverse KL divergence along diffusion trajectories to improve mode coverage. In parallel, recent works have focused on constructing neural samplers through the lens of diffusion models and stochastic optimal control~\cite{berner2024an,zhang2022path,Phillips2024Particle,liu2025adjoint,Havens2025Adjoint,guo2025proximal}.

\subsection{Additional experimental results}
\par In this subsection, we compare our method with a neural sampler~\cite{mate2023learning}. Specifically, the method leverages the continuity equation:
\[
    \partial_t f_t - \underbrace{\mathbb{E}_{p_{X_t}}[\partial_t f_t]}_{A_t} = -\langle \nabla f_t, u(t,x) \rangle + \nabla \cdot u(t,x),
\]
where the density is defined as $p_{X_t}(x) = e^{-f_t(x)} / Z_t$ for $Z_t\coloneqq \int  e^{-f_t(x)} \d x $. 
To solve this, the potential $f_t$ is parameterized as an interpolation $f_t = t f_1 + (1-t) f_0 + t(1-t) \phi_t$, where $\phi_t$ is approximated by a neural network $\phi_t^{\theta}$. The optimization objective is formulated as:
\[
\mathcal{L}(\theta) = \mathbb{E}_{t \sim \mathcal{U}[0,1], z \sim p_0} \Big[ \big| \partial_t f^{\theta_1}_t - A_t^{\theta_2} + \langle \nabla f_t^{\theta_1}, u^{\theta_3}(t,x) \rangle - \nabla \cdot u^{\theta_3}(t,x) \big| \Big], 
\]
where $\theta=(\theta_1, \theta_2, \theta_3)$ represents the trainable parameters. Since the approximation of high-dimensional integrals over the intermediate density $p_{X_t}$ is computationally intractable,~\cite{mate2023learning} proposed an iterative optimization strategy. Let $\theta'$ denote the parameters obtained from the previous minimization step. We solve the flow ODE $\partial_t \psi^{\theta'}_t = u^{\theta'}(t, \psi^{\theta'}_t)$ with the initial condition $\psi^{\theta'}_0 = \text{Id}$ using these fixed parameters $\theta'$. The loss is then evaluated on points generated by sampling $z\sim p_0$ and then applying $\psi_t^{\theta'}$. Figure~\ref{fig:example2d:neural} illustrates the experimental results on the \texttt{rings}, \texttt{MoG7x7}, and \texttt{Mog40} datasets.

\begin{figure}[htbp]
\centering
\includegraphics[width=0.7\linewidth]{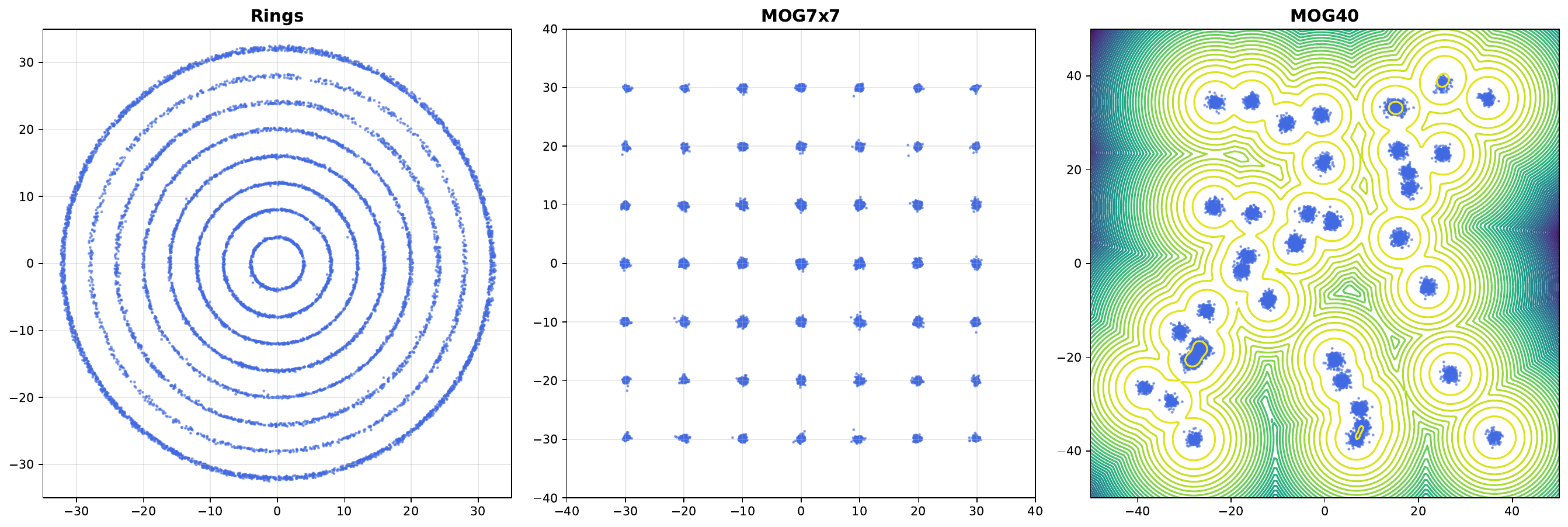}
\caption{Results for the neural sampling.}
\label{fig:example2d:neural}
\end{figure}

\par For these experiments, we utilized the training configuration with consistent hyper-parameters across all three datasets. The training process used a batch size of 4096, a learning rate of $1.0 \times 10^{-3}$, and a decay rate of 0.98. The total training duration was set to $N_{\text{train}} = 30,000$ steps, utilizing 51 inner steps for the flow integration. We initialized the noise level at $\sigma = 20.0$ and set the generalized Gaussian exponent to $p=2.0$ (corresponding to a standard Gaussian).